\documentclass[12pt, reqno]{amsart}
\usepackage[utf8x]{inputenc}
\usepackage[margin=1in,footskip=0.25in]{geometry}
\usepackage{amsthm,amssymb, amsmath}
\usepackage{subcaption}
\usepackage{graphicx}
\usepackage[percent]{overpic}
\usepackage[outline]{contour}
\usepackage{tikz-cd}
\usepackage{float}
\usepackage{hyperref}
\usetikzlibrary{calc}
\usepackage[nameinlink]{cleveref}
\newtheorem{theorem}{Theorem}
\newtheorem{corollary}[theorem]{Corollary}
\newtheorem{conjecture}[theorem]{Conjecture}
\newtheorem{lemma}[theorem]{Lemma}
\newtheorem{proposition}[theorem]{Proposition}

\theoremstyle{definition}

\newtheorem{question}[theorem]{Question}
\newtheorem{algorithm}[theorem]{Algorithm}
\newtheorem{example}[theorem]{Example}
\newtheorem{definition}[theorem]{Definition}
\newtheorem{remark}[theorem]{Remark}

\numberwithin{theorem}{section}

\newcommand*{\D}{\mathbb{D}}

\newcommand*{\e}{\mathrm{e}}
\newcommand*{\pt}{\tilde{p}}
\newcommand*{\gt}{\tilde{g}}

\makeatletter
\newcommand{\spx}[1]{%
\if\relax\detokenize{#1}\relax
    \expandafter\@gobble
\else
    \expandafter\@firstofone
\fi
{^{#1}}%
}
\makeatother

\newcommand{\genericdel}[4]{%
    \ifcase#3\relax
        \ifx#1.\else#1\fi#4\ifx#2.\else#2\fi\or
        \bigl#1#4\bigr#2\or
        \Bigl#1#4\Bigr#2\or
        \biggl#1#4\biggr#2\or
        \Biggl#1#4\Biggr#2\else
        \left#1#4\right#2\fi
}
\newcommand{\del}[2][-1]{\genericdel(){#1}{#2}}
\newcommand{\set}[2][-1]{\genericdel\{\}{#1}{#2}}

\newcommand{\sbr}[2][-1]{\genericdel[]{#1}{#2}}
\newcommand{\abr}[2][-1]{\genericdel<>{#1}{#2}}

\let\intoo\del

    \newcommand{\intco}[2][-1]{\genericdel[){#1}{#2}}

\newcommand{\sVert}[1][0]{%
    \ifcase#1\relax
        \rvert\or\bigr|\or\Bigr|\or\biggr|\or\Biggr
    \fi
}

\DeclareMathOperator{\ord}{ord}
\DeclareMathOperator{\cyc}{cyc}
\DeclareMathOperator{\Hyp}{Hyp}
\DeclareMathOperator{\Sat}{Sat}
\DeclareMathOperator{\Prim}{Prim}
\DeclareMathOperator{\MC}{Cyc}

\let\face\sbr

\newcommand{\C}{\mathbb{C}}
\newcommand{\Q}{\mathbb{Q}}
\newcommand{\Z}{\mathbb{Z}}
\newcommand{\R}{\mathbb{R}}
\newcommand{\N}{\mathbb{N}}
\newcommand{\M}{\mathcal{M}}
\newcommand*{\sm}{\setminus}
\newcommand{\ol}{\overline}
\newcommand*{\vp}{\varphi}
\newcommand*{\ra}{\rightarrow}
\newcommand*{\lra}{\longrightarrow}
\newcommand*{\Per}{\mathrm{Per}}
\newcommand*{\cal}{\mathcal}
\def\e{\epsilon}
\newcommand{\itin}{\operatorname{itin}}

\renewcommand*{\O}{\mathcal{O}}

\title{A cell decomposition for marked cycle curves}

\author{Caroline Davis}
\address{Indiana University}
\email{cda1@iu.edu}
\author{Malavika Mukundan}
\address{Boston University}
\email{mmukunda@bu.edu}
\author{Danny Stoll}
\address{University of Michigan}
\email{dastoll@umich.edu}
\author{Giulio Tiozzo}
\address{University of Toronto}
\email{tiozzo@math.utoronto.ca}

\begin{document}
\maketitle

\begin{abstract}
    We describe a family $\MC_p(\mathcal{F})$ of \emph{marked cycle curves} that parameterize the cycles of period $p$ of a given family $\mathcal{F}$ of dynamical systems.
    We produce algorithms to compute a canonical cell decomposition for the marked cycle curves over the family $\Per_1(0)$ of quadratic polynomials as well as over the family  $\Per_2(0)$ of quadratic rational maps with a critical 2-cycle.
    We obtain formulas for the number of $d$-cells in these decompositions, giving rise to e.g. a formula for their genus.
\end{abstract}

\section{Introduction}

A \emph{cycle} of \emph{period} $p$ for a rational map $f : \widehat{\C} \to \widehat{\C}$ is a set $\xi = (z_1, \dots, z_p)$ of $p \ge 1$ distinct points in $\widehat{\C}$ such that $f(z_i) = z_{i+1}$ for $i = 1, \dots, p$, where indices are interpreted mod $p$. Each such cycle lives in the \emph{configuration space} $\textup{Conf}_p$ 
of sets of $p$ distinct points in $\widehat{\C}$.

\begin{definition}
Let $U \subseteq \C$ be a Zariski dense open set, and $\cal{F} = \{ f_u \}_{u \in U}$ be an algebraic family of rational maps. For each $p \geq 1$, 
we define the \emph{affine marked cycle curve} of period $p$ as 
\[
    \MC^+_p(\mathcal{F}) := \set{ (u, \xi) \in U \times \textup{Conf}_p \ : \  \xi \textup{ is a cycle of period $p$ for }f_u }
\]
and let the \emph{marked cycle curve} $\MC_p(\mathcal{F})$ be the normalization of $\MC^+_p(\mathcal{F})$.
\end{definition}

The goal of this paper is to study the curves $\MC_p(\mathcal{F})$ in the following two cases: 

\begin{enumerate}

\item The family $\mathcal{F}_1 = \Per_1(0)$ of quadratic polynomials

\item The family $\mathcal{F}_2 = \Per_2(0)$ of rational maps with a critical cycle of period $2$

\end{enumerate}

In both cases, each $\MC_p(\mathcal{F})$ is a smooth projective curve, hence a Riemann surface.

Our main result concerns algorithmically constructing a cell decomposition for $\MC_p(\mathcal{F})$:

\begin{theorem}\label{main-thm}
    For each $p \geq 1$ and $m = 1,2$, the Riemann surface $\MC_p(\mathcal{F}_m)$ has an explicit cell decomposition $\Sigma_{p, m}$ given by gluing sides of a certain union of polygons,
    with edges labelled by primitive components of period $p$ in the Mandelbrot set, and vertices and faces labelled by canonical combinatorial data associated to each primitive component.
\end{theorem}

\begin{figure}[H]
    \includegraphics[width = 0.5 \textwidth]{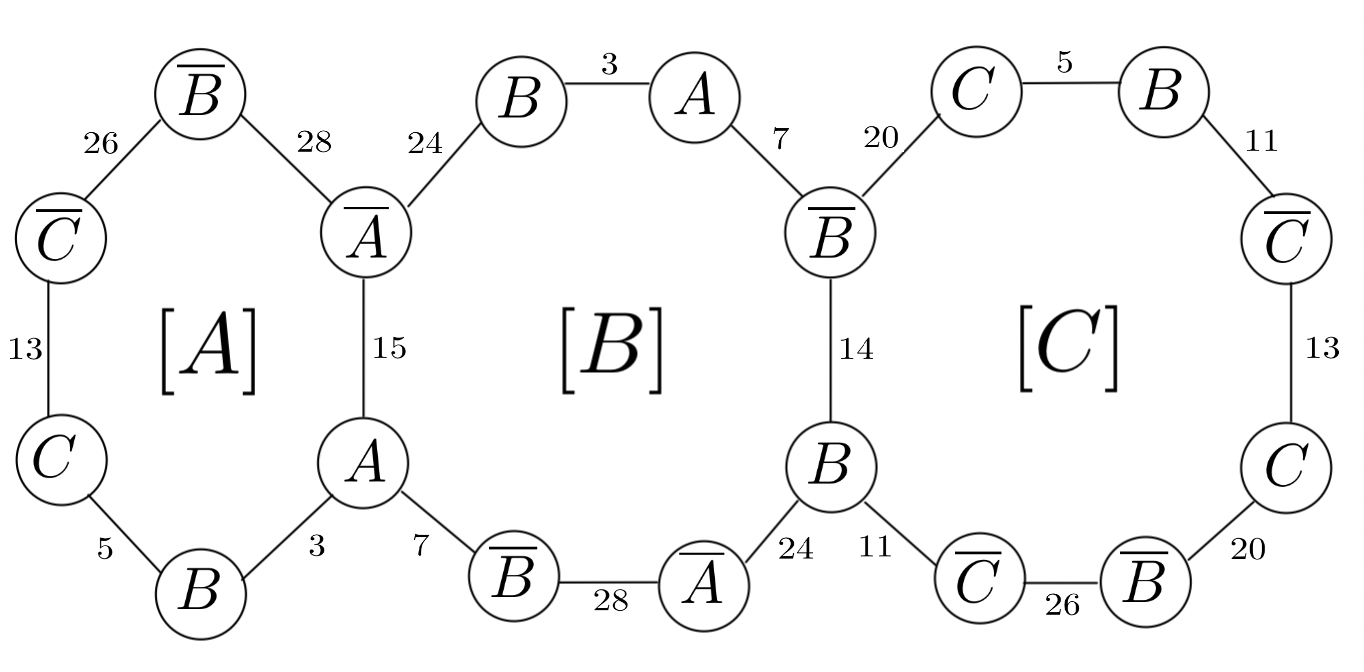}
    \caption{The cell decomposition for $\MC_5(\mathcal{F}_1)$, which has genus $2$.}
\end{figure}

We obtain two algorithms for $\MC_p(\mathcal{F}_1)$, described in \Cref{algo-per1-angle,algo-per1-knead}, and one algorithm for $\MC_p(\mathcal{F}_2)$, 
described in \Cref{algo-per2}.

\begin{figure}[t]
		\includegraphics[height = .9\textwidth, angle=90, origin=c]{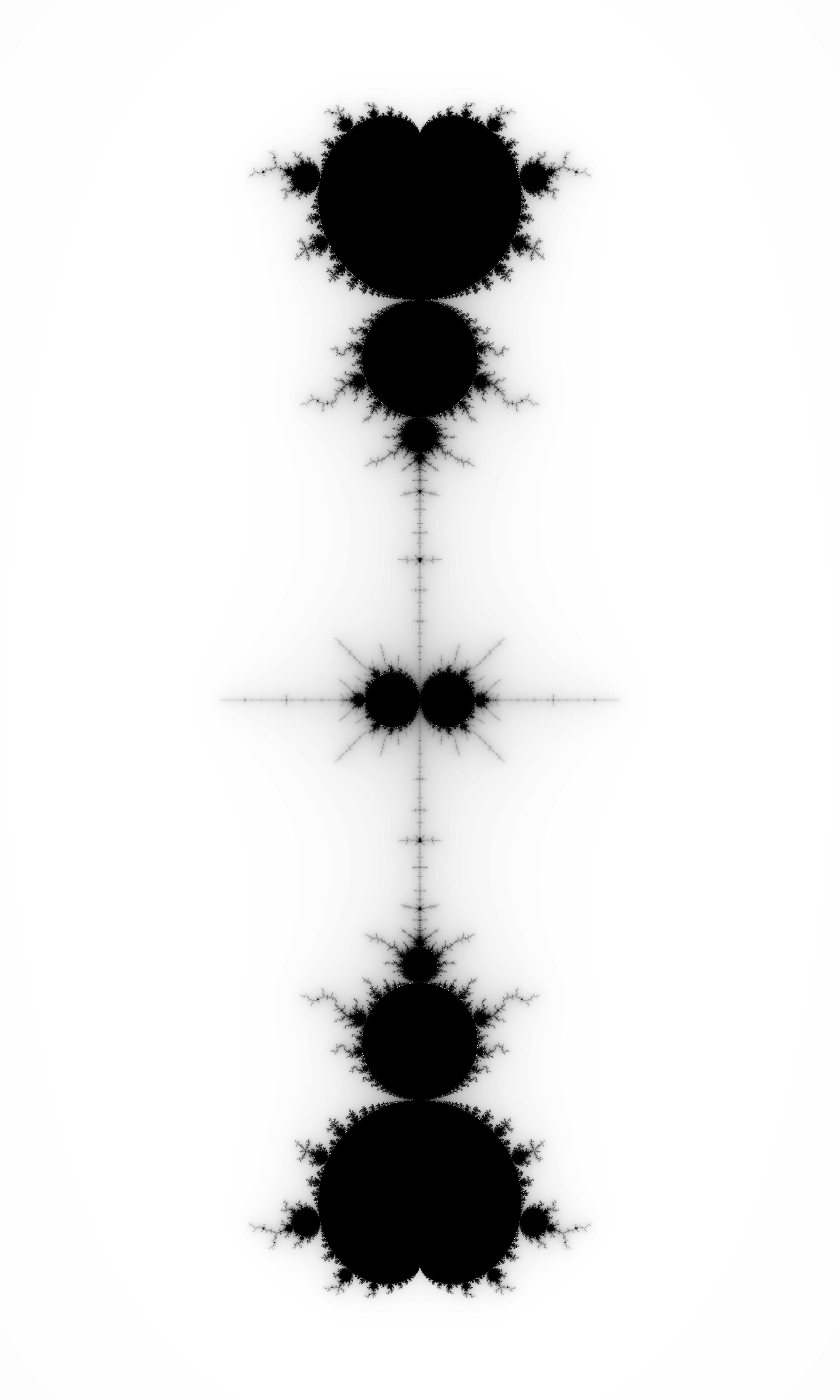}
		\vspace{-2.5cm}
		\caption{ The curve $\MC_3(\mathcal{F}_1)$. The two copies of the Mandelbrot set are joined at the cusp of the airplane component, which is
		the only primitive component of period $3$.} \label{fig:cyc-3}
\end{figure}
\subsection{Marked cycle curves for $\Per_1(0)$} 

To describe the first family of marked cycle curves in greater detail, let us recall that the family of quadratic polynomials is parameterized as 
$$f_c(z) := z^2 + c, \quad c \in \C,$$ 
hence we take $U = \C$. Each element of the family has a superattracting critical fixed point at $\infty$, so we can denote this family as $\Per_1(0)$. 

There are several ways to interpret and compute the curve $\MC_p(\mathcal{F}_1)$, let us see some. 
Recall that each cycle $(z_1, \dots, z_p)$ of period $p$ has an associated \emph{multiplier} $\lambda(c) := (f_c^{\circ p})'(z_1)$. 
Note that different points in the same cycle have the same multiplier. 

The \emph{multiplier function} $c \mapsto \lambda(c)$ plays a central role in holomorphic dynamics: in fact, it is well-known that the multiplier function is the Riemann map of the hyperbolic component for which the given cycle is attracting (\cite[Chapter 14]{Douady_exploringthe}). 
Thus, the maximal domain of extension of the multiplier is closely related to the geometry of hyperbolic components and of limbs in the Mandelbrot set (see e.g. \cite{Levin-multiplier}); moreover, critical points for the multiplier function equidistribute to the harmonic measure on  the Mandelbrot set
(\cite{FIRSOVA2021107591}, \cite{FIRSOVA_GORBOVICKIS_2023}).

Note that the multiplier function cannot be globally defined as a holomorphic function on $\mathbb{C}$; instead, the natural domain of 
the multiplier function for cycles of period $p$ is:

\begin{definition}
    The \emph{marked multiplier curve} of period $p$ is
    \[
        \textup{M}_p := \{ (c, \lambda) \in \C^2 \ : \ \exists \textup{ a cycle of period }p\textup{ and multiplier }\lambda \textup{ for }f_c\}.
    \]
\end{definition}

By considering the map that assigns to each period $p$ cycle its multiplier, we see that the curves $\MC^+_p(\mathcal{F}_1)$ and $\textup{M}_p$ are birationally equivalent, hence $\MC_p(\mathcal{F}_1)$ can also be defined as the normalization of $\textup{M}_p$. 

\begin{figure}
\includegraphics[width= .8\textwidth]{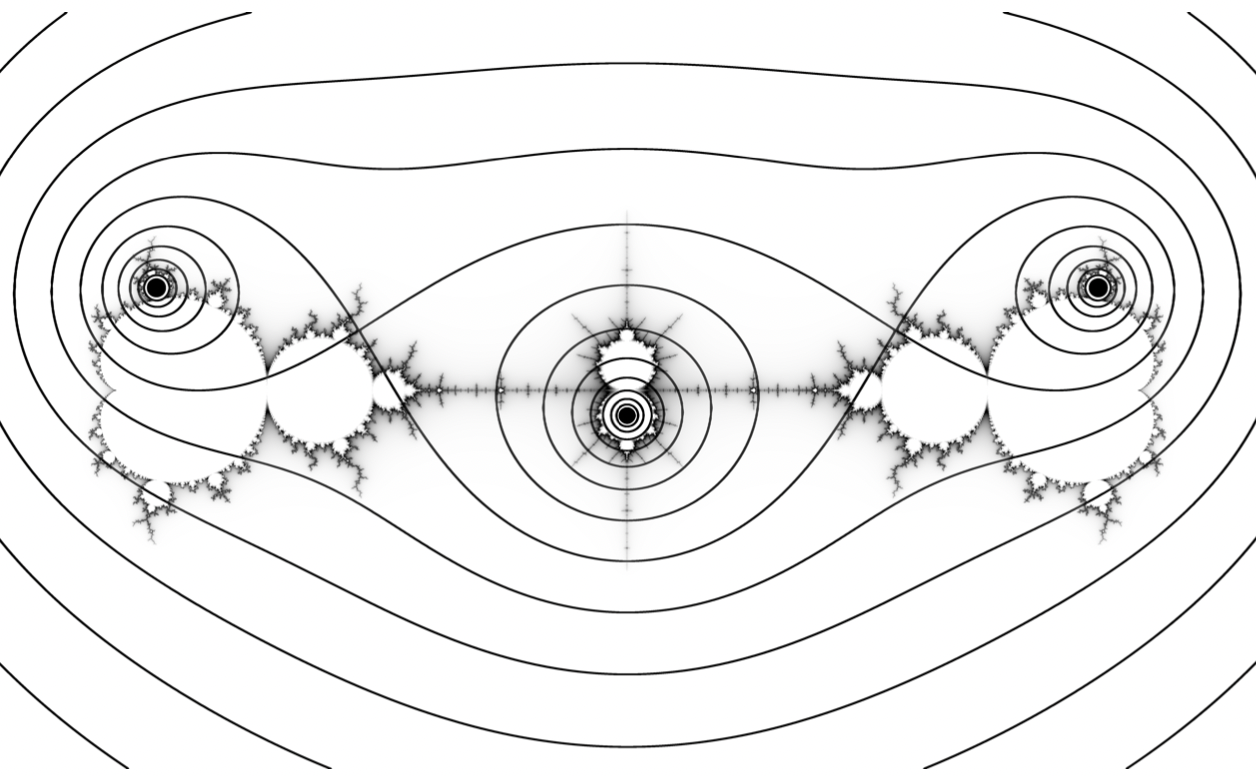}
\caption{The curve $\MC_3(\mathcal{F}_1)$ is a blow-up of $\textup{M}_3$ and is the natural domain of the multiplier function, whose contour lines are drawn here. Note that, as expected, the locus where the multiplier has modulus less than $1$ is the union of three hyperbolic components of period $3$.}
\end{figure}
\label{fig:multiplier}

In order to compute $\textup{M}_p$, let us also introduce the \emph{dynatomic polynomial} of period $p$ as
\begin{align*}
    \textup{Dyn}_p(c, z) & := \prod_{d | p} (f_c^{\circ d}(z) - z)^{\mu(d)}.
\end{align*}
The \emph{dynatomic curve} of period $p$ is then
\begin{align*}
    \textup{D}_p & :=  
                  \{ (c, z)   \in \mathbb{C}^2 \ : \ \textup{Dyn}_p(c, z) = 0 \}
\end{align*}
and can be thought of as the set of pairs $(c, z)$ where $z$ is a periodic point for $f_c(z)$ of period $p$. An affine model for the marked multiplier curve $\textup{M}_p$  is given as the zero set of the resultant
\[
    R(c, \lambda) := \textup{Res}(\textup{Dyn}_p(c, z), (f_c^p)'(z) - \lambda)
\]
hence, the curve $\MC_p(\mathcal{F}_1)$ can be obtained from the zero set of $R(c, \lambda)$ by successive blowups.
See \Cref{S:examples} for some explicit computations.

\subsection{The cell structure of marked cycle curves for $\Per_1(0)$}

Let us consider the branched cover $\pi: \MC_p(\mathcal{F}_1) \to \cal{F}_1=\widehat{\mathbb{C}}$.

Let $R_p \subseteq \mathbb{C}$ be the set of roots of primitive hyperbolic components of period $p$
in the Mandelbrot set $\mathcal{M}$.
For each root $c \in R_p$, let $V(c)$ be the \emph{vein} joining $0$ and $c$ in $\mathcal{M}$. Note that each vein $V(c)$ is a path in $\C$, c.f.\ \cite{veins}.
Let
\[V_p := \bigcup_{c \in R_p} V(c)\]
be the union of the veins joining $0 \in \mathbb{C}$ with the roots of the primitive components.
Topologically, $V_p$ is homeomorphic to a graph.

The Riemann sphere $\widehat{\mathbb{C}}$ has the following cell complex decomposition:
\begin{enumerate}
    \item one vertex, $0 \in \mathbb{C}$;
    \item one edge for each vein $V(c)$;
    \item one face, the complementary region
          \[U_\infty := \widehat{\mathbb{C}} \setminus V_p,\]
          which is homeomorphic to a disk.
\end{enumerate}

We define the cell complex structure of $\MC_p(\mathcal{F}_1)$ by lifting via $\pi$ the above cell complex structure on the sphere.

For each vein $V(c)$, the set $\pi^{-1}(V(c))$ has several connected components. Let us call \emph{branched lift} a connected component that contains a branched point for $\pi$.

\begin{definition}\label{D:Sigma_p}
    Let $p \ge 1$ be a period, and let $\pi : \MC_p(\mathcal{F}_1) \to \widehat{\mathbb{C}}$ be the projection onto the $c$ coordinate. The cell complex structure $\Sigma_{p, 1}$ on $\MC_p(\mathcal{F}_1)$ is defined as follows:
    \begin{enumerate}
        \item the vertices are the set $\pi^{-1}(0)$;
        \item there is one edge for each branched lift of a vein $V(c)$ with $c \in R_p$;
        \item faces are closures of the  connected components of $\pi^{-1}(U_\infty)$.
    \end{enumerate}
\end{definition}

After some preliminaries in \Cref{S:setup}, we will provide in \Cref{S:cell-p1} two algorithms to construct explicitly the cell complex $\Sigma_{p, 1}$.

\subsection{Marked cycle curves for $\Per_2(0)$} 

As one can interpret the moduli space of quadratic polynomials as $\mathcal{F}_1 = \Per_1(0)$, one can analogously define the marked cycle curves $\MC_p(\mathcal{F}_2)$ over $\mathcal{F}_2 = \Per_2(0)$, the moduli space of quadratic rational maps which have a superattracting $2$-cycle.

We consider the parameterization 
$$g_a(z) := \frac{z^2 + a}{1 - z^2}$$
with $a \in \C \setminus \{-1\}$. Similarly as for the polynomial family, we have a projection map $\pi_2 : \MC_p(\mathcal{F}_2) \to \mathcal{F}_2$.

In \Cref{sec:per2}, we generalize the algorithm underlying \Cref{main-thm} to construct an explicit cell decomposition $\Sigma_{p, 2}$ for $\MC_p(\mathcal{F}_2)$, and indicate how $\Sigma_{p,2}$ is related to $\Sigma_{p, 1}$.

In the case of $\Per_2(0)$, the same construction as above can be carried: 
\begin{enumerate}
\item
vertices correspond to lifts of the ``central" parameter $g_0(z) = \frac{z^2}{1-z^2}$, which is conjugate to the basilica quadratic polynomial; 
\item
similarly to veins for the Mandelbrot set, veins in $\Per_2(0)$ from the central component to roots of primitive components of period $p$ can be defined; 
 in the cell decomposition, there is one edge for each branched lift of such a vein;
\item
faces correspond to connected components of $\pi_2^{-1}(\mathcal{E})$, 
where $\mathcal{E}$ is the unbounded hyperbolic component in $\Per_2(0)$, where the Julia set is a quasicircle. 
\end{enumerate}

Finally, in \Cref{sec:combo}, we detail various combinatorial properties of $\Sigma_{p, 1}$ and $\Sigma_{p,2}$ which have interpretations for the structure of the Mandelbrot set.
In particular, we derive an explicit formula for the number of cells of each dimension in $\Sigma_{p, 1}$ and $\Sigma_{p,2}$, giving rise to a formula for the genus.

\subsection{Examples of rational curves} \label{S:examples}

Before discussing the general algorithms, let us look at some low period examples. As we will see in \Cref{thm:genus}, the genus of $\MC_p(\mathcal{F}_1)$ equals $0$ exactly for $2 \leq p \leq 4$;
in these cases, we can construct explicitly rational models for these curves. 

For $p = 2$, since each quadratic polynomial has one cycle of period $2$, the covering map $\pi$ has degree $1$, so $\MC_2(\mathcal{F}_1)$ is just a copy of $\mathcal{F}_1 = \widehat{\mathbb{C}}$. 

\subsection*{Period 3}
Let us denote as $\lambda$ the multiplier of the marked cycle.
By computing the resultant of the polynomials $\frac{f^{\circ 3}_c(z) - z}{f_c(z) -z}$  and $(f^{\circ 3}_c)'(z) - \lambda$ with respect to $z$, the marked multiplier curve of period $3$ is given by
\[\textup{M}_3 = \{ (c, \lambda)  \in \C^2 \ : \ 64 -16 \lambda + \lambda^2 + 64 c - 8 \lambda c + 128 c^2 + 64 c^3 = 0 \}.\]
The singular point on this curve is $(0, 8)$. Hence, by setting $a = \lambda - 8$, $s =  4 c$, we get
\[\textup{M}_3' = \{ (a, s) \in \C^2 \ :  \ a^2 - 2 a s + 8 s^2 + s^3 = 0 \}\]
whose singular point is now at $(0,0)$; now, we perform a blowup by setting $b = \frac{a}{s}$
which yields
\[\textup{M}''_3 = \{ (b, s) \in \C^2 \ : \ 8 - 2 b + b^2 + s = 0 \}\]
hence, this curve is rational, parameterized by
\[c = \frac{s}{4} = - \frac{8 - 2 b + b^2}{4}\]
hence a model for the marked cycle curve $\MC_3(\mathcal{F}_1)$ is the projective closure of
\[ \MC_3^*(\mathcal{F}_1) := \left\{ (c, b) \in \C^2 \ : \ c =  - \frac{8 - 2 b + b^2}{4} \right\}\]
where the map $\MC_3^*(\mathcal{F}_1) \to \textup{M}_3$ is given by $(c, b) \mapsto (c, 4 b c  + 8)$.

\subsection*{Period 4}
For period $4$, we have by computing the appropriate resultant
\begin{align*}
    \textup{M}_4 & = \{ (c, \lambda) \in \C^2 \ : \ -4096 + 768 \lambda - 48 \lambda^2 + \lambda^3 - 8192 c^2 + 256 \lambda c^2 + 16 \lambda^2 c^2 + \\
        & -12288 c^3 - 256 \lambda c^3 - 12288 c^4 - 256 \lambda c^4 - 12288 c^5 - 4096 c^6= 0\}
\end{align*}
and setting $s = 4 c, a = \lambda - 16$ and then $b = \frac{a}{s}$
\[\textup{M}'_4 = \{ (b, s) \in \C^2 \ : \ 256 - 48 b - b^3 + 64 s + 4 b s - b^2 s + 12 s^2 + b s^2 + s^3 = 0 \}\]
which has the singular point $(b,s) = (4, -8)$ (which corresponds to the Chebyshev polynomial), so we set $v = s+8$. Then we do a blowup, setting $r = \frac{b-4}{v}$,
obtaining
\[\textup{M}''_4 = \{ (r, v) \in \C^2 \ : \ 8 + 20 r + 4 r^2 - v - r v + r^2 v + r^3 v = 0 \}\]
which is a rational curve, given by
\[v = \frac{-4 (2 + 5 r + r^2)}{(-1 + r) \del{1 + r}^2}\]
and, since $v = 4 c + 8$ we get
\[c = \frac{- r(3  + 3 r +2 r^2)}{(-1 + r) \del{1 + r}^2}\]
and, if we set $u = \frac{r-1}{r+1}$ sending $-1$ to $\infty$ and $+1$ to $0$,
\[c = -\frac{4 + 3 u + u^3}{4 u}\]
so $\MC_4(\mathcal{F}_1)$ is isomorphic to the projective closure of
\[\MC_4^*(\mathcal{F}_1) := \left\{ (c, u) \in \C^2 \ : \ c = -\frac{4 + 3 u + u^3}{4 u} \right\}.\]

\begin{center}
    \begin{figure}[t]
        \includegraphics[height = 0.9 \textwidth, angle=90, origin=c]{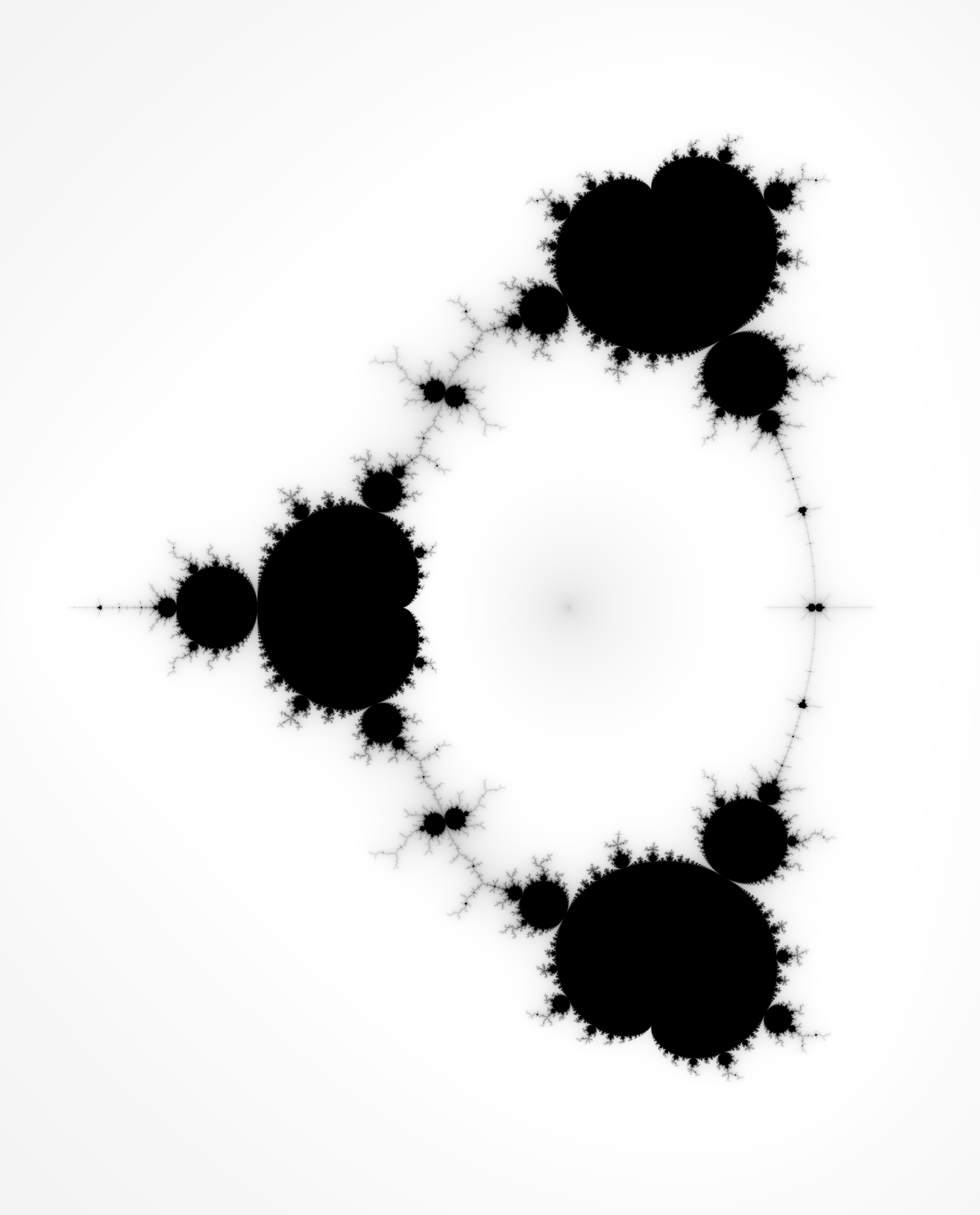}
    \vspace{-1.5 cm}    \caption{The curve $\MC_4(\mathcal{F}_1)$, in a rational parameterization.}
    \end{figure}
\end{center}

\subsection{History}

Dynatomic curves have been studied for several decades. The first proof that the dynatomic curves $\textup{Dyn}_p(c, z)$ are irreducible is due to \cite{Bousch-thesis}, and then to  \cite{LauSchleicher} using the ``monodromy action" on kneading sequences, in a way similar to our discussion here. They also compute the Galois group of dynatomic curves. The genus has been computed by \cite{Morton}. A more recent proof is given in \cite{BuffTanLei}.
For polynomials, a ``monodromy graph" closely related to our cell structure is produced in \cite{ReductionDynatomic}.

\subsection{Acknowledgments}

The authors wish to thank the Mathematical Sciences Research Institute (now SLMath) in Berkeley, where this project started during the semester on 
\emph{Complex Dynamics} in Spring 2022, NSF grant number 1440140. We especially thank Sarah Koch for raising some of this paper's questions and Eriko Hironaka for helpful conversations. 

Fun fact: our \Cref{algo-per1-knead} is nicknamed ``the bar method" as a tribute to the eponymous barre studio in Berkeley frequented by MSRI semester organizer Jasmin Raissy. 

Many figures in this paper were generated in dynamo, a program written by the third author and accessible at \href{https://github.com/dannystoll1/dynamo}{https://github.com/dannystoll1/dynamo}. 
\section{Setup} \label{S:setup}

\subsection{Cycles}
Let $f : X \to X$ be a continuous map. A point $x \in X$ has \emph{(exact) period} $p$ if $f^{\circ p}(x) = x$ and $f^{\circ k}(x) \ne x$ for $1 \leq k < p$.
Let $\textup{Per}(f, p)$ be the set of points of period $p$ for $f$.
A \emph{periodic cycle} of period $p$ is a set of cardinality $p$ on which $f$ acts bijectively and transitively.
Let $\textup{Cyc}(f, p)$ be the set of periodic cycles of period $p$ for $f$.

Let $\sigma : \set{0, 1}^\mathbb{N} \to \set{0, 1}^\mathbb{N}$ be the shift map.

\begin{definition}
    We define the set of \emph{abstract cycles of period $p$}
    as the set, which we denote by $C_p$, of periodic cycles of period $p$ for the shift map $\sigma$.
\end{definition}

Abstract periodic cycles of period $p$ can be labeled by binary sequences of period $p$, up to cyclic
permutation.
For instance, for $p = 4$ there are exactly $3$ abstract cycles:
\[0001, \qquad 0011, \qquad 0111,\]
where with a finite sequence, e.g. $0001$, we mean the periodic sequence $\del{0001}^\infty$.
Given a cycle $x$, we denote its \emph{conjugate} $\overline{x}$ to be the cycle obtained
from the binary expansion of $x$ by swapping the digit $1$ with $0$ and vice versa.

E.g.: for $p = 5$, there are $6$ abstract cycles:
\[
    \begin{array}{ll}
        A:            & 00001 \\
        \overline{A}: & 01111 \\
        B:            & 00011 \\
        \overline{B}: & 00111 \\
        C:            & 00101 \\
        \overline{C}: & 01101
    \end{array}
\]

Abstract periodic cycles appear in several branches of mathematics: for instance, the number of abstract cycles of period $p$, themselves examples of Lyndon words, equals the number of irreducible polynomials of degree $p$
over the finite field $\mathbb{F}_2$, see for example the necklace polynomial introduced in 1872 \cite{moreau}.

\subsection{Cycle duos}

\begin{definition}\label{def:cycle_duo}
    Let us define an equivalence relation on $C_p$ by setting $x \sim y$ if either $x = y$ or $x = \overline{y}$.
    Then we define the set of \emph{abstract cycle duos of period $p$} as
    \[ [C_p] := C_p/\sim,\]
    and we denote as $[\alpha]$ the class of the cycle $\alpha$.
\end{definition}

A cycle duo is \emph{reflexive}, or \emph{unramified}, if its equivalence class contains one element,
and \emph{ramified} if its equivalence class contains two elements.

\medskip
\textbf{Example.} For $p = 4$, there are $2$ cycle duos:
\[[A] =  \{ 0001, 0111 \} \qquad [B] = \{ 0011 \}\]
where $[B]$ is reflexive, while $[A]$ is not.
For $p = 5$, there are $3$ cycle duos:
\[[A] =  \{ 00001, 01111 \} \qquad [B] =  \{ 00011, 00111 \} \qquad [C] =  \{ 00101, 01101 \}\]
and all of them are ramified.
Note that if $p$ is odd, all cycles are ramified, hence
$\#[C_p] = \frac{\# C_p}{2}.$

\subsection{Hyperbolic components and external rays}
The exposition in this subsection is a summary of some material found e.g.\ in \cite{Milnor-External} and \cite{Douady_exploringthe}.

Let us consider the family of quadratic polynomials $f_c(z) := z^2 + c$, $c \in \mathbb{C}$.
The \emph{Mandelbrot set} $\mathcal{M}$ is the set of parameters $c$
for which the orbit of the critical point $z = 0$ is bounded under iteration by $f_c$.
A \emph{hyperbolic component} of period $p$ in the Mandelbrot set is a connected component in the set of parameters for which the critical point converges to a periodic cycle of period $p$.

Several essential results about \emph{external rays} are needed in order to eventually track marked cycles.
The Mandelbrot set $\M$ is connected, and its complement has a Riemann map $\Phi_M : \widehat{\mathbb{C}} \setminus \overline{\mathbb{D}} \to \widehat{\mathbb{C}} \setminus \mathcal{M}$, see \cite{Douady_exploringthe}.
For each $\theta \in \mathbb{R} / \mathbb{Z}$, the \emph{parameter external ray} of angle $\theta$ is the set $R_M(\theta) := \{ \Phi_M( r e^{2 \pi i \theta}), \ r > 1\}$, that is,
the image of a radial line in the complement of $\mathbb{D}$ under the Riemann map.

Similarly, for each $c \in \C$, the \emph{filled Julia set} $K_c$ is the set of $z \in \C$ for which the orbit $(f_c^{\circ n}(z))_{n \geq 0}$ is bounded,
and the \emph{Julia set} is $J_c := \partial K_c$.
If $K_c$ is connected, there is a Riemann map $\Phi_c : \widehat{\C} \setminus \overline{\D} \to \widehat{\C} \setminus K_c$, and similarly we define the \emph{dynamical external ray} at
angle $\theta$ as $R_c(\theta) := \{ \Phi_c( r e^{2 \pi i \theta}), \ r > 1\}$.
If $K_c$ is disconnected, the Riemann map $\Phi_c^{-1}$ may not extend to all of $\hat{\C} \setminus K_c$, but taking absolute values allows us to proceed, defining the Green's function $G_c : \hat{\C} \setminus K_c \to \mathbb{R}$:
\[G_c(z):= \lim\limits_{n\to\infty} \frac{1}{2^n} \log|f_c^{\circ n}(z)|.\]
Note that by construction $G_c(f_c(z)) = 2 G_c(z)$.
For each $t > 0$, we define the \emph{equipotential line} of potential $t$ as the set $\{z \in \mathbb{C} \ : \ G_c(z) = t \}$.
In fact, the inverse Riemann map $\Phi^{-1}_c$ is defined on the set $\{ z  \in \C \ : \  G_c(z) > G_c(0) \}$.
In parameter space, the Green's function for the Mandelbrot set is given by (see \cite{Douady_exploringthe})
\[G_{\M}(c) = G_c(c)\]
for $c \in \C \setminus \M$, and $G_{\M}(c) = 0$ for $c \in \M$.
Given $\theta \in \mathbb{R}/\mathbb{Z}$, the map $r \mapsto \Phi_c(r e^{2 \pi i \theta})$ is well-defined for $r > e^{G_c(0)}$,
and it extends to a maximal interval $(r_0, \infty)$.

Following \cite[Section 8]{Douady_exploringthe}, we give the definition:

\begin{definition}
    Given $\theta \in \mathbb{R}/\mathbb{Z}$, we say that the  ray $R_c(\theta)$ \emph{bifurcates} if the map
    $r \mapsto \Phi_c(r e^{2 \pi i \theta})$
    is only defined on $(r_0, \infty)$ with some $r_0 > 1$ and $\lim_{r \to r_0^+} \Phi_c(r e^{2 \pi i \theta})$ is a critical point of the Green's function.
    On the other hand, the dynamical ray $R_c(\theta)$ is said to $\emph{land}$ if $r_0 = 1$ and $\lim_{r \to 1^+} \Phi_c(r e^{2 \pi i \theta})$ exists.
\end{definition}

Similar definitions hold for parameter rays.

\begin{definition}
    Given two dynamical  rays $R_c(\alpha), R_c(\beta)$ that land at the same point $x$, the set $R_c(\alpha) \cup R_c(\beta) \cup \{x\}$
    disconnects $\C$ into two connected components.
    We define the \emph{wake} $\cal{W}(\alpha, \beta)$ as the connected component  that does not contain $0$.
    An analogous definition holds for parameter rays.
\end{definition}

The following few results are considered folklore but are explicitly formalized in the third author's thesis \cite{danny}, to which we refer for details.

Let $\tau : \mathbb{R}/\mathbb{Z} \to \mathbb{R}/\mathbb{Z}$ be the \emph{doubling map} $\tau(\theta) := 2 \theta \mod 1$.
It is useful to recall that the doubling map acts as a shift on binary sequences: if $\theta = .\theta_1 \theta_2 \dots$ is the binary expansion of $\theta$,
then the binary expansion of $\tau(\theta)$ is $\tau(\theta) = .\theta_2 \theta_3\dots$

When an angle $\theta$ is periodic of exact period $p$ for the doubling map, we say that the associated ray $R_M(\theta)$ is \emph{$p$-periodic}.
For any $p > 1$, the set of external rays of period $p$ is in 2:1 correspondence with the set of \emph{hyperbolic components} of period $p$.
That is, for any hyperbolic component $H$ of period $p > 1$, there are exactly two rays, called \emph{characteristic}, which land at the root of $H$.
Both characteristic rays have period $p$.

\begin{definition}
    A hyperbolic component of period $p$ with root $c$ is \emph{primitive} if
    the root of the critical value component in the Julia set has
    period exactly $p$.
\end{definition}

Equivalently, a hyperbolic component $H$ is primitive when the angles of its characteristic rays are in distinct orbits under the doubling map.
Let $P_p$ denote the set of primitive components of period $p$.

One combinatorial self-similarity feature of the Mandelbrot set is the maps $g$ in the wake of $f\in P_p$ also have a $p$-cycle reflecting $f$ (see for example \cite{BuffTanLei}, \cite{Milnor-External}).
\begin{proposition}
    Suppose $f\in P_p$ has characteristic rays $\theta_0<\theta_1$.
    Then for every parameter $c$ in the wake $\cal{W}(\theta_0, \theta_1)$, the rays $\theta_0$ and $\theta_1$ also land together in $J_c$.
\end{proposition}

Conjecturally, every parameter ray lands, and for dynamical rays we have the pair of statements:

\begin{theorem}\label{thm:rat-ray-landing-dynam}
For dynamical rays, the following hold:
    \begin{enumerate}
        \item [(a)] For $c \in \C \setminus \M$, every dynamical ray either bifurcates or lands;
        \item [(b)] For $c \in  \M$, every dynamical ray of rational angle $\theta\in \Q/\Z$ lands.
    \end{enumerate}
\end{theorem}

\begin{proof}
    Part (a) follows from straightforward considerations on Cantor set Julia sets; for details see \cite[Prop. 1.2.13]{danny}. Part (b) is \cite[Prop. 8.4]{Douady_exploringthe}.
\end{proof}

In fact, noting that dynamical plane bifurcation of rays occurs precisely at the critical points of the Green's function (and their iterated preimages),
we can state part (a) of \Cref{thm:rat-ray-landing-dynam} more precisely.
Recall we say that $c\in \C\sm\M$ has external argument $\theta$ when $c$ lies on the parameter external ray $\theta$.

\begin{proposition}\label{prop:bifurcation-orbit}
    If $c\in\C\sm\M$ has external argument $\theta$, $\theta\in \mathbb{S}^1$, then the dynamical ray of $f_c$ at angle $\theta'\in \mathbb{S}^1$ bifurcates if and only if $2^k\theta' = \theta$ for some $k\in\N^*$.
    If this is the case, then the point at which $R_c(\theta')$ bifurcates is an element of $f_c^{\circ - k}(c)$.
\end{proposition}
\begin{proof}
    Bifurcation of rays happens at critical points of the Green's function and their iterated preimages.
    Precisely, this is if and only if $R_c(\theta')$ passes through the critical value $c$ for $f_c$.
    This occurs when $2^k\theta' = \theta$, see \cite[Prop. 1.2.11]{danny}  for additional details.
\end{proof}

This lets us define active wakes for a particular rational $\theta$ whereupon the motion of external rays is only obstructed from being holomorphic by ray bifurcation.
\begin{definition}\label{defn:active-per1-wake}
    A wake $W = \cal{W}(\alpha, \beta)$ is \emph{active} at $\theta\in\Q/\Z$ if the orbit of $\theta$ under angle doubling contains either $\alpha$ or $\beta$.
\end{definition}

For a rational angle $\theta$, denote by $\cal{O}_\theta$ the union of the closures of all external rays in the orbit of $\theta$ under angle doubling.
In other words, letting $(\theta)$ denote the orbit of $\theta$ under the doubling map,
\[
    \cal{O}_\theta := \bigcup_{\alpha\in\del{\theta}} \ol{R_M(\alpha)}.
\]

For $c\in\C\sm\cal{O}_\theta$, it follows from \Cref{prop:bifurcation-orbit} and \Cref{thm:rat-ray-landing-dynam} that $R_c(\theta)$ lands at a point
$\gamma_\theta(c)\in J_c$.

\begin{lemma}\label{lem:dynam-ray-holomorphic-landing}
    For $\theta\in\Q/\Z$ fixed, the map $c \mapsto \gamma_\theta(c)$ is holomorphic on $\C\sm\mathcal{O}_\theta$.
\end{lemma}
\begin{proof}
    We first overview the proof, as we will aim to generalize it later to apply to \emph{bubble rays} in $\Per_2(0)$.
    There are three essential ingredients: First, the construction of a dynamical ray (here from the Green's function).
    Second, understanding when the dynamical rays bifurcate (at preimages of critical points).
    Third, promoting continuity to holomorphicity via Montel's theorem.

    Fix $\theta\in\Q/\Z$, and let $V = \C\sm\cal{O}_\theta$.
    For $\e>0$, define a map
    $ g_\e: V \lra \C $ by
    \[
        g_\e(c) := \Phi_c(\e e^{2 \pi i \theta}) \in \C\sm K_c.
    \]
    Note that $g_\e$ is continuous, since by \Cref{prop:bifurcation-orbit}, $R_c(\theta)$ does not bifurcate for $c\in V$.
    This accomplishes the first two steps, and the rest is about promoting this continuity to holomorphicity.

    First, we claim that $g_\e$ is holomorphic.
    To see this, fix some $c_0\in V$, and choose a neighborhood $U$ of $c$ whose closure $\ol{U}$ is compact and contained in $V$.
    Since the parameter space Green's function $G_\M$ is bounded on the compact set $\ol{U}$ (where we put $G_\M(c)=0$ for $c\in\M$), we may find $N\in\N$ sufficiently large such that
    \[
        2^N \e > G_\M(c)
    \]
    for all $c\in U$.
    It follows that
    \[
        G_c(f_c^{\circ N}(g_\e(c))) = 2^N G_c(g_\e(c)) = 2^N\e > G_\M(c) = G_c(c)
    \]
    for all $c\in U$.
    Thus, $\Phi_c^{-1}$ is well-defined at $f_c^{\circ N}(g_\e(c))$.
     Since $g_\e(c)$ is not an iterated preimage of a critical point of $f_c$ (otherwise by \Cref{prop:bifurcation-orbit}, $c$ would lie on the boundary of an active wake), the inverse function theorem implies that the holomorphic function
    \[
        F(z, c) = (f_c^{\circ N}(z), c)
    \]
    has a holomorphic inverse $F^{-1}$ defined on some neighborhood of $F(g_\e(c_0), c_0)$.
    Since
    \[
        F(g_\e(c), c) = \Phi_c(\exp(2^N( 2 \pi i  \theta + \e)))
    \]
    depends holomorphically on $c$, we have that $g_\e(c)$ also depends holomorphically on $c$ in a neighborhood of $c_0$.
    This proves that $g_\e$ is holomorphic on $V$.

    Now, note that $g_\e(c)$ is nonzero on $V$ since $g_\e(c)$ belongs to $\textup{Im }\Phi_c = \{ z \ : \ G_c(z) > G_c(0) \}$.
    Moreover, we also have $g_\e(c) \ne c$ on $V\sm\set{0}$ since $g_\e(c)$ belongs to the ray $R_c(\theta)$,
    but $c$ does not belong to it since $c \notin \mathcal{O}_\theta$. Thus, Montel's theorem implies the maps $\set{g_\e: \e>0}$
    form a normal family on $\C\sm\set{0}$.
    Since the ray $R_c(\theta)$ lands, for every $c \in V$ the limit $\lim_{\epsilon \to 0} g_\e(c) = \gamma_\theta(c)$ exists.
    Now, using Montel's theorem, the convergence $\gamma_\theta = \lim_{\e\ra 0}g_\e$ is uniform on compact subsets,
    hence the limit is holomorphic on $V\sm\set{0}$.
    Since $\gamma_\theta$ is holomorphic and bounded in a neighborhood of $0$, it follows from classification of singularities that $\gamma_\theta$ extends analytically
    to $V$.
\end{proof}

We present two corollaries useful in the proof of the cell structure, which are proved in \cite{danny}.

\begin{corollary}\label{cor:co-landing-implies-primitive}
    Suppose $\theta,\theta'\in\Q/\Z$ are distinct with period $p$ under angle doubling, and let $\gamma_\theta$, $\gamma_{\theta'}$ be as in \Cref{lem:dynam-ray-holomorphic-landing}.
    Suppose further that $\theta$ and $\theta'$ do not share an orbit.
    For any $c\in\C\sm(\cal{O}_\theta\cup\cal{O}_{\theta'})$, we have that $\gamma_\theta(c)$ and $\gamma_{\theta'}(c)$ share an orbit if and only if $c$ belongs to a primitive wake that is active at $\theta$ and $\theta'$.
\end{corollary}
\begin{corollary}\label{cor:unique-ext-arg}
    For a $p$-periodic angle $\theta$, if $c$ does not belong to the closure of any wake active at $\theta$, then $\gamma_c(\theta)$ has exact period $p$, and $\theta$ is its unique external argument.
\end{corollary}

\subsection{Kneading data}
The final tool we will use largely in this paper is kneading data.
Let $f_c$ be a map in a hyperbolic component of period $p$ with characteristic angle $\theta$.
Note that both rays $R_c\left({\theta}/{2}\right)$ and $R_c\left(({\theta+1})/{2}\right)$ land at the same point, the root of the Fatou component containing the critical point of $f_c$.
Hence, the union of external rays
\[\Delta_c:= \overline{R_c\left( \dfrac{\theta}{2}\right)} \cup \overline{R_c\left(\dfrac{\theta+1}{2}\right)}\]
disconnects the complex plane $\C$.
If $\theta \ne 0$, we label the component of $\C \setminus \Delta_c$ which contains the critical value $c$ with a $1$, and the other component with a $0$.
If $\theta = 0$, by convention we label the upper half plane with a $0$ and the lower half plane with a $1$.

\begin{definition}
    Fix some $f_c$, and let $z$ be a point in $J_{c}$.
    The \emph{itinerary} of $z$ with respect to $c$ is the binary string
    \[\textup{itin}(f, z) := (\e_k)_{k \geq 0}\]
    where the $k$th digit records which component of $\C\setminus \Delta_c$ contains $f_c^{\circ k}(z)$.
    If $f^{\circ k}(z)$ is the landing point of either $R_c(\theta/2)$ or $R_c((\theta+1)/2)$, then we record that it is in both components with an $\ast$ instead of a 0 or 1.
\end{definition}

We can further define the kneading sequence of a hyperbolic component:

\begin{definition}
    Let $H$ be a hyperbolic component with characteristic angle $\theta$, and let $c$ be its center.
    Then, the \emph{kneading sequence} of $H$ is the itinerary of $c$ with respect to $f_c$.
    Note that this is independent of the choice of characteristic angle.
\end{definition}

Each such kneading sequence is of the form $\sigma_1\sigma_2\cdots \sigma_{p-1}\ast$.
In this way, we get a function $K: \mathcal{H}_p\to \{0,1,\ast\}$ from the set $\mathcal{H}_p$ of hyperbolic components of period $p$ to binary strings of length $p-1$ with an $\ast$ in the $p$th position, reflecting the ambiguity of which characteristic angle we chose to define the partition.
One last characterization of primitive components is that $H$ is primitive if and only if its kneading sequence produces two binary strings of exact period $p$ when interpreting $\ast$ as 0 or 1.

Much work in this paper will be devoted to interpreting (strictly) $p$ periodic binary strings as external angles, kneading sequences, or elements in $C_p$.
Although in general the correspondence between $p$-binary strings and those obtained as kneading sequences from hyperbolic components is neither surjective nor injective, when $c$ is positive and outside $\M$, we have a simple but useful equivalence between kneading sequences and external angles.

\begin{lemma}\label{KD-EA}
    Let $c\in \R$, $c>1/4$.
    Suppose the ray $R_c(\theta)$ lands at point $z$.
    Then, the itinerary of  $z$ with respect to $c$ is the
    binary expansion of the external angle.
\end{lemma}
\begin{proof}
    For $c\not\in \M$, we may still define its kneading sequence via the partition $\Delta_c=R(\theta/2)\cup R((\theta+1)/2)$ according to the external ray $\theta$ on which $c$ lies.
    When $c\in \R$, $c>1/4$, the partition $\Delta_c$ is the upper and lower half plane.
    Note that in this case the Julia set does not intersect the real line, hence the coding of the landing point of any ray $R_c(\theta)$
    is well-defined.
\end{proof}

\section{The cell complex structure of marked cycle curves} \label{S:cell-p1}

\subsection*{Vertices}

Vertices are defined as lifts of the parameter $0 \in \mathcal{M}$, that is as lifts of the polynomial $f_0(z) = z^2$.
We denote each vertex as $(0, \zeta)$, where $\zeta$ is a cycle of period $p$ for $f_0$.

The following proposition justifies our combinatorial parameterization of vertices.

\begin{proposition} \label{P:vertices}
    The lifts of $0 \in \mathcal{M}$ under the branched cover $\pi: \MC_p(\mathcal{F}_1) \to \widehat{\mathbb{C}}$ are in bijective
    correspondence with the set of abstract cycles $C_p$.
\end{proposition}

\begin{proof}
    By definition, lifts of $0$ correspond to distinct pairs $(0,\zeta)$ where $\zeta$ is a $p$-cycle for $f_0(z) = z^2$.
    Every non-fixed periodic point of $f_0$ is repelling hence in the Julia set, which is the unit circle. Moreover, on each point of the circle lands exactly one external ray.

    Each $p$-periodic point $x \in J_0$ is the landing point of a unique dynamic ray $R_0(\theta)$, where $\theta$ is of least period $p$ under doubling: let $\theta = .\overline{\theta_0 \theta_1\cdots \theta_{p-1}}$ be its binary expansion.
    Since for $x', x\in \zeta$, the sequence $\varphi(x)$ is a rotation of $\varphi(x')$, the map $\varphi$ descends to a map $\Phi$ from
    $p$-cycles for $f_0$ to abstract cycles $C_p$.
    The map $\Phi$ is surjective since every $p$-periodic binary sequence determines an external ray, and that ray lands at a $p$-periodic point.
    Moreover, $\Phi$ is injective since $\varphi(x)$ is a cyclic permutation of $\varphi(x')$ if and only if $x$ and $x'$ belong to the same orbit
    for $f_0$.
 \end{proof}

E.g.: for $p = 5$, there are $6$ vertices:
\[
    \begin{array}{ll}
        A:            & 00001 \\
        \overline{A}: & 01111 \\
        B:            & 00011 \\
        \overline{B}: & 00111 \\
        C:            & 00101 \\
        \overline{C}: & 01101
    \end{array}
\]

\subsection*{Faces}
Faces are defined as closures of connected components of the lifts of $\widehat{\C} \setminus V_p$.
\begin{remark}
    Recall that $V_p$ denotes the union of all veins in $\M$ connecting $0$ to the roots of primitive hyperbolic components of period $p$.
    Since $V_p$ is a topological tree, $\widehat{\C}\sm V_p$ is simply connected and it contains $\widehat{\C}\sm \M$, which is also simply connected.
    Thus, by homotopy extension, we can replace lifts of $\widehat{\C}\sm V_p$ with lifts of $\widehat{\C}\sm \M$ in our definition of faces.
\end{remark}
We show in the following lemma that faces are labelled by pairs of conjugate period $p$ cycles.

\begin{lemma}\label{lem:faces_are_cyclepairs}
    The lifts of $\widehat{\mathbb{C}} \setminus \mathcal{M}$ under the branched cover $\pi: \MC_p(\mathcal{F}_1) \to \widehat{\mathbb{C}}$ are in bijective
    correspondence with the set of cycle duos $[C_p]$.
    Moreover:
    \begin{itemize}
        \item if the cycle duo is reflexive, then the restriction of $\pi$ to the associated disc is a homeomorphism;
        \item
              if the cycle duo is not reflexive, the restriction of $\pi$ on the associated disc has local degree $2$, branched at the unique lift of $\infty$.
    \end{itemize}
\end{lemma}
\begin{proof}
    Fix an equipotential line $\gamma \subseteq \mathbb{C} \setminus \M$, and let $c_+$ denote the unique point in $\R^{+} \cap \gamma$.
    For every $c \in \gamma$,
    the Julia set $J(f_c)$ is a Cantor set.
    Consider the bundle
    \[\mathcal{J} := \set{ (c, z) \in (\mathbb{C} \setminus \M) \times \C\ : \ z \in J(f_c) }\]
    which is a fiber bundle over $\mathbb{C} \setminus \M$ with Cantor sets as fibers. Let $z$ be a periodic point for $f_{c_+}$.
    For each $\theta \in \intco{0, 2 \pi}$, we can follow the periodic point $z$ to a point $z(\theta)$ in the Julia set for $f_c$ with $c = \Phi_M(r e^{2 \pi i \theta})$, along $\gamma$.
    Let $R_c(\theta)$ denote the external ray at angle $\theta$ in $J_c$.
    Since each Julia set is a Cantor set which does not intersect the diagonal $\Delta_\theta := R_c(\frac{\theta}{2}) \cup R_c(\frac{\theta+1}{2})$, the itinerary  $\itin(f_c, z(\theta))$ is constant for $\theta \in [0, 2 \pi)$.

    However, the diagonal $\Delta_\theta $ completes only a half turn as $\theta$ makes one full turn in $\R/{2\pi \Z}$.
    Hence, the monodromy induced by $\gamma$ interchanges the two
    complementary regions of $\Delta_\theta$.
    So,
    \[\lim_{\theta \to 2\pi^-} z(\theta) = -z,\]
    and the itinerary of $-z$ with respect to $c_+$ is obtained from  $\itin(f_{c_+},z)$ by flipping the bits. By varying $r$, we also see that $\itin(f_c,z)$ is independent of $\gamma$.

    We have a natural map $\varphi: \MC_p(\mathcal{F}_1) \setminus \pi^{-1}(\mathcal{M}) \ra [C_p]$ which sends a point $(c,\zeta)\in \MC_p(\mathcal{F}_1)$ to $[\nu] \in [C_p]$ where $\nu = \itin(f_c, z)$ for some point $z \in \zeta$.
    By the above discussion, $\varphi$ is constant on lifts of $\widehat{\C} \setminus \mathcal{M}$.
    We thus obtain a well-defined map $\vp$ from the lifts of  $\widehat{\C} \setminus \mathcal{M}$ to the set $[C_p]$.

    Recall $c_+$ is the (unique) point in $\R^+\cap \gamma$.

    Let $[\nu]\in [C_p]$.
    Then, interpret $\nu$ as the binary expansion of some $p$ periodic angle $\theta$, and let $\zeta$ be the periodic cycle in $J_{c_+}$ formed by the orbit of $\theta$.
    Using \Cref{KD-EA}, the face of $\Sigma_{p, 1}$ containing $(c_+, \zeta)$ is sent to $[\nu]$ under the map $\varphi$. It follows that $\varphi$ is surjective.

    Now we show that $\varphi$ is injective. Suppose that $\vp(c_+, \zeta_0)=\vp(c_+, \zeta_1)$ for two lifts $(c_+, \zeta_0), (c_+, \zeta_1) \in$  $\MC_p(\mathcal{F}_1)$.
    We show both lifts are in the same face to conclude $\vp$ is injective.
    From our hypothesis, there exist $z_0 \in \zeta_0$ and $z_1 \in \zeta_1$ such that one of the following conditions are satisfied:
    \begin{align*}
        \itin(f_{c_+}, z_0)=\itin(f_{c_+}, z_1) \text{ or } \itin(f_{c_+}, z_0)=\overline{\itin(f_{c_+}, z_1)}.
    \end{align*}
    In the former case, \Cref{KD-EA} shows both $\zeta_0$ and $\zeta_1$ corresponding to the same external angle and hence, $\zeta_0 = \zeta_1$.
    In the latter case, following the monodromy of $z_0$ once around $\gamma$, we obtain a new $p$-periodic point $\widehat{z}_0$ of $f_{c_+}$ such that $\itin(f_{c^+}, \widehat{z}_0) = \overline{\itin(f_{c^+}, z_0)}$.  Thus,  the cycle $\widehat{\zeta}_0$ containing $\widehat{z}_0$ coincides with $\zeta_1$.
    We conclude that $(c_+,\zeta_1)$ and $(c_+,\zeta_0)$ are in the same lift of $\widehat{\C} \setminus \mathcal{M}$.

    For the final dichotomy of the lemma, we note that, for $\nu\in C_p$, with $\zeta$ the corresponding $p$-cycle in the dynamical plane of $f_{c^+}$,  the pre-images of $c^+$ under $\pi$ are  $(c^+,\zeta)$ and $(c^+,\ol{\zeta})$ . Thus, the degree of $\pi $ is the cardinality of the set $\set{\nu, \overline{\nu}}$.
\end{proof}
\noindent
In our example for period $5$, we have $3$ faces:
\[
    [A] = \{ 00001, 01111 \}, \qquad [B] =  \{ 00011, 00111 \}, \qquad [C] =  \{ 00101, 01101 \}.
\]

\subsection*{Edges}

Let $P_p$ be the set of primitive hyperbolic components of period $p$.
Both external angles landing at a primitive component of period $p$ are of the form $\theta = \frac{k}{2^p-1}$.
Hence, let us label the primitive components by choosing one of the two angles landing at the
root of the component, and writing down the value of $k$.

For instance, we can write
\[P_5 = \left\{ 3, 5, 7, 11, 13, 14, 15, 20, 24, 26, 28 \right \}\]
where $3$ represents the component of external ray $\frac{3}{31}$ (in this case, the other external ray landing on it is $\frac{4}{31}$).

\begin{proposition} \label{per1-edge-branching}
    The map $\pi : \MC_p(\mathcal{F}_1) \to \mathcal{F}_1$ is branched precisely over the roots of the primitive components of period $p$,
    as well as over $\infty$.
    At each root of a primitive component, the local degree equals $2$.
\end{proposition}

\begin{proof}
    The map $\pi$ is branched over $\infty$ as all cycles collide.
    In $\C$, this follows from Milnor's work on external rays and orbit portraits, e.g.\ \cite[Lemma 2.7]{Milnor-External}, Buff-Lei \cite[Proposition~4.4]{BuffTanLei}.
\end{proof}

\subsection*{Vertex relations: ray connections}

We now describe the vertex relations in the cell complex $\Sigma_{p, 1}$. To ease notation, we will sometimes denote $\Sigma_{p, 1}$ 
simply as $\Sigma_p$. 
Recall that vertices correspond to abstract cycles of period $p$, while edges correspond to primitive hyperbolic
components, also of period $p$. Given an angle $\theta$ of period $p$ under angle doubling, denote as $b(\theta)$
the abstract cycle associated to the binary expansion of $\theta$, modulo cyclic permutation.

\begin{proposition} \label{P:vertex-relation}
    Two vertices in $\Sigma_p$ are connected by an edge if and only if there exists a primitive hyperbolic component whose pair of characteristic angles' binary expansions agree with the abstract cycles  corresponding to the vertices.
\end{proposition}

\begin{proof}
    More precisely, given two abstract $p$-cycles $\nu_1$, $\nu_2$, we need to show that there exists an edge in $\Sigma_p$ connecting
    $\nu_1$ and $\nu_2$ if and only if there exists a primitive component $H$, with characteristic angles $(\theta_1, \theta_2)$,
    whose binary expansion satisfies $b(\theta_1) = \nu_1$ and $b(\theta_2) = \nu_2$.

    First, let us show that the two characteristic angles of a primitive hyperbolic component $H$ are joined by an edge in $\Sigma_p$.
    Consider two external angles $\theta_1, \theta_2$ of rays landing at the root $c_1$ of a primitive component of period $p$ in the Mandelbrot set.

    For each parameter $c$ on the vein joining $0$ and $c_1$ and $i = 1, 2$, let $z_i(c)$ denote the landing point of $R_c(\theta_i)$.
    For $c$ as above, since $c$ is not in the wake of $c_1$, we have $z_1(c)\ne z_2(c)$.
    For each $i$, the point $z_i(c)$ has period $p$ under iteration of $f_c$, and, since the landing point of $R_c(\theta_i)$ is repelling for $c \ne c_1$, the function $c \mapsto z_i(c)$ is continuous.

    Consider the two landing points $z_1, z_2$ of the rays at angle $\theta_1$, $\theta_2$ in the dynamical plane of $f_0(z) = z^2$.

    As one move the parameter $c$ along the vein from $c= 0$ to $c = c_1$, one can follow the two periodic points, obtaining $z_1(c)$ and $z_2(c)$ the landing points
    of $R_c(\theta_1)$ and $R_c(\theta_2)$.
    As $c\to c_1$ along the vein, the cycles $\zeta_1(c)$ and $\zeta_2(c)$, respectively containing $z_1(c)$ and $z_2(c)$, collide.
    Thus, the unique lift of the vein from $0$ to $c_1$ that starts at the point $(0, \zeta_1(0)) \in \MC_p(\mathcal{F}_1)$ ends at the point $(c_1, \zeta_1(c_1)) = (c_1, \zeta_2(c_1)) \in \MC_p(\mathcal{F}_1)$.
    Similarly, going back along the vein from $c_1$ to $0$ and lifting starting from $(c_1, \zeta_1(c_1))$, yields a path
    in $\MC_p(\mathcal{F}_1)$ that ends at the point $(0, \zeta_2(0))$.
    Thus, the lift of the vein segment from $0$ to $c_1$ is a path from $(0, \zeta_1(0))$
    to $(0, \zeta_2(0))$ that maps with degree $2$ onto its image.
    This is an edge in our cell complex structure for $\MC_p(\mathcal{F}_1)$.

    For the other direction, we note that by \Cref{per1-edge-branching}, the only branching points of $\pi$ are at the roots of primitive components.
\end{proof}

\medskip
\subsection*{Face relations: kneading sequences}

Recall that each hyperbolic component $H$ of period $p$ has a kneading sequence of the form $k(H) = \sigma_1 \sigma_2 \dots \sigma_{p-1} *$,
with $\sigma_i \in \{0, 1\}$ for each $i$.
From each such sequence we can produce the two abstract cycles:
\[k_0(H) := \sigma_1 \sigma_2 \dots \sigma_{p-1} 0 \qquad \textup{and} \qquad k_1(H) := \sigma_1 \sigma_2 \dots \sigma_{p-1} 1.\]
We say that $k_0(H)$ and $k_1(H)$ are \emph{perturbations} of $k(H)$.

\begin{lemma} \label{L:knead}
    For any primitive component $H$ of period $p$, the edge labeled by $H$ in $\Sigma_p$ bounds the faces labeled under the correspondence of \Cref{lem:faces_are_cyclepairs}
    as the cycle duos $[k_0(H)]$ and $[k_1(H)]$.
\end{lemma}

\begin{proof}

    Let $c$ be the root of $H$, and let $\theta, \eta$ with $\theta < \eta$ be its characteristic angles, and $c_0$ be the center of $H$.
    Let $f_{c'}$ be the tuning of $f_{c_0}$ by the basilica, whose corresponding external rays are $\theta', \eta'$ with $\theta' < \eta'$.
    Let $c''$ be a point in $\C \setminus \M$ in the wake $\cal{W}(\theta', \eta')$.
   
    Consider the dynamical plane for $f_{c''}$.
    To determine the itinerary of a point, we have to look at the partition of the plane determined by the curve $\Delta_{\theta''} := \ol{R_{c''}(\frac{\theta''}{2}) \cup R_{c''}(\frac{\theta''+1}{2})}$.
    Let $z, z'$ be the landing points of $\theta$ and $\theta'$, respectively, which are repelling in $J_{c''}$.
    Note, since $c''$ is in the wake of both $c$ and $c'$, that $z, z'$ are also the landing points of $\eta$ and $\eta'$.

    Now, $f_{c''}^{\circ k}(z)$ and $f_{c''}^{\circ k}(z')$ lie on the same side of the curve $\Delta_{\theta''}$ for any $k < p - 1$.
    Let us look at the $p$th digit of their itinerary.
    Note that we have $\theta' < \theta'' < \eta$, which implies  $\theta'/2 < \theta''/2 < \eta/2 < \frac{\theta''+1}{2}$.
    Thus, $f_{c''}^{\circ (p-1)}(z)$ and $f_{c''}^{\circ (p-1)}(z')$ lie on two distinct sides of the partition, hence the two periodic points $z, z'$ have itineraries, respectively, $k_0(H)$ and $k_1(H)$, with respect to $c''$.

    We now conclude the proof.
    Consider the pair of hyperbolic components $H^\pm $ in $\MC_p(\mathcal{F}_1)$ whose mutual root is a branch point mapping to $c$.
    Since the projection $\pi$ has local degree $2$ at $c$, the two lifts $H^\pm$ lie on opposite sides of the edge corresponding to $H$.
    Note that in one face the cycle corresponding to $z$ is marked, while in the other face the marked cycle is the one corresponding to $z'$;
    thus, the faces adjacent to the edge $H$ are labelled by the itineraries of $z$ and $z'$, hence by the cycle duos $[k_0(H)]$ and $[k_1(H)]$.
\end{proof}

\begin{figure}
    \centering
    \begin{overpic}[scale=.5]{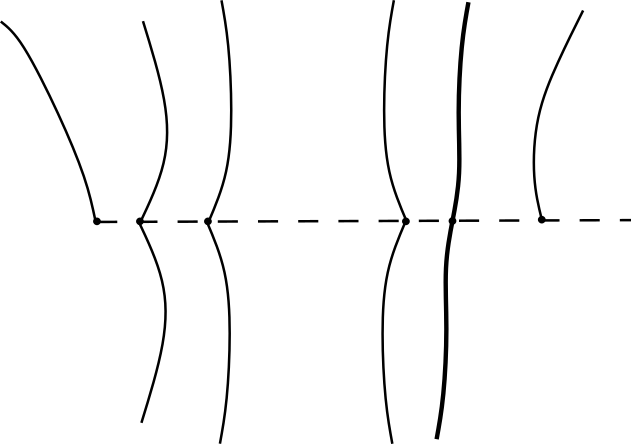}
        \put(14,30){$c''$}
        \put(25,38){$z'$}
        \put(31,30){$z$}
        \put(40,38){$f^{\circ(p-1)}(z)$}
        \put(74,38){$0$}
        \put(84,30){$f^{\circ(p-1)}(z')$}
        \put(6,61){$\theta''$}
        \put(25,61){$\theta'$}
        \put(25,8){$\eta'$}
        \put(37,61){$\theta$}
        \put(37,8){$\eta$}
        \put(62,61){\large $\frac{\eta}{2}$}
        \put(56,8){\large $\frac{\theta}{2}$}
        \put(74,61){\large $\frac{\theta''}{2}$}
        \put(72,8){\large $\frac{\theta''+1}{2}$}
        \put(92,61){\large $\frac{\theta'}{2}$}
    \end{overpic}
    \caption{The angles in the proof of \Cref{L:knead}.}
\end{figure}
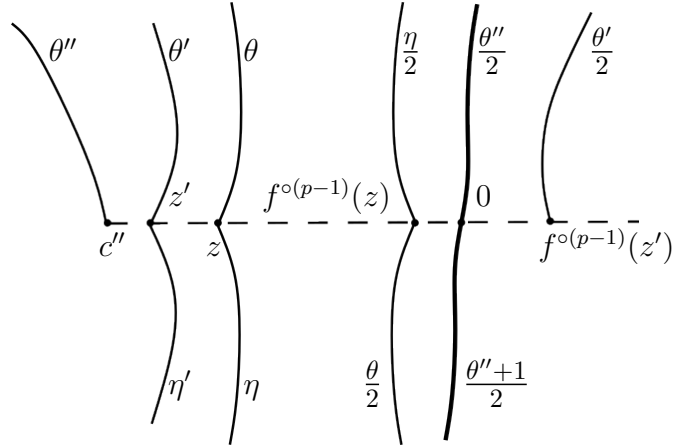

\subsection{The algorithms}

Given this setup, we are now in a position to algorithmically construct the cell decomposition for $\MC_p(\mathcal{F}_1)$.
In fact, we will now describe two algorithms, which we will then show to be equivalent.

\begin{algorithm}[``Telephone algorithm", based on external rays]\label{algo-per1-angle}\mbox{}

    \begin{enumerate}

        \item Label each pair of characteristic angles of a primitive component with their associated pair of cycles in $C_p$.
              For instance, if the characteristic angles are $(\frac{3}{31}, \frac{4}{31})$, this is written in binary as $(00011, 00100)$,
              hence one associates to this pair the label $(B, A)$.

        \item Let us pick a cycle, for instance $A$.
              We will describe how to obtain the face associated to the cycle duo $[A]$.

        \item Start at the external angle $\theta = 0$ in the Mandelbrot set, and, proceeding counterclockwise, look for the first ray pair $(\theta_1, \theta_2)$ for which either $\theta_1$ or $\theta_2$ has label $A$.

        \item When you see a ray pair $(\theta_1, \theta_2)$ where one of the two labels is $A$, write down an edge
              with label $(\theta_1, \theta_2)$.
              Now, mark the vertex corresponding to the other label of this ray pair.
              For instance, let us suppose that the other label is $B$.

        \item Continue in counterclockwise order up until you see a ray pair where one of the two labels is $B$.

        \item Continue in this way until you go back to finding $A$ and the next edge you put in at $A $ is something you've already entered.
    \end{enumerate}

\end{algorithm}

\begin{example}
    Let consider $p = 5$, and let us start with the $A$ cycle.
    Since all external angles involved are of the form $\frac{k}{31}$, for ease of notation we will drop the denominators
    when referring to external angles. For instance, we will write $(3, 4)$ to refer to the ray pair $(\frac{3}{31}, \frac{4}{31})$.
    \begin{enumerate}
        \item
              Starting at the angle $\theta  =0$ and going in counterclockwise order, the first ray pair with an $A$ label is the pair $(3, 4)$, which has label $(A, B)$.
              Thus, write down the edge $(3, 4)$ and the new vertex is $B$.
        \item
              Then the next pair containing $B$ is $(5, 6)$, labeled $(B, C)$.
              Hence, write down the edge $(5, 6)$ and the new vertex is $C$.
        \item
              The next pair containing $C$ is $(13, 18)$ with label $(C, \overline{C})$.
              Hence, write down the edge $(13, 18)$ and the new label is $\overline{C}$.
        \item
              The next pair containing $\overline{C}$ is $(25, 26)$ with label $(\overline{B}, \overline{C})$. Hence, add the edge is $(25, 26)$ and the new label is
              $\overline{B}$.
        \item
              The next pair containing $\overline{B}$ is $(27, 28)$.
              Hence, add the edge $(27, 28)$ and the new vertex label is $\overline{A}$.
        \item
              Finally, the next pair containing $\overline{A}$ is $(15, 16)$, with label $(A, \overline{A})$. Hence, write down the edge $(15, 16)$ and the new vertex label is
              $\overline{A}$.
              If one keeps going counterclockwise, the next vertex labeled with $A$ is $(3, 4)$, which we already encountered, so the process stops.
    \end{enumerate}
    As a result, we have constructed the $\face{A}$ face:
    \[
        \begin{tikzcd}[column sep=0.6cm, row sep=0.6cm]
            & B \ar{dr}{5} & \\
            A \ar{ur}{3} & & C \ar{d}{13} \\
            \overline{A} \ar{u}{15} & & \overline{C} \ar{dl}{26} \\
            & \overline{B} \ar{ul}{28} &
        \end{tikzcd}
    \]
    (note that, for ease of visualization, in the picture we labeled edges with only one of the two angles).
    Similarly, we get the $\face{B}$ face:
    \[
        \begin{tikzcd}[column sep=0.6cm, row sep=0.6cm]
            & \overline{A} \ar{r}{28} & \overline{B} \ar{dr}{7} & \\
            B \ar{ur}{24} & & & A \ar{d}{15} \\
            \overline{B} \ar{u}{14} & & & \overline{A} \ar{dl}{24} \\
            & A \ar{ul}{7} &B \ar{l}{3} &
        \end{tikzcd}
    \]
    and the $\face{C}$ face:
    \[
        \begin{tikzcd}[column sep=0.6cm, row sep=0.6cm]
            & \overline{B} \ar{r}{26} & \overline{C} \ar{dr}{11} & \\
            C \ar{ur}{20} & & & B \ar{d}{14} \\
            \overline{C} \ar{u}{13} & & & \overline{B} \ar{dl}{20} \\
            & B \ar{ul}{11} & C \ar{l}{5} &
        \end{tikzcd}
    \]
    Note that each edge appears exactly twice, with opposite orientation; to obtain the desired cell decomposition, we now glue the edges with the same number
    so that vertices match up.
\end{example}

\begin{proposition} \label{P:algo1works}
    The cell complex constructed by \Cref{algo-per1-angle} is isomorphic to the cell complex $\Sigma_p$ defined in \Cref{D:Sigma_p}.
    As a consequence, the surface constructed by the algorithm is homeomorphic to $\MC_p(\mathcal{F}_1)$.
\end{proposition}

\begin{proof}
    By \Cref{P:vertices}, the vertices from \Cref{algo-per1-angle} are in one-one correspondence with vertices of $\Sigma_p$.  Moreover, by \Cref{P:vertex-relation}, every edge produced by the algorithm corresponds to an  edge of $\Sigma_p$. Thus, it suffices to show that \Cref{algo-per1-angle}  produces exactly the faces of $\Sigma_p$.

    The surface $\MC_p(\mathcal{F}_1)$ is a Riemann surface, whose orientation is given by pulling back the standard orientation on $\widehat{\C}$ by the holomorphic map $\pi : \MC_p(\mathcal{F}_1) \to \widehat{\C}$.
    Let $\cal{R}$ be the tubular neighbourhood of $V_p$ bounded by an equipotential line $\gamma$. Each face of $\Sigma_p$ contains a unique lift of $\partial \cal{R}$ under $\pi$, and each such lift is contained in some face. Moreover, pulling back the standard orientation on $\partial \cal{R} = \gamma$, we induce an orientation on each lift of $\gamma$. 

    Let $c^+ := \gamma \cap \mathbb{R}^+$. Given an abstract cycle $\nu \in C_p$, let $(c^+, \zeta)$ be the lift of $c^+$ to $\MC_p(\mathcal{F}_1)$ that corresponds to $\nu$.
    The point  $(c^+, \zeta)$ lies in the interior of a face $F$ of $\Sigma_p$, which by definition is the closure $F = \overline{U}$ of some lift $U$ of $\widehat{\C} \setminus V_p$. Since $\widehat{\C} \setminus V_p$  has boundary $V_p$, the boundary of $U$ is a union
    $\partial U = \bigcup_{i \in I} e_i$, where each $e_i$ is a lift of some vein from $0$ to a root of a primitive hyperbolic component.
    Some $e_i$'s contain a branch point (hence we call them \emph{branched lifted veins}) and some of them do not.
    Note that $\partial F$, which is the boundary of the \emph{closure} of $U$, is the union
    \[\partial F = \bigcup_{\stackrel{i \in I}{e_i \textup{ branched} }} e_i\]
    of all the branched lifted veins in $\partial U$.

    Recall that, by \Cref{P:vertex-relation}, any branched lifted vein $e_i$ is a  path between vertices $v_i =(0, \zeta_i)$ and $v_{i+1}=(0, \zeta_{i+1})$ where $\zeta_i$ and $\zeta_{i+1}$ each correspond to different elements $[\theta_i], [\theta_{i+1}]$ in $C_p$, and $\theta_i, \theta_{i+1}$ are the external rays landing at the branch point of $\pi(e_i)$.

    Given a branched lifted vein $e_i$, let $e_{i+1}$ be the next one in the cyclic order given by the face $F$.
    Then $e_{i+1}$ is the first lifted vein with respect to the cyclic order induced by $\gamma$ that is branched, so
    $\pi(e_{i+1})$ is the first vein after $\pi(e_i)$ in the cyclic order for which one of its associated external rays lies in $[\theta_{i+1}]$.
    Thus, $e_{i+1}$ corresponds to a pair of angles $[\theta_{i+1}], [\theta_{i+2}]$, and we repeat the process until we obtain again $e_1$.

    This produces all sides of the face $F$ associated to $[\nu]$, showing that the order of combinatorial edges produced by \Cref{algo-per1-angle} coincides with the cyclic order of edges of $F$ in $\Sigma_p$.

    Finally, note that by construction the algorithm produces every face starting from all different lifts of the positive real axis:
    since every face of $\Sigma_p$ intersects a lift of the positive real axis, the algorithm produces all faces of $\Sigma_p$.
    Moreover, since every edge of $\Sigma_p$ lies on the boundary of some face, the algorithm also produces every edge.
\end{proof}

\begin{algorithm}[``Bar method", based on kneading sequences]\label{algo-per1-knead}\mbox{}
    \begin{enumerate}
        \item
              List all the kneading sequences of the primitive hyperbolic components of period $p$

        \item
              Each kneading sequence is of the form $\sigma_1 \sigma_2 \dots \sigma_{p-1} *$.
              From each such sequence we can produce two sequences:
              \[\sigma_1 \sigma_2 \dots \sigma_{p-1} 0 \qquad \textup{and} \qquad \sigma_1 \sigma_2 \dots \sigma_{p-1} 1.\]
              To each of the sequences above, we assign the corresponding cycle in $C_p$.
              When $p = 5$, for instance, we obtain the following table:

              \[
                  \begin{array}{|l|l|l|}
                      \hline
                      \textup{angle} & \textup{kn.seq.} & \textup{cycles}            \\
                      \hline
                      3 \sim 28      & 1110*            & \overline{A}, \overline{B} \\
                      5 \sim 26      & 1101*            & \overline{A}, \overline{C} \\
                      7 \sim 24      & 1100*            & B, \overline{B}            \\
                      11 \sim 20     & 1010*            & C, \overline{C}            \\
                      13             & 1011*            & \overline{C}, \overline{A} \\
                      14             & 1001*            & C, \overline{B}            \\
                      15             & 1000*            & A, B                       \\
                      \hline
                  \end{array}
              \]

        \item
              Let us fix a cycle duo, say $[A]$. We list  all the angles whose kneading sequences correspond to $A$ or $\overline{A}$. In order to keep track of the difference between $A$ and $\overline{A}$, we accent the angles corresponding to  $\overline{A}$  with a ``bar".

              In our example, the angles  corresponding to $A$ or $\overline{A}$ are:
              \[\overline{3}, \overline{5}, \overline{28}, \overline{26}, \overline{13}, 15 \]

        \item
              In this list, first write the barred numbers in ascending order, and then the unbarred ones in ascending order.
              In our example, we obtain the ordered list
              \[\textup{sides}([A]) ;= (\overline{3}, \overline{5}, \overline{13}, \overline{26}, \overline{28}, 15)\]

        \item We define the combinatorial face corresponding to $[A]$ as a polygon whose sides are labelled by the elements of  of the list  $ \textup{sides}([A])$ above, with counterclockwise order coinciding with the list order.

        \item We do this for every other cycle duo in $[C_p]$.
              In particular, we obtain
              \[\textup{sides}([B]) = (\overline{3}, \overline{7}, \overline{14}, \overline{24}, \overline{28}, 7, 15, 24)\]
              \[\textup{sides}([C]) = (\overline{5}, \overline{11}, \overline{13}, \overline{20}, \overline{26}, 11, 14, 20)\]

        \item Lastly we form a surface  from the combinatorial faces defined above, by gluing together sides with the same numerical label (regardless of whether they have a ``bar" ).
    \end{enumerate}
\end{algorithm}

\begin{proposition}
    The cell complexes produced by \Cref{algo-per1-angle,algo-per1-knead} are isomorphic.
\end{proposition}

\begin{proof}
    It is sufficient to prove that the faces produced by \Cref{algo-per1-knead} are the same as the faces of $\Sigma_p$.

    Consider the map
    $\psi: \MC_p(\mathcal{F}_1) \setminus \pi^{-1}(\mathcal{M} \cup \mathbb{R}^+) \ra C_p$ which sends a point $(c,\zeta)\in \MC_p(\mathcal{F}_1)$ to $\nu \in C_p$ where $\nu = \itin(f_c, z)$ for some point $z \in \zeta$. Since this map is continuous, it is constant on connected components of its domain.

    Fix a cycle duo $[\nu]$, with $\nu \in C_p$, and let $F$ be the face constructed by the algorithm corresponding to $[\nu]$.
    Let $\Theta_\nu$ be the set of external angles whose kneading sequences can be perturbed to $\nu$.

    By the algorithm, the edges of $F$ correspond to primitive hyperbolic components whose kneading sequences can be perturbed
    to either $\nu$ or $\overline{\nu}$.

    Let $\gamma$ be some equipotential line in $\mathbb{C} \setminus \mathcal{M}$,
    and let $e$ be an edge of $F$; then $e$ corresponds to some edge $\hat{e}$ in $\Sigma_p$.
    By \Cref{L:knead}, the edge $\hat{e}$ is adjacent to two faces, one of which contains a pair $(c, \zeta)$ where $c \in \gamma$ such that either $\psi(c, \zeta) = \nu$
    or  $\psi(c, \zeta) = \overline{\nu}$.

    Let $\hat{F}$ be the face of $\Sigma_p$ containing such $(c, \zeta)$, and assume without loss of generality that $\psi(c, \zeta) = \nu$.

    It suffices to show that the next edge in $F$, let us call it $e'$, and the next edge in $\hat{F}$, let us call it $\hat{e}'$, correspond to the same hyperbolic component.
    Let $\theta_1 < \theta_2$ be the characteristic angles of $e$, and $\theta'_1 < \theta'_2$ be the characteristic angles of $e'$.

    Let $\widetilde{\gamma}$ be the lift of $\gamma$ to $\MC_p(\mathcal{F}_1)$ inside $\hat{F}$.
    The next edge of $\hat{F}$ is given by following $\tilde{\gamma}$ starting at $(c, \zeta)$ until it intersects a lift of an external ray of angle $\varphi$ in $\Theta_{\nu} \cup \Theta_{\overline{\nu}}$.
    There are two cases:

    \begin{enumerate}

        \item the arc of $\widetilde{\gamma}$ in between the lifts of $R_M(\theta_2)$ and $R_M(\theta_1')$ does not intersect any lift of the positive real axis;
              then the point $(c_1, \zeta_1)$ of intersection between $\widetilde{\gamma}$ and the lift of $R_M(\theta_1')$ has itinerary $\nu$, and $\theta_1'$ is the smallest
              angle larger than $\theta_2$ in $\Theta_\nu$. This shows that $\varphi = \theta_1'$.

        \item the arc of $\widetilde{\gamma}$ in between the lifts of $R_M(\theta_2)$ and $R_M(\theta_1')$ intersects a lift of the positive real axis;
              then the point $(c_1, \zeta_1)$ of intersection between $\widetilde{\gamma}$ and the lift of $R_M(\theta_1')$ has itinerary $\overline{\nu}$.
              Then $\theta_2$ is the largest angle in $\Theta_\nu$, and $\theta_1'$ is the smallest angle in $\Theta_{\overline{\nu}}$,
              showing that $\varphi = \theta_1'$.

    \end{enumerate}

    In both cases, $e'$ and $\hat{e}'$ correspond to the same hyperbolic component, hence the two cell complexes are isomorphic.

\end{proof}

\section{The case of $\Per_2(0)$}\label{sec:per2}

\subsection{Setup}
Let us now consider the space $\text{Per}_2(0)$ of quadratic rational maps with a critical cycle of period $2$.
\begin{align*}
    \text{Per}_2(0) & = \{g: \hat{\C}  \rightarrow \hat{\C} \ |\  g \text{ is rational of degree 2 and has a critical 2-cycle }\}/\sim
\end{align*}
where $g \sim g'$ if there exists $\phi \in \text{Aut}(\hat{\C})$ such that $g = \phi \circ g' \circ \phi^{-1}$.
Given $g \in \Per_2(0)$, $g$ is conjugate to either $g_\infty(z)=\frac{1}{z^2}$, or to a rational map of the form
\[g_a(z)=\frac{z^2+a}{1-z^2},\]
where $a\ne -1$.
Note that the critical points of $g_a$ are $0$ and $\infty$, its critical 2-cycle is $g_a(\infty) = -1, g_a(-1) = \infty$ and its free critical value is $a$.

\begin{remark}
    Topologically, $\Per_2(0)$ is a once-punctured sphere, where the point $a = -1$ is a puncture.
    For $t\in \C^*$, the map $h_t(z)=t+\frac{1}{z^2-t^2}$ is conformally conjugate to $g_a$ with $a=t^{-3} -1 $, via the conjugacy
    \[
        -\frac 1 t h_t(z) = g_a\del{-\frac z t}.
    \]
    As $a\ra\infty$, and so $t\ra 0$, note that $h_t$ converges pointwise to the map $g_{\infty}(z)=z^{-2}$, which is also an element of $\Per_2(0)$.
    These coordinates show that $g_\infty$ is an orbifold point of $\Per_2(0)$ with isotropy group $\Z/3\Z$.
    Compare Milnor \cite[Section 9, page 62]{milnor-qrm}.
\end{remark}

Mating is well-studied in the literature and we refer to external sources for additional context (e.g. \cite{wittner, Luo, Dudko}), but roughly we sketch the construction of the mating $f\sqcup g$ of two geometrically finite polynomials $f, g$. 
We say that $z\in J_f$, $w\in J_g$ are in the same \emph{ray equivalence class} if there is some ray of angle $\theta$ landing at $z$ such that the ray of angle $-\theta$ lands at $w$. The \emph{topological mating} of $f$ and $g$ is defined as the quotient $K_f \sqcup K_g/\sim$ by the ray equivalence relation. Due to the existence of multiaccessible points in $J_f$ or $J_g$, which are associated to multiple external angles, determining ray equivalence classes quickly becomes complicated. However, as long as there are no sphere-separating equivalence classes, the quotient is a sphere and the topological mating is combinatorially equivalent to a rational map, which we call the \emph{geometric mating} $f\sqcup g$ of $f$ and $g$.

Let $K_B = K_{f_{-1}}$ be the filled Julia set of the basilica polynomial, and $J_B = J_{f_{-1}}$ be the Julia set.
Let $\mathcal{L}_B \subset \mathcal{M}$ be the basilica limb, and define  $\mathcal{M}^*$ to be the set $\mathcal{M}-\mathcal{L}_B$, and  $K_B^*$ to be the connected component of $K_B-\{\alpha_B\}$ containing the critical point, where by $\alpha_B \in \C$ we denote the $\alpha$-fixed point of the basilica polynomial.

This mating construction, so far only defined between two dynamical systems, can be transferred to parameter space, 
and facilitates a thorough understanding of $\Per_2(0)$ as a \emph{mating} of a partial Mandelbrot set $\mathcal{M}^*$ and a partial basilica $K_B^*$, \cite{wittner, Luo, Dudko}. Here, the process is the same: take two objects carrying external rays ($\M^*$ and $K_B^*$), and quotient their disjoint union by ray equivalence classes. As $\Per_2(0)$ is a parameter space, each point in the induced mating $\M^*\sqcup K_B^*$ corresponds to a particular dynamical system: points in $\M^*$ correspond to particular disjoint type maps (i.e., maps whose critical points have disjoint orbits), points in $\mathring K_B^*$ correspond to particular capture type or bitransitive type maps (i.e. in which one critical point maps to the other).

Care must be taken about local connectivity: $\M$ has not been proven to be locally connected, and there are parameters in $\M$ whose Julia sets are not locally connected. To get around this, we leverage \emph{laminations}, collections of arcs inscribed in $\overline \D$ which define an equivalence relation modelling the topology of the Julia set, see \cite{Dudko} for futher context. $K_B$ has a simple lamination $\cal{L}_B$ (all of whose elements are preimages of the arc $(1/3,2/3)$ in $\overline{\D}$) at which angles $\M$ is locally connected, and we define the mating $\M^*\sqcup K_B^*$ by first embedding $\cal{L}_B$ into $\Per_1(0)-\M$, and then quotienting. For maps in $\M^*$ which are not locally connected, a good account of puzzles in $\Per_2(0)$ still allows us to view them as matings with the basilica, see \cite[Main Theorem]{AY}, \cite[Theorem 1.4]{Dudko}.

Define the \emph{non-escaping set} $\M_2$ as the set of parameters for which the free critical point does not converge to the $2$-cycle $\{ \infty, -1 \}$.

\begin{theorem}[Wittner, Luo, Dudko]\label{thm:per2-structure}
We have homeomorphisms $\Per_2(0)\simeq\mathcal{M}^*\sqcup K_B^*$ and $\mathcal{M}_2\simeq \mathcal{M}^*\sqcup J_B^*$.
\end{theorem}

Note that the above homeomorphisms preserve puzzle pieces, see \cite[Theorem 1.2]{Dudko}. In addition to showing that every hyperbolic PCF map in $\mathcal{M}_2$ is a mating of the basilica with a hyperbolic PCF map $g\in \mathcal{M}^*$, this theorem also shows that the complement of $\mathcal{M}_2$ is the union of infinitely many hyperbolic components with a touching structure isomorphic to $K^*_B$. One application of this is the generalization of \emph{external rays} from $\Per_1(0)$ to \emph{bubble rays} in $\Per_2(0)$.

The idea of bubble rays was introduced by \cite{Luo} in his thesis and developed by many authors since. We now recall some context about bubble rays, in particular following  \cite{timorin}, \cite{AY}, \cite{Dudko}.
\texttt{}
First observe that  when $g_a$ is not in the unbounded hyperbolic component $\cal{E}$ of $\Per_2(0)$, then $J(g_a)$ is not a quasi-circle. In fact, the boundaries of the Fatou components containing the points $\infty$ and $-1$ touch at a single point $\alpha_a$, which is fixed.

Recall that $g_0(z) = \frac{z^2}{1-z^2}$ is conjugate to the basilica quadratic polynomial $f_{-1}(z) = z^2 -1$.
The important fact of each $g_a\not\in \cal{E}$ is that each map $g_a$ has ``basilica dynamics",
as stated in the following proposition. 

\begin{proposition}\label{prop:basilica-dynamics}
    Let $K_\infty(a)$ denote the basin of $\infty$ for $g_a$ together with the grand orbit of $\alpha_a$.
    Then for any $a \notin \cal{E}$, there is a unique continuous semi-conjugacy $\vp_a$ such that the following diagram commutes:
    \[
        \begin{tikzcd}
            {K_\infty(0)} & {} & {K_\infty(a)} \\
            {K_\infty(0)} && {K_\infty(a)}
            \arrow["{\vp_a}", from=1-1, to=1-3]
            \arrow["{g_0}"', from=1-1, to=2-1]
            \arrow["{\vp_a}"', from=2-1, to=2-3]
            \arrow["{g_a}", from=1-3, to=2-3]
        \end{tikzcd}
    \]
    and which is conformal in the interior of $K_\infty(0)$.
\end{proposition}

\begin{proof}
	When $a\in \M_2$, the proposition is clear from the continuity of mating in $\Per_2(0)$, and $\varphi_a$ is in fact a conjugacy.
	Generally, this follows from Proposition 4.2 in \cite{Dudko}. See also Proposition 4.6 in \cite{AY} for a more detailed account of the semi-conjugacy when $a\not\in \M_2$.
\end{proof}

As a consequence, the semi-conjugacy $\varphi_a$ lets us use chains of touching Fatou components to parameterize points on the Julia set of $g_a$. 
Let $\cal{A}_\infty(0)$ denote the immediate basin of $\infty$ for $g_0$ (the reader familiar with the Mandelbrot set would expect $\cal{A}_\infty(0)$ to be the basin of 0 for the basilica, but under our coordinates $g_0$ is a map in $\Per_2(0)$, and $\infty$ is the critical point in the 2-cycle).
The construction of bubble rays originates from the fact that points in $J(g_0)$ can be parameterized either from ``the outside" via external rays or from ``the inside" as a limit of chains of Fatou components emanating from $\cal{A}_\infty(0)$.
\begin{definition}[Bubble Rays]\label{defn:bubble-rays}
	For any $z\in \partial K_\infty(0)$, there is a unique (finite or infinite) sequence of Fatou components $(B_i)_{i\in \cal{I}}\subset K_\infty(0)$ such that
	\begin{enumerate}
	\item [(a)]  $B_0$ is the Fatou component containing $\infty$; 
		\item [(b)] The intersection of $\partial B_i$ and $\partial B_{i+1} $ consists of a unique point $\alpha^i_0$, which is in the grand orbit of $\alpha_0$; and
		\item [(c)] Either $\cal{I} = \{0, \dots, n\}$, 
		in which case $z\in \partial B_n$  and we let $\alpha^n_0 = z$, or $\cal{I} = \mathbb{N}$, 
		in which case the Hausdorff limit of $(B_i)$ is $\{ z\}$.
	\end{enumerate}
	For $\theta \in \mathbb{R}/\mathbb{Z}$, let $z \in  \partial K_\infty(0)$ be the landing point of the external ray of angle $\theta$, 
	and let $(B_i)$ be the sequence constructed above. 
	Each $B_i$ is itself a Fatou component equipped with internal rays, which will comprise the bubble ray $\cal{B}_0(\theta)$, which we can define inductively.
	Start with $b_0$  be the internal ray from $\infty$ to $\alpha^0_0$.
	Then for $i>0$, append to $b_{i-1}$ two internal rays within $B_i$: first the internal ray connecting $\alpha^{i-1}_0$ to the center of $B_i$, and then second the internal ray connecting the center of $B_i$ to $\alpha^{i}_0$.
	Then we define $\cal{B}_0(\theta) := \bigcup_{i \in \mathcal{I}} b_i$ as the bubble ray of angle $\theta$.
\end{definition}

\begin{definition}
For $g_a \in \Per_2(0) \setminus \cal{E}$, we define the \emph{bubble ray of angle $\theta$} for $g_a$ as 
$$\cal{B}_a(\theta) := \varphi_a(\cal{B}_0(\theta)).$$
\end{definition}

One of the essential ingredients in describing the motion of rational external rays within $\Per_1(0)$ was that we understood when they bifurcated, at the iterated preimages of critical points of the Green's function. Specifically, we saw in $\Per_1(0)$, that if $c\in \C\setminus \M$ lies on the parameter ray at angle $\theta$, then the dynamical ray $R_c(\theta)$ passes through the critical value of $f_c$. That is, the bifurcation structure we saw was really a condition on the free critical value $c$, or equivalently its free critical point at 0. Recall in $\Per_2(0)$, the free critical point is the critical point \emph{not} in the 2-cycle.
\begin{lemma}
	A bubble ray $\cal{B}_a(\theta)$ in the dynamical plane of $g_a\in \Per_2(0)\setminus \cal{E}$ bifurcates if and only if $\cal{B}_a(\theta)$ hits the free critical point
	or its iterated preimages. If a finite bubble ray does not bifurcate, it lands.
\end{lemma}
\begin{proof}
	This is \cite[Lemma 4.4]{Dudko}.
\end{proof}

An upshot is that we can extend $\varphi_a$ to the boundary of every Fatou component in $K_\infty(a)$.

The notion of bubble rays is unnecessary for maps $g_a$, $a\in \cal{E}$ as these maps have quasi-circle Julia sets which in particular are locally connected. Moreover, \cite{timorin} shows how the circular notion of angle for maps in $g_a$ is compatible with the bubble ray notion of angle for maps in $\partial \cal{E}$. 

\begin{figure}[t]\label{fig:basilica-lam}
    \begin{subfigure}{0.48\textwidth}
        \centering
        \begin{overpic}[width=\textwidth]{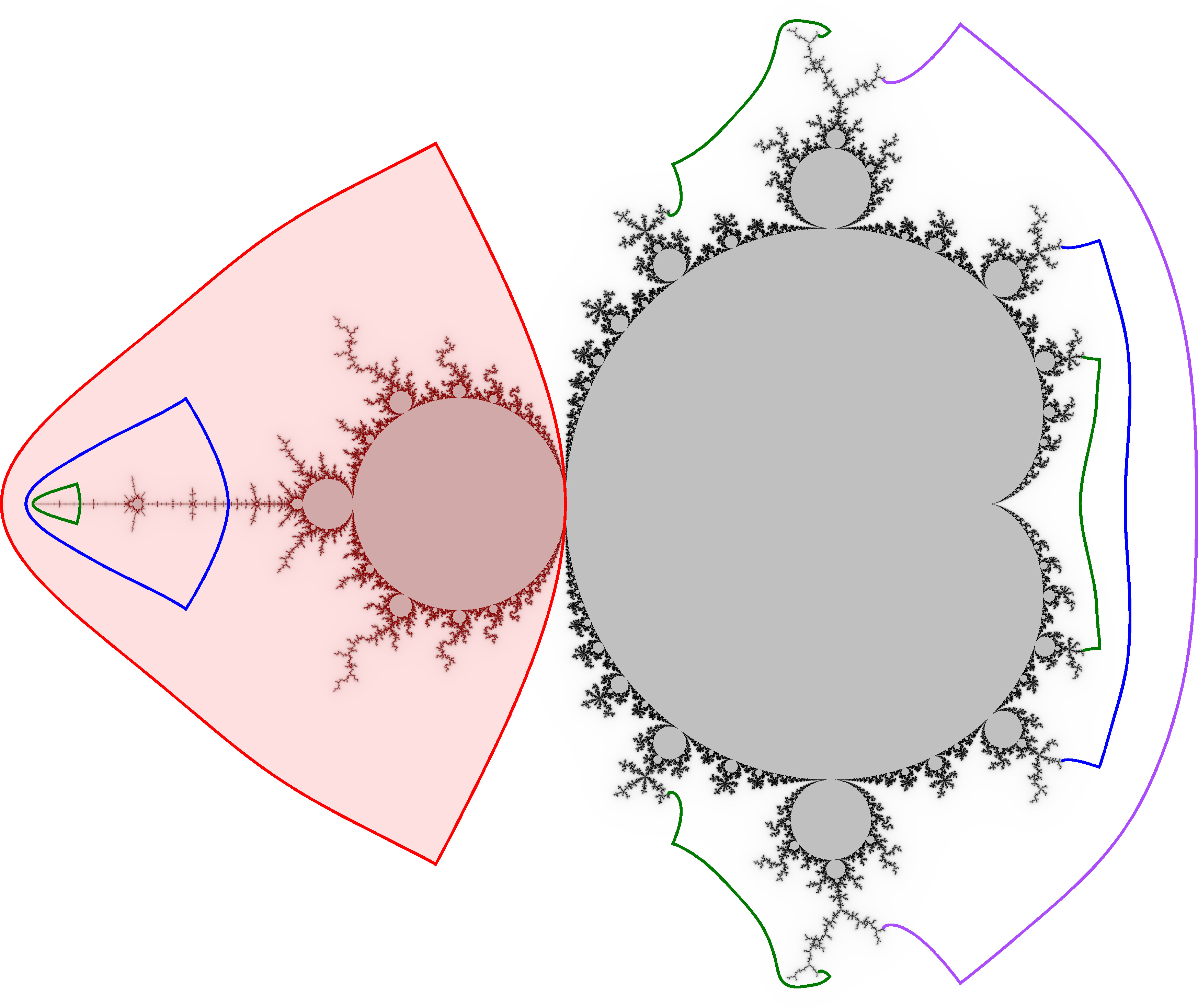}
            \put(11,22){\makebox(0,0){%
                    \contour{black}{%
                        \textcolor{red}{$W_B$}}%
                }
            }
        \end{overpic}
        \caption{%
            $\M$ with the basilica lamination embedded on the outside.
            The shaded region $W_B$ contains the $1/2$-limb.
        }
    \end{subfigure}
    \hfill
    \begin{subfigure}{0.48\textwidth}
        \centering
        \includegraphics[width=0.65\textwidth]{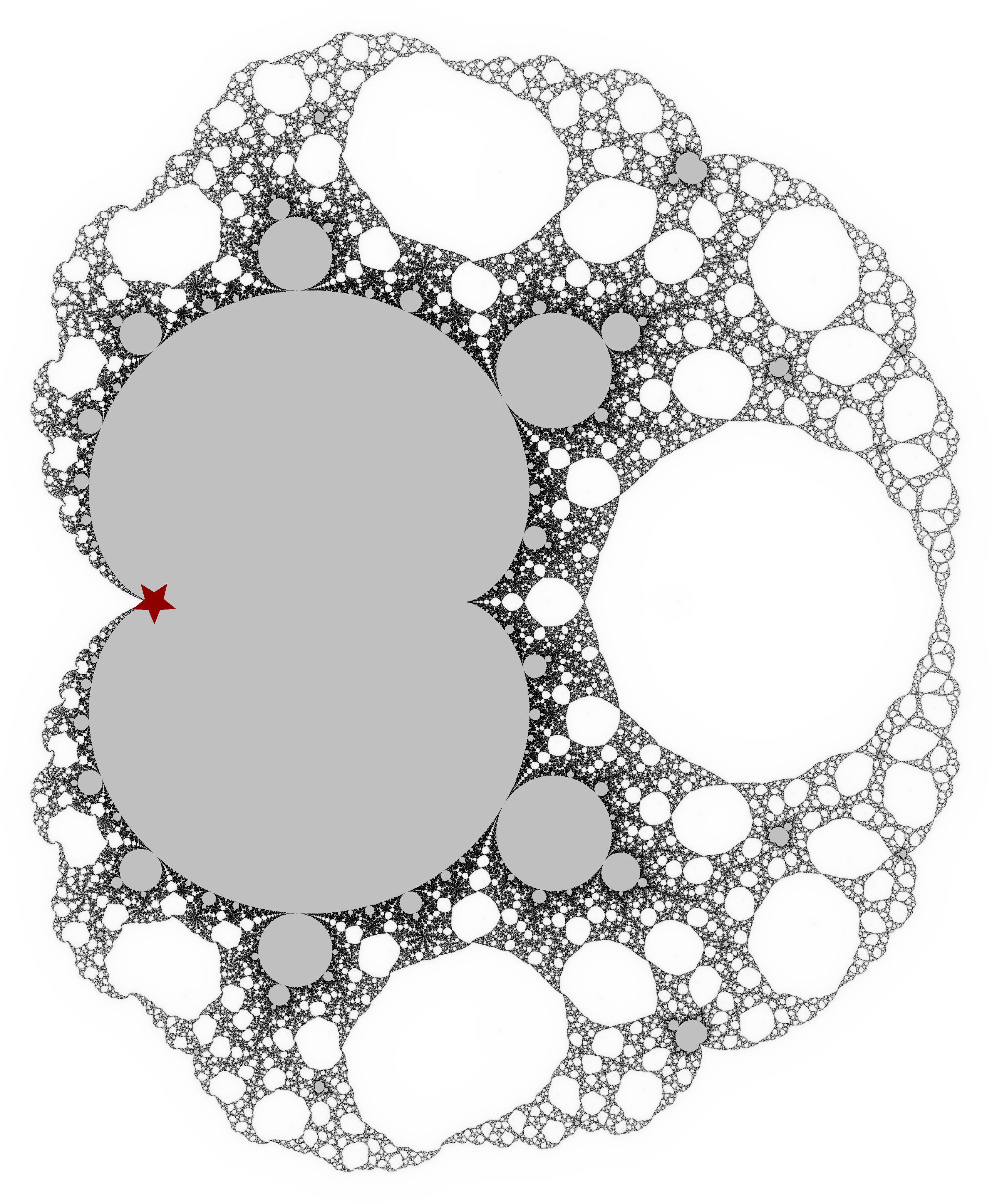}
        \caption{%
            Bifurcation locus in $\Per_2(0)$ in the coordinates $g_a(z)=\frac{z^2+a}{1-z^2}$.
            The puncture, marked with a red star, corresponds to the quotient of the region $W_B$ to a point.
        }
    \end{subfigure}
    \label{fig:per2_lamination}
\end{figure}

Recall from \Cref{thm:per2-structure} that $\Per_2(0)$ is also a ``mating" with the basilica, $\Per_2(0)= \M^*\sqcup K_B^*$, where the capture components in $\Per_2(0)\setminus \M_2$ correspond to bounded Fatou components in $K_B^*$. 
It follows we can analogously define parabubble rays in $\Per_2(0)$, cf \cite{Luo}, \cite[Section 5.3]{Dudko}, \cite{AY}. For an angle $\theta$ landing at $z\in J_B$ with associated bubble sequence $(B_i)$ as in \Cref{defn:bubble-rays}, then there is a corresponding sequence of touching capture components $\cal{H}_{B_i}$ in $\Per_2(0)\setminus \M_2$. The union of the appropriate arcs of $\cal{H}_{B_i}$ comprise $\cal{B}_{\M_2}(\theta)$, the parabubble ray of angle $\theta$.

\begin{proposition}\label{prop:parabubble}
	Every rational parabubble ray lands, and rational rays of odd denominator land at the root of a hyperbolic component in $\M_2$. Conversely, 
   let $g$ be the center of a hyperbolic component in $\M_2$. Then $g=f_{-1}\sqcup f$ for some PCF hyperbolic quadratic polynomial $f\in \M$.
    Moreover, for this $f$, if $f$ has characteristic external rays of angles $\theta_f^\pm$, then $g$ has characteristic bubble rays of angles $-\theta_f^\pm$.
\end{proposition}

Wakes in $\Per_2(0)$ also exist; intuitively, these are the images of wakes in $\M$ under the mating process, but see Section 6 of \cite{AY} for a proper treatment. Hence wakes can also be said to be \emph{active} in the sense of \Cref{defn:active-per1-wake}, recall:

\begin{definition}\label{active-per2-wake}
    A wake $W = \cal{W}(\alpha, \beta)$ is \emph{active} at $\theta\in\Q/\Z$ if the orbit of $\theta$ under angle doubling contains either $\alpha$ or $\beta$.
\end{definition}

\begin{proposition}[\cite{Luo}, Propositions 6.3-6.4]\label{prop:rational_rays_behavior}
If $\theta$ is rational and $g_a \in \Per_2(0)$, then the dynamical bubble ray $\cal{B}_a(\theta)$ either bifurcates or lands. 
\end{proposition}

The parabubble structure lets us state the previous lemma more concretely, in terms of angles. Recall that $a\in \Per_2(0)-\M_2$ lives in a hyperbolic component, which by \Cref{thm:per2-structure}, corresponds to a Fatou component in $\mathring{K}_B^*$. Suppose $\phi(a)=\varphi_a(a)$ is the point in $\mathring{K}_B^*$ corresponding to $a$, in the Fatou component $U$. Then, $\phi(a)$ lies along some internal ray of angle $\eta$. Let $z$ be the point in $\partial U$ which is the end point of the internal ray of angle $\eta$. Let $\theta$ be the external ray landing at $z$. To conclude the parallel constructions to $\Per_1(0)$, we set $\arg_B(a)=\theta$.

\begin{proposition}\label{prop:per2-bifurcation}
Moreover, if $a\in \Per_2(0)\setminus\M_2$ has argument $\theta$, then the dynamical bubble ray of $g_a$ at angle $\theta'$ bifurcates if and only if $2^k\theta'=\theta$ for some $k\in \N^*$. 
\end{proposition}
\begin{proof}
	This follows from Theorem 5.1 in \cite{Dudko}. 
\end{proof}

So, we are in a position to recover the similar lemma from $\Per_1(0)$ about continuity of external rays away from active wakes to continuity of bubble rays away from active wakes.

\begin{definition}
    For $\theta\in \R/\Z$, define
    \[
        \O_{\theta, 2} := \bigcup_{k\ge 0}\overline{ \mathcal{B}_{\cal{M}_2}(2^k\theta)} \subset \Per_2(0).
    \]
    In particular, if $\theta$ is rational, then the union is finite, so since rational rays land, $\O_{\theta, 2}$ has empty interior.
\end{definition}

 For $\theta\in\Q/\Z$, we denote by $\beta_a(\theta)$ the landing point of the dynamical ray $\mathcal{B}_a(\theta)$ (assuming it is well-defined). One version of the following result is given in \cite[Lemma~5.6]{Dudko}, but we produce a proof here for reference.  
 
\begin{lemma}\label{lem:per2-dynam-ray-holomorphic-landing}
    Fix $\theta\in\Q/\Z$. The map $g_a \mapsto \beta_a:=\beta_a(\theta)$ is holomorphic on $ \Per_2(0)\sm\O_{\theta, 2}$.
\end{lemma}
\begin{proof}
It is clear that $\beta_a$ is well-defined for all $g_a \in \Per_2(0)\sm \O_{\theta, 2} $.
Let $W$ be a neighborhood of $a_0$ that is contained in $\Per_2(0) \setminus \O_{\theta, 2}$.  Note that for each $i$, the function $g_a \mapsto \alpha^i_a$ is holomorphic on $W$, and avoids the points $\{\infty, +1, -1\}$. 
	Consider a map $g_{a_0}\in\Per_2(0)\setminus \overline{\cal{E}}$.
	Without loss of generality suppose that $\theta$ corresponds to an infinite bubble ray.
	Then, there is a sequence $\alpha^i_{a_0}$ of preimages of the fixed point $\alpha_{a_0}$ which converges to $\beta_{a_0}$.  
	By Montel's theorem, there exists a subsequence of maps $g_a \mapsto \alpha^{i_k}_a$ that converges locally uniformly to a holomorphic map $G: W \to \widehat{\C}$. On the other hand, for each $g_a \in W$, we have $\alpha^{i_k}_a \to \beta_a$. Combining these statements, we see that $G(g_a) = \beta_a$ for all $a \in W$, finishing the proof.
    \begin{figure}
       \begin{overpic}[width=0.6\textwidth]{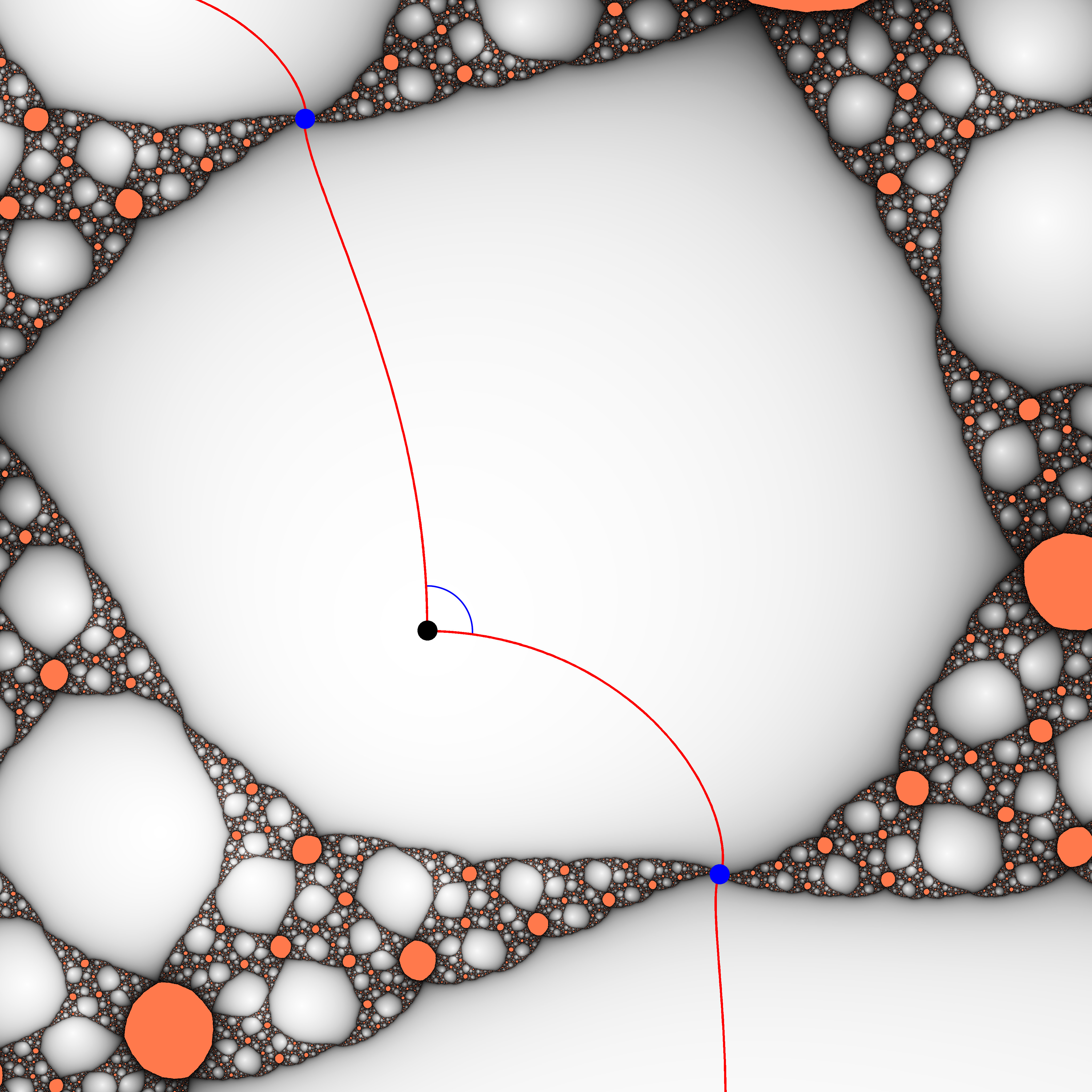}
            \put(61,23){\textcolor{blue}{\(\alpha^2_a\)}}
            \put(42,46){{\textcolor{blue}{\(\hat\theta_3=-\frac{1}{4}\)}}}
            \put(29,85){\textcolor{blue}{\(\alpha^3_a\)}}
        \end{overpic}
        \caption{Segment of the dynamical bubble ray at angle $\theta=11/15$ for $g_a(z)=\frac{z^2+a}{1-z^2}$ with $a\approx 1.052+1.659i$.}
    \end{figure}
\end{proof}

Just as in the case of $\Per_1(0)$, we have the following key corollary:
\begin{corollary}\label{cor:per2-co-landing-implies-primitive}
    Suppose $\theta,\theta'\in\Q/\Z$ are distinct angles with period $p$ under angle doubling, and let $\beta_a(\theta)$, $\beta_a(\theta')$ be as in \Cref{lem:per2-dynam-ray-holomorphic-landing}.
    Suppose further that $\theta$ and $\theta'$ do not share an orbit.
    For any $a\in\C\sm(\cal{O}_{\theta, 2} \cup\cal{O}_{\theta', 2})$, the points $\beta_a(\theta)$ and  $\beta_a(\theta')$ share an orbit if and only if $a$ belongs to a primitive wake that is active at $\theta$ and $\theta'$.
\end{corollary}

\subsection{Underlying combinatorics of the cell structure.}
The goal of this section is to generalize and relate the algorithm from \Cref{S:cell-p1} to construct a cell structure for $\MC_{p}(\mathcal{F}_2)$, the marked cycle curve of period $p$ over $\mathcal{F}_2 = \Per_2(0)$.

The following definition will be justified in \Cref{lem:bub-ray-props}.
\begin{definition}
    A PCF hyperbolic map $g\in \M_2$ is \emph{primitive} if $g=g_0\sqcup f$ for some PCF hyperbolic primitive polynomial $f\in \M$.
\end{definition}

Let $R_{p,2}$ be the set of roots of primitive hyperbolic components in $\M_2$.
For each $a \in R_{p,2}$, let $V_2(a)$ denote the vein from the basilica polynomial $g_0$ to $a$. 
Its existence is nontrivial due to the complicated topology of the non-escaping set in $\cal{M}_2$. 
However, \Cref{thm:per2-structure} plus understanding the lamination of $g_0$ means we can guarantee that the vein $V_2(a)$ does not intersect itself.

\begin{lemma}
    If an angle $\theta \in \mathbb{S}^1 \setminus [1/3, 2/3]$ is biaccessible in the basilica lamination, then it is not biaccessible in $\M$.
\end{lemma}
\begin{proof}
    If $\theta$ is biaccessible in the basilica lamination,
    there exists $n \ge 0$ such that $2^n \theta = 1/3$ or
    $2^n \theta = 2/3$.
    If $\theta$ is also biaccessible in the Mandelbrot set $\M$, then there exists
    a quadratic polynomial $f_c$ and $\theta' \ne \theta$ such that $\theta, \theta'$ both land at the critical value,
    or at the root of the Fatou component containing the critical value.
    
    Now, we have two cases: if $2^n \theta \neq 2^n \theta'$, then $\theta_1 = 2^n \theta \in \{1/3, 2/3\}$ is biaccessible, since the ray at angle $2^n \theta'$ also lands on the same point.
    Since $\theta_1$ is biaccessible and the angle of another ray landing at the same point must have the same period, and the only two angles of period $2$ are $1/3, 2/3$, then the leaf $(1/3, 2/3)$ belongs to the polynomial lamination.
 But the leaf $(1/3, 2/3)$ belongs to a quadratic lamination if and only if the polynomial lives in the $1/2$-limb, which we excluded.
    
    If, on the other hand, $2^n \theta = 2^n \theta'$, then there exists $j \ge 0$ such that $2^j \theta = 2^j \theta' + 1/2$.
    This implies that the landing point of the external ray $R_{c}(2^j \theta)$ must be $0$, the critical point.
    Hence, the critical point lies in the Julia set,
    hence the Julia set is a dendrite and
    the ray pair $(\theta, \theta')$ lands on the critical value $c$. Thus, there exists
    $j$ such that $f^j(c) = 0$, so the critical point is purely periodic. This contradicts the fact that
    $c$ lies in the Julia set, concluding the argument.
\end{proof}

\begin{corollary}\label{cor:veins-embed}
For each $p \geq 3$, $\mathcal{V}_p$ embeds into $\Per_2(0)$. 
\end{corollary}

\noindent Denote the image of this embedding by 
$$\cal{V}_{p,2}:=\bigcup_{a\in R_{p,2}} V_2(a),$$ 
the union of all veins from the basilica component containing $g_0$ to all roots of a primitive component of period $p$.

The complement of $\mathcal{V}_{p,2}$ in $\Per_2(0)$ is open, connected and simply connected.
Similar to the $\Per_1(0)$ case, we can define the tessellation on $\MC_p(\mathcal{F}_2)$ by pulling back a tessellation on $\Per_2(0)$. 
Consider the following tessellation of  $\Per_2(0)$: 

\begin{enumerate}
    \item [(a)] One vertex for $g_0$;
    \item [(b)] One edge for each vein $V_2(a)$;
    \item [(c)] One face containing the external component $\mathcal{E}$ of $\Per_2(0)\sm\M_2$.
\end{enumerate}

Let $\pi_2: \MC_p(\mathcal{F}_2)\to \mathcal{F}_2$ be the covering map forgetting the cycle data.
For each vein $V_2(a)$, the set $\pi_2^{-1}(V_2(a))$ has several connected components. Let us call a component \emph{branched} if that component contains a branch point for $\pi_2$.

\begin{definition}
    Consider $p\geq 1, p\neq 2$. The cell complex $\Sigma_{p, 2}$ is defined as follows:
    \begin{enumerate}
        \item [(a)] Vertices at $\pi_2^{-1}(g_0)$;
        \item [(b)] One edge for each branched lift of a vein in $V_2(a)$ with $a\in R_{p,2}$;
        \item [(c)] Faces are closures of connected components of $\pi_2^{-1}(\mathcal{E})$.
    \end{enumerate}
\end{definition}

In the next sections we will see how to enumerate all $0$, $1$, and $2$-cells of the complex  $\Sigma_{p, 2}$.

\subsubsection{Ternary cycle classes}
As in the previous section, the combinatorics of a shift map plays a central role.
Whereas $f_0(z)=z^2$ underpinned the combinatorics of $\Per_1(0)$,
the special map in $\Per_2(0)$ is the map specified by the orbifold point at infinity, $z\mapsto z^{-2}$.

Let $\sigma_2: \set{0,1,2}^\N\to \set{0,1,2}^\N$ be the left-shift map.

\begin{definition}
    An \emph{abstract ternary $p$-cycle} is an element $\xi$ of $\set{0,1,2}^\N / \sigma_2$ with exact period $p$;
    that is to say, satisfying $\#\xi = p$.

    An abstract ternary $p$-cycle $\xi=(x_0,x_1,\dots)$ is \emph{admissible} if some (or equivalently every) element $\mathbf{x}=(x_0,x_1,\dots)\in\xi$ satisfies $x_i \ne x_{i+1}$ for all $i$ modulo $n$.

    We denote by $C_{p,2}$ the set of admissible abstract ternary $p$-cycles.
\end{definition}

For instance, for $p=4$ there are exactly three admissible abstract ternary cycles:
\begin{align*}
	0102 \hspace{2cm} 1210 \hspace{2cm} 2021.
\end{align*}

\begin{proposition}\label{prop:anti-doubling-coding}
    $C_{p,2}$ is in canonical bijection with the set of $p$-cycles of the anti-doubling map $-2:  \mathbb{S}^1 \ra \mathbb{S}^1 $ given by
    \begin{align*}
        \theta  & \mapsto-2\theta \mod 1,
    \end{align*}
    or equivalently, the set of $p$-cycles of $z\mapsto z^{-2}$ excluding $\set{0,\infty}$.
\end{proposition}
\begin{proof}
  Let $\theta$ be an angle that is $p-$periodic under the anti-doubling map. 
  For $m \in \{0,1,2\}$, let $I^{(m)} := (\frac{m}{3}, \frac{m+1}{3})$. Note that for each $m$, $(-2)\overline{I^{(m)}} = \mathbb{S}^1 \setminus I^{(m)}$. 
  
    The angle $\theta$  can be associated uniquely with the sequence of bits $b_0b_1b_2\cdots$ with $b_j\in\{0,1,2\}$ for $j=0,1,2,3,\cdots$, with $(-2)^j\theta \in I^{(b_j)}$. Note that $b_0b_1b_2\cdots \in C_{p,2}$ since $(-2)I^{(m)} \cap I^{(m)} = \emptyset$ for all $m$. 
    
Conversely, let $b_0b_1b_2\cdots b_{p-1} \in C_{p,2}$. Then for each $j \in \{0,1,\cdots,p-1\}$, there exists a map $\varphi_j: \overline{I^{(b_j)}} \to \overline{I^{(b_{j-1})}}$  such that $\varphi_j (-2x) = x$ for all $x \in \overline{I^{(b_{j-1})}}$.  Letting $\varphi = \varphi_1 \circ \varphi_2 \circ \cdots \circ \varphi_p$ and $W = \varphi(\overline{I^{(b_j)}})$, we see that $W$ is a closed interval with the property that $\varphi(W) \subsetneq W$. 
This shows that there exists a unique $ \theta \in W$ such that $(-2)^p\theta = \theta$. This $\theta$ is associated with the sequence $b_0b_1b_2\cdots b_{p-1}$. 
\end{proof}

Given a cycle $x$, we denote its \emph{rotation} $\rho(x)$ to be the cycle obtained from the expansion of $x$ by sending each element $x_i$ to $x_i + 1\pmod{3}$.

\begin{definition}\label{def:cycle_trio}
    Let us define an equivalence relation on $C_{p,2}$ by setting $x\sim y$ if $x= \rho^i(y)$
    for $i = 0, 1, 2$.
    Then, we define the set of \emph{abstract cycle trios} of period $p$ as
    \[ [C_{p,2}] := C_{p,2}/\sim. \]
    We denote as $\sbr{\alpha}$ the class of the cycle $\alpha \in C_{p,2}$.
\end{definition}

A cycle trio is \emph{rotationally invariant}, or \emph{unramified}, if its equivalence class contains one element, and \emph{ramified} if its equivalence class contains three elements.

\medskip
\begin{example}
    In period $p=4$, there is only one cycle trio, namely
    \[
        [A] = \set{0102, 1210, 2021}.
    \]
    For $p=5$, there are two cycle trios:
    \begin{align*}
        \sbr{01021} = \set{01021, 12102, 20210}, \\
        \sbr{02012} = \set{02012, 10120, 20201}.
    \end{align*}
\end{example}
The first unramified cycle trio for $p > 3$ appears in $p=9$, namely $\sbr{010121202}$.
Note that if 3 does not divide $p$, then all cycles are ramified, hence $\#\sbr{C_{p,2}} = \#C_{p,2}/3$.

\subsection{The cell complex structure of $\MC_p(\mathcal{F}_2)$}

We apply the previous subsection to defining the components of our cell complex.

\subsubsection{Vertices}
Recall that over $\Per_1(0)$, vertices of $\Sigma_p$ correspond to lifts of the origin parameter $z^2$. Over $\Per_2(0)$, the basilica polynomial $g_0$ is the analogue of $z^2$, so that vertices of $\Sigma_{p, 2}$ are lifts of the basilica polynomial in $\Per_2(0)$, equivalence classes of the form $[(g_0,\zeta)]$.
Let 
$$\Sigma_{adm} := \{ \omega \in \{0, 1, 2 \}^\mathbb{N} \ : \ \omega_i \neq \omega_{i+1} \textup{ for all }i \}$$ 
be the space of infinite admissible ternary sequences, and let $\sigma : \Sigma_{adm} \to \Sigma_{adm}$ be the left shift.  
It is immediate to see that elements of $C_{p, 2}$ are in bijective correspondence with cycles of period $p$ for $(\Sigma_{adm} , \sigma)$. 

To state the next lemma, let $J_{g_0}^{PF} := \{ z \in J_{g_0} \ : \ \exists n \geq 0 \ g_0^{\circ n}(z) = g_0^{\circ n+1}(z) \}$ be the set of prefixed points for $g_0$. 
Given angles $\alpha, \beta$, let $R(\alpha, \beta) \subseteq J_{g_0}$ be the set of landing points of  external rays of angles $\theta$ for $\alpha < \theta < \beta$. 

\begin{lemma} \label{lem:semi-conj}
There is a continuous semiconjugacy 
$$\Phi: (\Sigma_{adm}, \sigma) \to (J_{g_0}, g_0).$$
For any $p \geq 3$, this semiconjugacy induces a bijective correspondence between the set 
$C_{p, 2}$ and the set of cycles of period $p$ for $g_0$. 

Then the inverse map to $\Phi$, denoted as 
$$\psi: (J_{g_0} \setminus J_{g_0}^{PF},  g_0) \to (\Sigma_{adm}, \sigma),$$
is given by taking the symbolic coding with respect to the following partition. 
Let $a_n := \frac{1}{3 \cdot 2^n}$, $\overline{a}_n := 1 - \frac{1}{3 \cdot 2^n}$ for $n \geq 0$.
Set 
\begin{align*} 
J^{(0)} & := \bigcup_{n \geq 0} R(a_{2n +1}, a_{2n}) \cup R(\overline{a}_{2n+1}, \overline{a}_{2n + 2}) \\
J^{(1)} & := R(a_0, \overline{a}_0) \\ 
J^{(2)} & := \bigcup_{n \geq 0} R(a_{2n +2}, a_{2n+1}) \cup R(\overline{a}_{2n}, \overline{a}_{2n+1})
\end{align*}
This partition of $J_{g_0}$ is shown in \Cref{fig:bas-codeshop}. Then the map is $\psi(z) := (\omega_i)_{i \geq 0}$ where $\omega_i = k$ if $g_0^{\circ i}(z) \in J^{(k)}$.
\end{lemma}

\begin{figure}
\includegraphics[width=\textwidth]{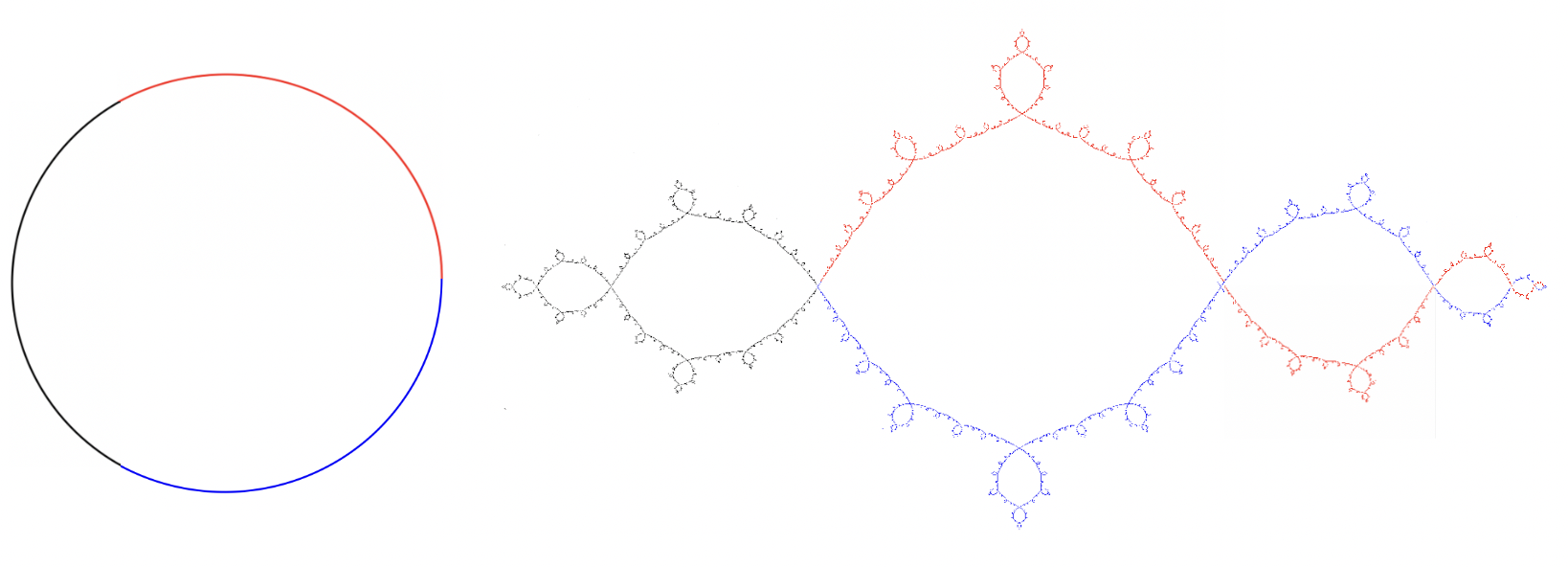}
\caption{ Left: the partition $\mathcal{I} = \{I^{(k)}\}$ of the circle used to encode $p$-cycles for $z^{-2}$ with $C_{p, 2}$ in \Cref{prop:anti-doubling-coding}. 
Right: the partition $\mathcal{J} = \{ J^{(k)} \}$ of the basilica Julia set described in \Cref{lem:semi-conj}, used to encode $p$-cycles with $C_{p,2}$.
By moving in parameter space across the puncture $a = -1$ along the negative real axis, the quasicircle Julia set gets ``twisted" into a Julia set homeomorphic to the basilica, and the partition $\mathcal{I}$ transforms into the partition $\mathcal{J}$. } 
 \label{fig:bas-codeshop}
\end{figure}

To prove \Cref{lem:semi-conj}, consider the basilica lamination $\mathcal{L}$, and let $J := \mathbb{S}^1/\mathcal{L}$ be 
the quotient, homeomorphic to the basilica Julia set. 
The doubling map $x \mapsto 2 x$ on $\mathbb{S}^1$ descends to a map $g$ on $J$.
Moreover, since the lamination $\mathcal{L}$ is invariant under the reflection $x \mapsto -x$, the anti-doubling map $x \mapsto - 2 x$ descends to an (orientation reversing) map $-g$ on $J$, 
that we call the \emph{anti-basilica}. 

\begin{lemma} \label{lem:twist}
There is a topological conjugacy $\varphi : (J, g) \to (J, -g)$.
\end{lemma}

The idea is to ``twist" the basilica Julia set at each cut point, i.e. at each point in the grand orbit of the $\alpha$-fixed point. 

\begin{proof}
For each leaf $\alpha$ of $\mathcal{L}$, let $I_\alpha$ be the shorter arc of the circle delimited by $\alpha$. Let $x \in J$ be the equivalence class of the end-points of $\alpha$. For any map $h : J \longrightarrow J$, we let $h(\alpha)$ denote the unique leaf of $\mathcal{L}$ whose end-points have equivalence class $h(x) \in J$.   
Define for each $\alpha$ the map $\varphi_\alpha : J \to J$ which flips the interval $I_\alpha$ and leaves everything else fixed. 

For each point $x \in \mathbb{S}^1$, let $(\alpha_n(x))_{n \leq N(x)}$ be the set of leaves that separate $x$ from the center of the disk, ordered so that 
$$I_{\alpha_1(x)} \supseteq I_{\alpha_2(x)} \supseteq \dots \supseteq I_{\alpha_n(x)} \dots$$
and 
$$x \in \bigcap_{n \leq N(x)} I_{\alpha_n(x)}$$
Note that $N(x)$ may be either finite or $\infty$, and in the latter case we have 
$\{ x \} = \bigcap_{n \leq N(x)} I_{\alpha_n(x)}$.

Let us set 
$$\widetilde{\alpha}_1(x) := \alpha_1(x), \qquad \widetilde{\alpha}_n(x) := \varphi_{\widetilde{\alpha}_{n-1}(x)} \circ \dots \circ \varphi_{\widetilde{\alpha}_1(x)}(\alpha_n(x))$$ 
Note also that $\textup{diam }I_{\widetilde{\alpha}_n(x)} \leq 2^{-n}$ for any $n \leq N$; 
so we can define the map
$$\varphi(x) := \lim_{n} \varphi_{\widetilde{\alpha}_n(x)} \circ \varphi_{\widetilde{\alpha}_{n-1}(x)} \circ \dots \circ \varphi_{\widetilde{\alpha}_1(x)}(x)$$
We claim that $\varphi : J \to J$ is a homeomorphism, and it conjugates $g$ to $-g$.
That is, we want to prove that 
\begin{equation} \label{E:conj}
\varphi(g(x)) = - g(\varphi(x)) \qquad \textup{for any }x \in J.
\end{equation} 
Let us observe that, for any leaf $\alpha \in \mathcal{L}$, we have 
$g \circ \varphi_\alpha = \varphi_{g(\alpha)} \circ g$
and also 
\begin{equation} \label{E:sym}
- g \circ \varphi_\alpha = \varphi_{-g(\alpha)} \circ (-g).
\end{equation}

 Let $A_0 :=\left(\frac{1}{6}, \frac{1}{3} \right) \cup \left( \frac{2}{3}, \frac{5}{6} \right)$, $A_1 := \left(\frac{1}{3}, \frac{2}{3} \right)$ and $A_2 := \left( \frac{5}{6}, \frac{1}{6} \right)$, so that $g(A_0) = A_1$. Denote as $\alpha_0$ the leaf that delimits $\left( \frac{1}{3}, \frac{2}{3} \right)$. 

There are three cases. Suppose that $x \in A_0$, and its associated nested sequence of leaves is $\alpha_1(x) < \alpha_2(x) < \dots < \alpha_n(x) < \dots$. 
Then, $g(x) \in A_1$, and its associated nested sequence of leaves is 
\begin{align*}
\alpha_1(g(x)) & = \alpha_0, \qquad \alpha_{n+1}(g(x)) = g(\alpha_n(x)) \qquad \textup{for any }n \geq 1.
\end{align*}
We claim that 
\begin{equation}  \label{E:claim}
\widetilde{\alpha}_{1}(g(x)) =  \alpha_0, \qquad \widetilde{\alpha}_{n+1}(g(x)) =  - g (\widetilde{\alpha}_n(x))
\qquad \textup{for any }n \geq 1
\end{equation}
To prove the claim, let us proceed by induction. By definition, 
\begin{align*} 
\widetilde{\alpha}_{n+1}(g(x)) & = \varphi_{\widetilde{\alpha}_n(g(x))} \circ \dots \circ \varphi_{\widetilde{\alpha}_1(g(x))} (\alpha_{n+1}(g(x))) \\
\intertext{and by inductive hypothesis} 
 & = \varphi_{- g (\widetilde{\alpha}_{n-1}(x))} \circ \dots \circ \varphi_{-g (\widetilde{\alpha}_1(x))} \circ \varphi_{\alpha_0} (g (\alpha_n(x)))\\
\intertext{and, noting that, if $x \in A_0$, we have $I_{\alpha_n(x)} \subset A_0$, so $I_{g( \alpha_n(x)) }\subset A_1$ and $\varphi_{\alpha_0}(g( \alpha_n(x))) = - g (\alpha_n(x))$,} 
& = \varphi_{- g (\widetilde{\alpha}_{n-1}(x))} \circ \dots \circ \varphi_{-g (\widetilde{\alpha}_1(x))} (-g(\alpha_n(x)))\\
\intertext{and using \eqref{E:sym},}
&  = - g \circ \varphi_{ \widetilde{\alpha}_{n-1}(x)} \circ \dots \circ \varphi_{ \widetilde{\alpha}_1(x)} (\alpha_n(x)) = -g (\widetilde{\alpha}_n(x))
\end{align*}
proving the claim \eqref{E:claim}. 

\noindent Now, to prove that $\varphi$ is a conjugacy, we write
\begin{align*}
\varphi(g(x)) & = \lim_{n} \varphi_{\widetilde{\alpha}_{n+1}(g(x))} \circ \dots \circ \varphi_{\widetilde{\alpha}_2(g(x))} \circ \varphi_{\widetilde{\alpha}_1(g(x))}(g(x)) \\
\intertext{and,  since if $x \in A_0$, we have $g(x) \in A_1$ and $\varphi_{\alpha_0}(g(x)) = - g(x)$, so using \eqref{E:claim},} 
& = \lim_{n} \varphi_{-g (\widetilde{\alpha}_n(x))} \circ \dots \circ \varphi_{-g (\widetilde{\alpha}_1(x))}(-g(x)) \\
\intertext{and then using \eqref{E:sym},}
& = \lim_{n} - g \circ \varphi_{ \widetilde{\alpha}_n(x)} \circ \dots \circ \varphi_{ \widetilde{\alpha}_1(x)}(x) \\
& =  - g \circ \left( \lim_{n} \varphi_{ \widetilde{\alpha}_n(x)} \circ \dots \circ \varphi_{ \widetilde{\alpha}_1(x)}(x)\right) =  - g(\varphi(x))
\end{align*}
proving \eqref{E:conj} in the case $x \in A_0$. 
The cases of $x \in A_1$ and $x \in A_2$ are analogous, so we will not write out all the details. 
\end{proof}

As a consequence of the above semiconjugacy, we obtain a way to parameterize periodic cycles for the basilica 
using admissible ternary sequences $C_{p, 2}$. 

\begin{proof}[Proof of \Cref{lem:semi-conj}]
Recall that considering the itinerary for the partition $I^{(0)} = \left( 0, \frac{1}{3} \right)$, $I^{(1)} = \left(  \frac{1}{3}, \frac{2}{3} \right)$, and $I^{(2)} = \left(  \frac{2}{3}, 1 \right)$
of the unit circle, one obtains a continuous semiconjugacy 
$$(\Sigma_{adm}, \sigma) \to (\mathbb{S}^1, z^{-2}).$$
Moreover, as $(J, -g)$ is a quotient of $(\mathbb{S}^1, z^{-2})$ and $(J, g)$ is conjugate to $(J_{g_0}, g_0)$, we obtain, by composing with the conjugacy $\varphi$ from \Cref{lem:twist}, a semiconjugacy 
$$(\Sigma_{adm}, \sigma) \to (\mathbb{S}^1, z^{-2}) \to (J, -g) \overset{\varphi^{-1}}{\to} (J, g) \to (J_{g_0}, g_0),$$
proving the first part of the lemma. 
Now, the inverse map is defined by taking the itinerary under the basilica map $g_0$ with respect to the partition 
which is the image of the partition $\mathcal{I} = \{I^{(k)}\}_{0 \leq k \leq 2}$ under the semiconjugacy $(\mathbb{S}^1, z^{-2}) \to (J_{g_0}, g_0)$.
This is precisely the partition $\mathcal{J} = \{J^{(k)}\}_{0 \leq k \leq 2}$ shown in  \Cref{fig:bas-codeshop}.
\end{proof}

The above coding comes from the fact that the hyperbolic components containing $z^{-2}$ and the basilica are adjacent in parameter space, both containing the (unique) puncture of $\Per_2(0)$ on their boundary. Tracking how the Julia set varies while passing from the external component containing $z^{-2}$ into the component containing the basilica appears to change the \Cref{prop:anti-doubling-coding} coding of the circle to the one from \Cref{lem:semi-conj} according to \Cref{fig:bas-codeshop}: $I^{(1)}$ becomes a twisted interval, and $I^{(0)}$ and $I^{(2)}$ alternatingly intertwine.

In precise terms, consider the negative arc $(-\infty, 0]$ of the real line in $\Per_2(0)$. 
For each $a < -1$, the Julia set $J_{g_a}$ is a quasi-circle and $g_a$ has three fixed points, that disconnect the Julia set into three arcs. 
Let us label these three arcs with $0, 1, 2$ in counterclockwise order, and so that the real fixed point separates the arcs labeled by $0$ and $2$. 
Let us denote this partition as $\mathcal{I}_a := \{ I_a^{(k)}, k = 0, 1,2 \}$. By using this partition, we label each periodic point for $g_a$ 
with an element of $C_{p, 2}$. 
By the implicit function theorem, each such periodic point of period $p \geq 3$ can be analytically continued as the parameter $a$ moves along the arc $(-\infty, 0]$, obtaining  a periodic point for $g_0$. 
Moreover, by \Cref{lem:semi-conj} the periodic cycles for $g_0$ are also encoded by $C_{p, 2}$, by considering the partition $\mathcal{J} = \{ J^{(k)},  k = 0, 1, 2 \}$. 
We conjecture: 

\begin{conjecture} \label{C:continuation}
Given a periodic cycle of period $p \geq 3$ for $g_a$ with $a < -1$ with itinerary $\omega \in C_{p, 2}$ with respect to the partition 
$\mathcal{I}_a = \{ I_a^{(k)}\}$, by analytically continuing the cycle across the arc $[a, 0]$ in parameter space, we obtain a periodic cycle of period $p$ 
for $g_0$ and its itinerary with respect to the partition $\mathcal{J} = \{ J^{(k)} \}$ equals $\omega$. 
\end{conjecture}

\subsubsection{Faces}
Faces are lifts of the external component $\cal{E}$ of $\Per_2(0)-\mathcal{M}_2$.  Let us consider a new parameterization of $\Per_2(0)$ centered at $\cal{E}$. 

Let $\D_r$ be the disk of radius $r>0$ centered at the origin, and for $\epsilon \in \D_r$ define: 
\[p_{\epsilon}(z) = \frac{\epsilon z^2+1}{z^2 - \epsilon^2}.\]
The points $0$, $\infty$ are the critical points of this map, and $\infty \mapsto \epsilon \mapsto \infty$. So $[p_{\epsilon}] \in \Per_2(0)$. 
Let $\omega$ be a third root of unity.

\begin{lemma}\label{lem:symmetry-of-new-param}
	The $p_\e(z)$ parameterization of $\Per_2(0)$ has a $\Z_3$-symmetry in the sense
	$[p_{\epsilon}] = [p_{\epsilon'}]$ if and only if $\epsilon' = \omega^j \epsilon$ for $j\in \{0,1,2\}$.
\end{lemma}
\begin{proof}
	Necessity follows from the straightforward identity $p_{\omega^2 \epsilon}(z) = \frac{1}{\omega} p_{\epsilon}(\omega z)$.
	
	For the other direction, assume that $p_{\epsilon'}=p_\epsilon$ in $\Per_2(0)$. That is, there exists $\psi \in \text{Aut}(\hat{\C})$ such that $p_{\epsilon'} = \psi \circ p_\epsilon \circ \psi^{-1}$. Since $0,\infty$ are critical points, $\psi$ must fix them setwise. There are two possibilities. 
	
	First, if $\psi$ interchanges $0$ and $\infty$, then $\psi(z) = \frac{a}{z}$ for some $a \ne 0$, and $\psi^{-1}(z) = \frac{a}{z}$. Simple  computation shows
	\[\frac{\epsilon'z^2+1}{z^2 - (\epsilon')^2}  = \frac{a((\frac{a}{z})^2 - \epsilon^2)}{\epsilon (\frac{a}{z})^2 +1} = \frac{- a\epsilon^2z^2 + a^3 }{z^2 + a^2\epsilon }\]
	and by comparing coefficients,
	$\epsilon'  = -a\epsilon^2$.
	Since $\psi = \psi^{-1}$, by symmetry we also have
	$\epsilon  = -a(\epsilon')^2$
	hence
	$ \epsilon  = -a^3 \epsilon^4$.
	This implies $\epsilon = 0$ or $a = \frac{-\omega^j}{ \epsilon}$ for some $j$, implying $\epsilon' = \omega^j\epsilon$.
	
	Otherwise, $\psi$ must fix both $0$ and $\infty$. Then, $\psi(z) =az$ for some $a \ne 0$, and $\psi^{-1}(z) = \frac{z}{a}$. Hence we have the following computation,
	\[
	\frac{\epsilon'z^2+1}{z^2 - (\epsilon')^2}  = \frac{a(\epsilon(\frac{z}{a})^2 + 1)}{(\frac{z}{a})^2  - \epsilon^2} = \frac{a\epsilon z^2 + a^3}{z^2 - a^2\epsilon^2}.
	\]
	By comparing coefficients, we have both
	\begin{align*}
		\epsilon'    & = a\epsilon     \\
		a^3 & = 1 
	\end{align*}
	so that $a = \omega^j$ and $\epsilon' = \omega^j\epsilon$.
\end{proof}

The next lemma says that faces are labelled by $[C_{p,2}]$, trios of conjugate period $p$ cycles:
\begin{lemma}\label{lem:per2-faces}
	The lifts of the external complement $\cal{E}$ of the non-escaping set $\M_2$ under the branched cover $\pi_2:\MC_p(\mathcal{F}_2) \to \cal{F}_2$ are in bijective correspondence with the set of cycle trios $[C_{p,2}]$. Moreover:
	\begin{itemize}
		\item if the cycle trio is ramified, then the restriction of $\pi_2$ to the associated disc is a homeomorphism;
		\item if the cycle trio is unramified, the restriction of $\pi_2$ to the associated disc has local degree 3, branched at the unique lift of $z^{-2}$.
	\end{itemize}
\end{lemma}

A way to think about this correspondence is the following. 
For each $a \in \mathcal{E}$ the Julia set $J_{g_a}$ is a quasi-circle, and $g_a$ has three fixed points on it, 
that disconnect the Julia set into three arcs. By labeling the three complementary arcs with $0, 1, 2$ in counterclockwise order 
and taking the symbolic itinerary with respect to this partition and then its associated cycle trio, we obtain a map from
$\MC_p(\mathcal{E}) := \pi_2^{-1}(\mathcal{E}) \subseteq \MC_p(\mathcal{F}_2)$ that we denote
\begin{equation} \label{E:face}
\Phi_\mathcal{E}: \MC_p(\mathcal{E})\to [C_{p, 2}]
\end{equation}
and is locally constant on each face. Note that the choice of labeling the arcs with $0, 1$, or $2$ is not canonical, but different labelings 
(as long as they are in counterclockwise order, that is indeed canonical since $\Per_2(0)$ is oriented) give rise 
to the same class in $[C_{p, 2}]$.

\begin{proof}	
	Recall that a cycle trio $[\alpha]\in[C_{p,2}]$ is $\alpha\in C_{p,2}$ mod the action of $R(x) := x+\frac{2}{3} (\text{mod 1})$.
	
	Let us consider the parameterization $p_\e(z)$.
	As $\epsilon \rightarrow 0$, $p_{\epsilon}(z) \rightarrow p_0(z) =  \frac{1}{z^2}$.
	For $r$ sufficiently small, $\{[p_{\epsilon}]: \epsilon \in \D_r\}$ is an open neighborhood $U_r$ of $[\frac{1}{z^2}]$ entirely contained in $\cal{E}$. 
	By \Cref{lem:symmetry-of-new-param}, the function $\Omega: \D_r \rightarrow U_r$ mapping $\epsilon$ to $[p_{\epsilon}]$ is a degree $3$ branched cover ramified at $\epsilon = 0$.
	
	Let $U$ be a connected component of $\pi_2^{-1}(\cal{E})$.
	The map $\pi_2|U : U \rightarrow \cal{E}$ is a cover, possibly branched, with the only possible branch point being $[\frac{1}{z^2}]$.

	We now construct a branched cover $\Psi: \D_r \rightarrow V_r$, where $V_r$ is an open subset of $U$ such that $(\pi_2|V_r) \circ \Psi = \Omega$. 
	
	Let $[(\frac{1}{z^2},\zeta_0)] \in U$ be the pre-image of $[\frac{1}{z^2}]$ under $\pi_2$.
	Since, $p\ge 3$, any $p-$cycle for $p_0(z) = \frac{1}{z^2}$ is repelling, it can be perturbed to give $p-$cycles $\zeta_\epsilon$ for maps $p_\epsilon$ close to $p_0$.
	In other words, if $r$ is chosen to be small, then for any $\epsilon \in \D_r$, there exists a cycle $\zeta_\epsilon$ such that $[(p_\epsilon,\zeta_\epsilon)] \in U$.
	
	Define $\Psi(\epsilon) = [(p_\epsilon, \zeta_\epsilon)]$. Observe that $\Psi$ is unique up to pre-composition by $\omega$. It is easy to see that $\Psi$ is a branched cover: for any $\epsilon_0 \ne 0$ and point $[(p_{\epsilon_0}, \zeta_{\epsilon_0})] \in V_r$, the map $[(p_\epsilon
	, \zeta_\epsilon)] \rightarrow \epsilon$ is locally well-defined, and an inverse to $\Psi$.
	By \Cref*{lem:symmetry-of-new-param} if $\Psi(\epsilon) = \Psi(\epsilon')$ for  distinct non-zero $\epsilon,\epsilon'$, then we must have $\epsilon = \omega^j\epsilon'$ for some $j \in \{1,2\}$.
	That is, the deck transformation group of $\Psi|\D_r^*$ is either trivial or the cyclic group generated by $\epsilon \mapsto \omega \epsilon$, and we have $\deg\Psi \in \{1, 3\}$. 
	
	The degree of $\Psi$ is determined by $\zeta_0$: $\Psi$ has degree 3 if and only if $\zeta_0 = \omega^2\zeta_0$. Indeed for  $\Psi(\epsilon) = [(p_\epsilon, \zeta_\epsilon)]$, we observe
	\begin{align*}
		p_{\omega^2 \epsilon}(\omega^2 z) = \frac{1}{\omega}p_{\epsilon}( z)  = \omega^2p_{\epsilon}(z).
	\end{align*}
	That is, $[(p_\epsilon, \zeta_\epsilon)] = [(p_{\omega^2 \epsilon}, \omega^2 \zeta_\epsilon)]$ for all $\epsilon \in \D_r$.
	
	If $\zeta_0 \ne \omega^2 \zeta_0$ and $\Psi$ has degree 3, then $\Psi(\omega^2 \epsilon) = \Psi(\epsilon)$ for all $\epsilon$.
	This implies that $\omega^2 \zeta_\epsilon =  \zeta_{\omega^2\epsilon}$ for all $\epsilon \ne 0$.
	As we let $\epsilon \rightarrow 0$, this gives a contradiction.
	In other words, if $\zeta_0 \ne \omega^2 \zeta_0$, $\Psi$ must have degree 1.
	On the other hand, if $\zeta_0 = \omega^2\zeta_0$, since $\zeta_\epsilon$ is defined to be the unique analytic continuation of $\zeta_0$, we must have $\omega^2 \zeta_\epsilon = \zeta_{\omega^2\epsilon}$.
	Thus, $\Psi(\omega^2\epsilon) = \Psi(\epsilon)$, and $\deg\Psi = 3$.
\end{proof}

For example, for $p=5$ we have two faces:
\[[A] = \{ 02021, 02101, 02121\}, \hspace{2cm} [B]=\{02012, 01212, 01201\}.\]

\subsubsection{Edges}
Recall mating with the basilica is a way to form a rational map in $\M_2$ from a polynomial in $\M$ outside the $(1/3, 2/3)$ limb, and every map in $\M_2$ is a mating.
\begin{lemma}\label{lem:bub-ray-props}
    Suppose $g=f_{-1}\sqcup f_c$ is a center of a hyperbolic component $\cal{H}_{g}$ of period $p$ in $\M_2$ with bubble rays $\cal{B}^\pm_g$ landing at the root of $\cal{H}_g$
   Then,
    \begin{enumerate}
        \item [(a)] each $\cal{B}^\pm_g$ corresponds to a unique element $\nu^\pm \in C_{p,2}$;
        \item [(b)] $g$ is primitive if and only if $\nu^\pm$ correspond to different elements of $C_{p,2}$;
        \item [(c)] the projection $\pi_2: \MC_p(\mathcal{F}_2) \to \mathcal{F}_2$ is branched over the root of $\cal{H}_g$ if and only if $g$ is primitive.
    \end{enumerate}
\end{lemma}
\begin{proof}
    Item (a): Suppose $f_c$ is a $p$-cycle in $\M$.
    The landing point $x^\pm\in J_{f_{-1}}$ defines a $p$-cycle in $J_{f_{-1}}$, whose $C_{p,2}$ labelling is given by \Cref{lem:semi-conj}.

    Item (b): By \Cref{lem:semi-conj}, $g$ has two distinct labels in $C_{p,2}$ if and only if $f_c$ has two distinct labels in $C_p$.
    But then, by definition, $f_c$ has two distinct labels if and only if $f_c$ is primitive.

    Item (c): Follows from (b) plus \Cref{per1-edge-branching}.
\end{proof}

\begin{proposition}\label{prop:Per2-branching}
    If $p\geq 3$, then $\pi_2: \MC_p(\mathcal{F}_2)\to \mathcal{F}_2$ is branched precisely at the roots of primitive hyperbolic components.
    For $p=1$, the map $\pi_2$ is also branched at the lift of the puncture in $\Per_2(0)$.
    For $p=2$, the map $\pi_2$ is a homeomorphism.
\end{proposition}
\begin{proof}
	The proof for that $\pi_2$ branches precisely at the roots of primitive components, in light of \Cref{lem:bub-ray-props}, follows exactly as it does in the $\Per_1(0)$ setting.
    The existence of punctures is novel in $\Per_2(0)$.
     For $p=1$, we can observe directly $\pi_2$ branches. For $p>1$, there are finitely many $p$-cycles, each a bounded distance away from the puncture.
    From this it follows from \Cref{lem:per2-dynam-ray-holomorphic-landing} that taking a sufficiently small loop around the puncture to avoid the active $p$ cycles, that $\gamma_\theta(c)$ varies continuously.
\end{proof}

\subsection{Algorithm for $\Sigma_{p, 2}$}

Finally, we see the payoff of our careful setup above to mirror the construction over $\Per_1(0)$, as we are in the position to immediately generalize the Telephone \Cref{algo-per1-angle} to a ``bubble ray" algorithm for constructing a cell complex $\Sigma_{p, 2}$ for $\MC_p(\mathcal{F}_2)$.
First we present the vertex relations.

\begin{lemma}
    Let $v_1, v_2$ be two vertices in $\Sigma_{p, 2}$.
    Then $v_1$ and $v_2$ are joined by an edge if and only if there exists a primitive hyperbolic component in $\M_2$ whose bubble rays produce sequences in $C_{p,2}$ agreeing with $v_1, v_2$.
\end{lemma}
\begin{proof}
    This follows from the proof of \Cref{P:vertex-relation}, using veins $V_{2}(c)$ in $\M_2$, which are embedded by \Cref{cor:veins-embed}.
\end{proof}

Then, before we present the algorithm, as $\M_2$ is not simply connected, the notion of ``order" of going around $\M_2$ is not immediately sensible.
However, as \Cref{thm:per2-structure} gives an isomorphism  $\varphi: \M^*\sqcup K_B^*\to \Per_2(0)$, we understand the counterclockwise turning around $\M_2$ as the image of that order around $\M^*$ under $\varphi$.
Formally, this shows

\begin{lemma}\label{lem:ccw}
	There is a natural counterclockwise order around $\M_2$.
\end{lemma}

Now we are ready to state our main algorithm to produce the cell decomposition for the marked cycle curves over $\Per_2(0)$.

\begin{algorithm}[Telephone algorithm for $\Per_2(0)$]\label{algo-per2}
Let $p \geq 2$ be an integer. Then the cell decomposition for $\MC_p(\mathcal{F}_2)$ is obtained as follows:
    \begin{enumerate}
    	\item Write down the list of all primitive hyperbolic components of period $p$ in $\M_2$. The root of each such component is the landing point 
	of two bubble rays. 
	\item Associate to each angle of such a bubble ray a label in $C_{p, 2}$, by taking its itinerary under the partition $\mathcal{J} = \{J^{(k)}, k = 0, 1, 2\}$ 
	of \Cref{lem:semi-conj}. 
	\item Take one element of $C_{p, 2}$ and set it as the \emph{active label}. Then, starting with the external parameter angle $0$ 
	and going around $\M_2$ in counterclockwise order, look for the first angle pair that contains the active label. 
	\item When you encounter a ray pair whose labels contain the active label, 
	add a new edge to the current face with vertices the two labels of this ray pair, and edge labeled by the current hyperbolic component. 
	\item Then reset the active label to be the other label of the current ray pair, and keep going around $\M_2$. 
	\item The face is completed once you encounter the same edge you started with, with the same orientation. 
	\item Repeat the process starting from the other elements of $C_{p,2}$, until all available edges with both orientations have been listed. 
	\item Glue together edges with the same label, with opposite orientation.
    \end{enumerate}
\end{algorithm}

As an example, let us produce the tessellation for $\MC_5(\mathcal{F}_2)$. Let us consider the period $5$ abstract ternary cycles: 
\begin{align*}
A & \quad 02021 \qquad & B \quad 02012 \\
\bar{A} & \quad 10102 \qquad & \bar{B} \quad 10120 \\
\bar{\bar{A}} & \quad 21210 \qquad & \bar{\bar{B}} \quad 21201 
\end{align*}
Now, let us compute the ternary coding of the primitive hyperbolic components of period $5$. 
For instance, for the angle $3/31$ you get the ternary itinerary
$$3/31: 2 \qquad 6/31: 0 \qquad 12/31 : 1 \qquad 24/31: 2 \qquad 17/31 : 1$$
so its label is $20121 \sim 21201 = \bar{\bar{B}}$.
Computing all labels for all ray pairs landing at roots of hyperbolic components we have the following. 
\begin{align*}
3/31: \bar{\bar{B}} & \qquad 4/31: B \\
5/31: \bar{B} & \qquad 6/31: \bar{\bar{B}} \\
7/31: \bar{A} & \qquad 8/31: B \\
23/31: A & \qquad 24/31: \bar{\bar{B}} \\
25/31: \bar{A} & \qquad
26/31: \bar{\bar{A}} \\
27/31: A & \qquad
28/31: \bar{A}
\end{align*}
Applying the algorithm results in the tessellation from \Cref{fig:M52-ex}.  

\begin{theorem}
    \Cref{algo-per2}
    constructs a surface homeomorphic to $\MC_p(\mathcal{F}_2)$.
\end{theorem}

\begin{proof}
    The proof of \Cref{P:algo1works} generalizes immediately, swapping $p$'s for $(p,2)$'s.
\end{proof}

\definecolor{mygreen}{rgb}{0.0, 0.71, 0.0}
\definecolor{mymagenta}{rgb}{0.8, 0.0, 0.8}
\definecolor{myred}{rgb}{1.0, 0.0, 0.0}
\definecolor{myorange}{rgb}{0.8, 0.396, 0.0}
\definecolor{myskyblue}{rgb}{0.0, 0.556, 0.9}

\begin{figure}
    \centering
    \begin{subfigure}{0.8\textwidth}
        \centering
        \[
\begin{tikzcd}
                & {\bar{\bar{B}}} &&&& {\bar{\bar{B}}} \\
                {B} & \bar{B} & {A} &&& {[A]} & {} \\
                & {[B]} &&& {A} & \bar{\bar{A}} & {B} & {} \\
                & {\bar{A}} &&&& {\bar{A}} \\
                &&&& {B} \\
                {[B]} & \bar{B} & {\bar{\bar{B}}} & {[A]} & \bar{\bar{A}} & \bar{A} \\
                &&&& {A}
                \arrow["27"', color={rgb,255:red,214;green,92;blue,214}, no head, from=3-5, to=4-6]
                \arrow["3"', color={rgb,255:red,202;green,132;blue,104}, no head, from=3-7, to=1-6]
                \arrow["23"', color={rgb,255:red,92;green,214;blue,92}, no head, from=1-6, to=3-5]
                \arrow["23", color={rgb,255:red,92;green,214;blue,92}, no head, from=1-2, to=2-3]
                \arrow["7"', color={rgb,255:red,92;green,92;blue,214}, no head, from=2-1, to=4-2]
                \arrow["3"', color={rgb,255:red,214;green,153;blue,92}, no head, from=1-2, to=2-1]
                \arrow["5", color={rgb,255:red,80;green,180;blue,226}, shorten <=2pt, Rightarrow, no head, from=1-2, to=2-2]
                \arrow["27", color={rgb,255:red,214;green,92;blue,214}, no head, from=2-3, to=4-2]
                \arrow["7"', color={rgb,255:red,92;green,92;blue,214}, no head, from=4-6, to=3-7]
                \arrow["25", color={rgb,255:red,214;green,92;blue,92}, shorten >=2pt, Rightarrow, no head, from=4-6, to=3-6]
                \arrow["5"', color={rgb,255:red,80;green,180;blue,226}, shorten >=3pt, Rightarrow, no head, from=6-2, to=6-3]
                \arrow["23"', color={rgb,255:red,92;green,214;blue,92}, from=6-3, to=7-5]
                \arrow["3", color={rgb,255:red,214;green,153;blue,92}, no head, from=6-3, to=5-5]
                \arrow["7", color={rgb,255:red,92;green,92;blue,214}, no head, from=5-5, to=6-6]
                \arrow["25", color={rgb,255:red,214;green,92;blue,92}, shorten >=3pt, Rightarrow, no head, from=6-6, to=6-5]
                \arrow["27", color={rgb,255:red,214;green,92;blue,214}, no head, from=6-6, to=7-5]
            \end{tikzcd}
        \]
        \caption{Cell structure for the marked cycle curve $\MC_5(\mathcal{F}_2)$, which has genus $g=0$. Top: the faces $[A]$ and $[B]$ as generated by the Bubble Ray Algorithm; the colored edges correspond to the odd characteristic angle for the map $g$ within the mating $f_{-1}\sqcup g$ where a cycle branches; compare to Example 3.8. Bottom: gluing of the faces  $[A]$ and $[B]$.}
        \label{fig:M52-ex}
    \end{subfigure}\\
    \begin{subfigure}{0.8\textwidth}
        \centering
        \includegraphics[width=\textwidth]{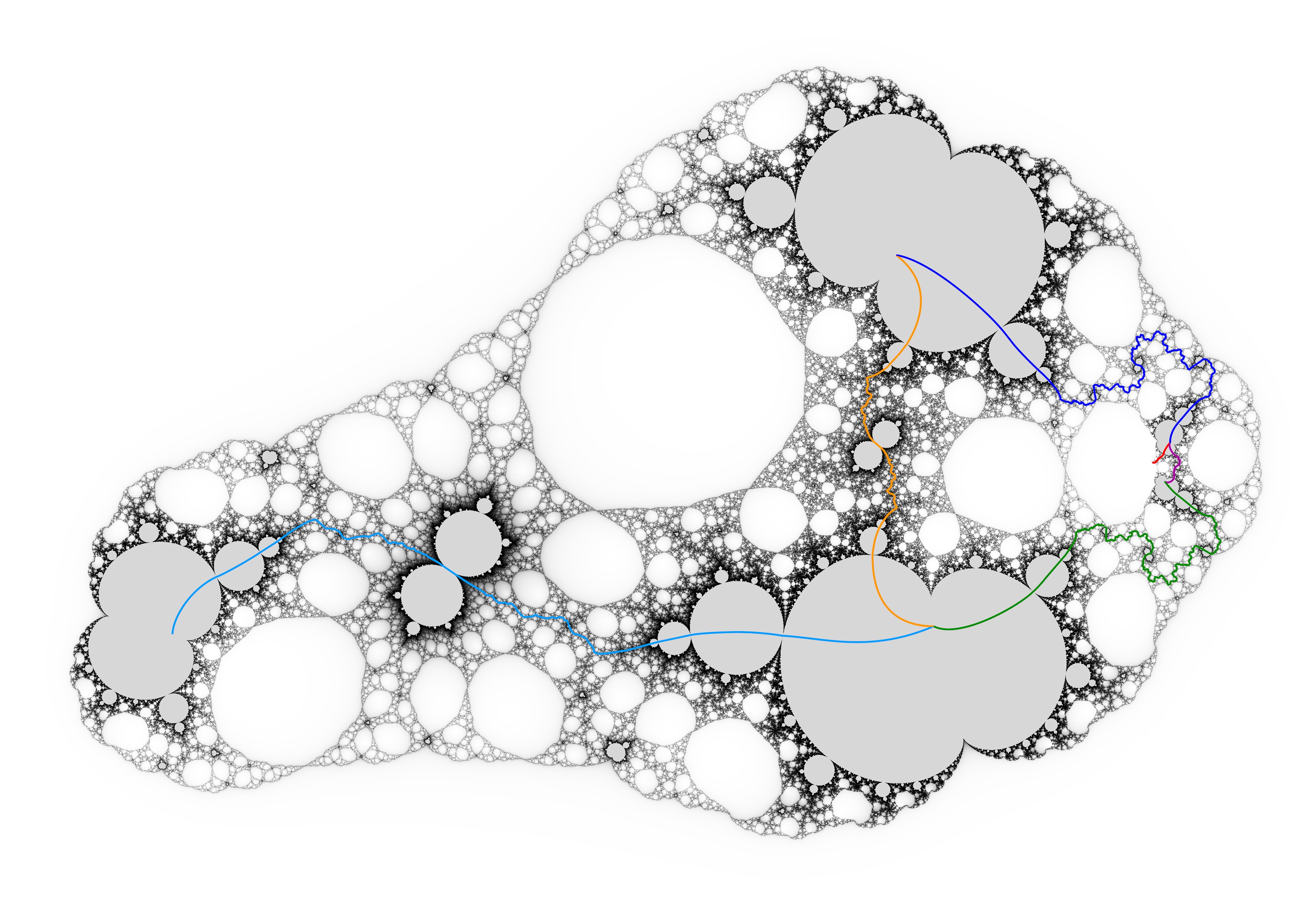}
        \caption{Cells visualized in $\MC_5(\mathcal{F}_2)$.}
    \end{subfigure}
    \caption{Comparison of combinatorial structure of $\MC_5(\mathcal{F}_2)$ to geometric structure.}
\end{figure}

\bigskip
We also obtain which faces have degree $3$, and which have degree $1$, as follows. 
Given a face, we write down the labels of the edges on its boundary, and we read them in clockwise order (note that the order \emph{is} important, 
as the canonical order is counterclockwise around $\M_2$, but it becomes clockwise as seen from the center of a face, which is a lift of $\infty$).

Given a vertex $v$ on a face $F$, let $\theta_1, \theta_2$ be the characteristic angles of the edges of $F$ incident to $v$, so that $\theta_1$ is 
the angle of the edge immediately before $v$ and $\theta_2$ is the angle of the edge immediately after $v$ in the orientation of $F$. 
We say that such vertex $v$ is \emph{median} if the positively oriented arc on the circle from $\theta_1$ to $\theta_2$ contains the angle $1/2$. 

The number of median vertices gives the local degree of the face, which can be either $1$ or $3$. 

\bigskip
Now, if \Cref{C:continuation} is true, we can also obtain from the previous algorithm the label in $[C_{p, 2}]$ of each face, using the following observation. 

\begin{lemma} \label{L:consistent}
If \Cref{C:continuation} is true, then for each face its label as defined in \Cref{lem:per2-faces} 
and \Cref{E:face} coincides with the abstract cycle trio associated to any of its median vertices.  
\end{lemma}

\begin{proof}
Let $v = (0, \zeta)$ be a median vertex on a face $F$, seen as a point on $\MC_p(\mathcal{F}_2)$. 
By definition of median vertex  there is a lift $\widetilde{\gamma}$ of the negative real axis $\gamma = (-\infty, 0]$ in $\Per_2(0)$ that connects $v$ to a lift of $\infty$. 

As described in \Cref{algo-per2} (2), the label of $v$ is the symbolic itinerary of the cycle $\zeta$ with respect to the partition $\mathcal{J}$ for $J_{g_0}$. Let us denote this label as $\omega \in C_{p, 2}$. By analytically following $\zeta$ as the parameter moves along $\gamma$, we then 
obtain a cycle $\zeta'$ for $g_a$, where we can take, say, $a = -100$. By  \Cref{C:continuation}, the coding of the cycle $\zeta'$ for $g_a$ under the partition $\mathcal{I}$ is also equal to $\omega$. 
By definition of $\Phi_\mathcal{E}$, this is the label associated to the face $F$ in \Cref{lem:per2-faces}.
\end{proof}

Thus, the algorithm to obtain the labels of a face is as follows: consider any one of its median vertices, and label the face 
with the cycle trio associated to the label of the vertex.
Note that as a consequence of \Cref{L:consistent}, all median vertices in the same face have labels that correspond to the same abstract cycle trio. 

\bigskip
For example in the left face in \Cref{fig:M52-ex} the labels in clockwise order read as 
$$(5/31, 23/31, 27/31, 7/31, 3/31, 5/31).$$
The threshold of $1/2$ is crossed going from $5/31$ to $23/31$, so at the vertex which has label $\bar{\bar{B}}$, hence we label the corresponding face as $[B]$. 
Similarly, the threshold is also crossed going from $7/31$ to $3/31$, as the angle is increasing, so it goes from $7/31$ to $1$, crossing $1/2$, and then from $0$ to $3/31$; the corresponding vertex has label $B$, which also gives rise to the class $[B]$ for the face. 
Finally, it is crossed a third time going from $5/31$ to $5/31$ (at a vertex with label $\bar{B}$). 
Since there are $3$ median vertices, the local degree of this face equals $3$. 

\bigskip
\Cref{L:consistent} shows how to obtain face relations for $\Sigma_{p, 2}$. Note that in $\Per_1(0)$ one could also obtain face relations via kneading data, as in \Cref{algo-per1-knead}. 

Analogously, in $\Per_2(0)$ we have defined ``kneading data" for parameters in either the external component $\mathcal{E}$ or the basilica component:  in the external component $\mathcal{E}$, one considers the symbolic itinerary with respect to the partition $\mathcal{I}_a$, and in the basilica component, one takes the symbolic itinerary with respect to the partition $\mathcal{J}$.

\begin{question}   \label{question:KD-for-p2} \label{question:labeling-hc}
Can we analogously define ``kneading data" for all hyperbolic components in $\Per_2(0)$?  
More precisely, let $\Omega \subseteq \Per_2(0)$ be the union of all hyperbolic components. For any hyperbolic component, 
is there a canonical partition of the Julia set into three parts such that  by taking the corresponding itinerary 
we obtain a map 
$$\Phi : \MC_p(\Omega) \to [C_{p, 2}]$$
that is constant on faces? This seems particularly challenging for the capture components. 
\end{question}

Sometimes a face in $\Sigma_{p, 2}$ will be degenerate, in the sense there may be some vertex $v\in C_{p,2}$ witnessed only by a single primitive component's bubble ray. See \Cref{fig:M52-ex} for an example in $p=5$.
In $\Per_1(0)$ we do not see the phenomenon because real primitive components in $\M$ provide edge connections between vertex pairs $v$ and $\overline{v}$. However, real polynomials do not correspond to any map in $\Per_2(0)$ as they form obstructed matings.

Note that $\MC_{3}(\mathcal{F}_2)$ is disconnected, as $C_{3}$ (and hence $C_{3,2}$) has two elements, but the unique primitive 3-cycle whose characteristic angles connects them is real (the airplane), so there is no corresponding edge joining them (see \Cref{fig:per2mc3}). 
However, this appears to be the only exception. If one were able to show there is a path between any two vertices in $C_{p, 2}$ corresponding to a concatenation of non-primitive polynomials outside the $1/2$-limb, then one would show the following conjecture:

\begin{conjecture}
    For $p\ne 3$, the marked cycle curve $\MC_p(\mathcal{F}_2)$ is connected.
\end{conjecture}

\begin{figure}
    \centering
    \includegraphics[height=6cm]{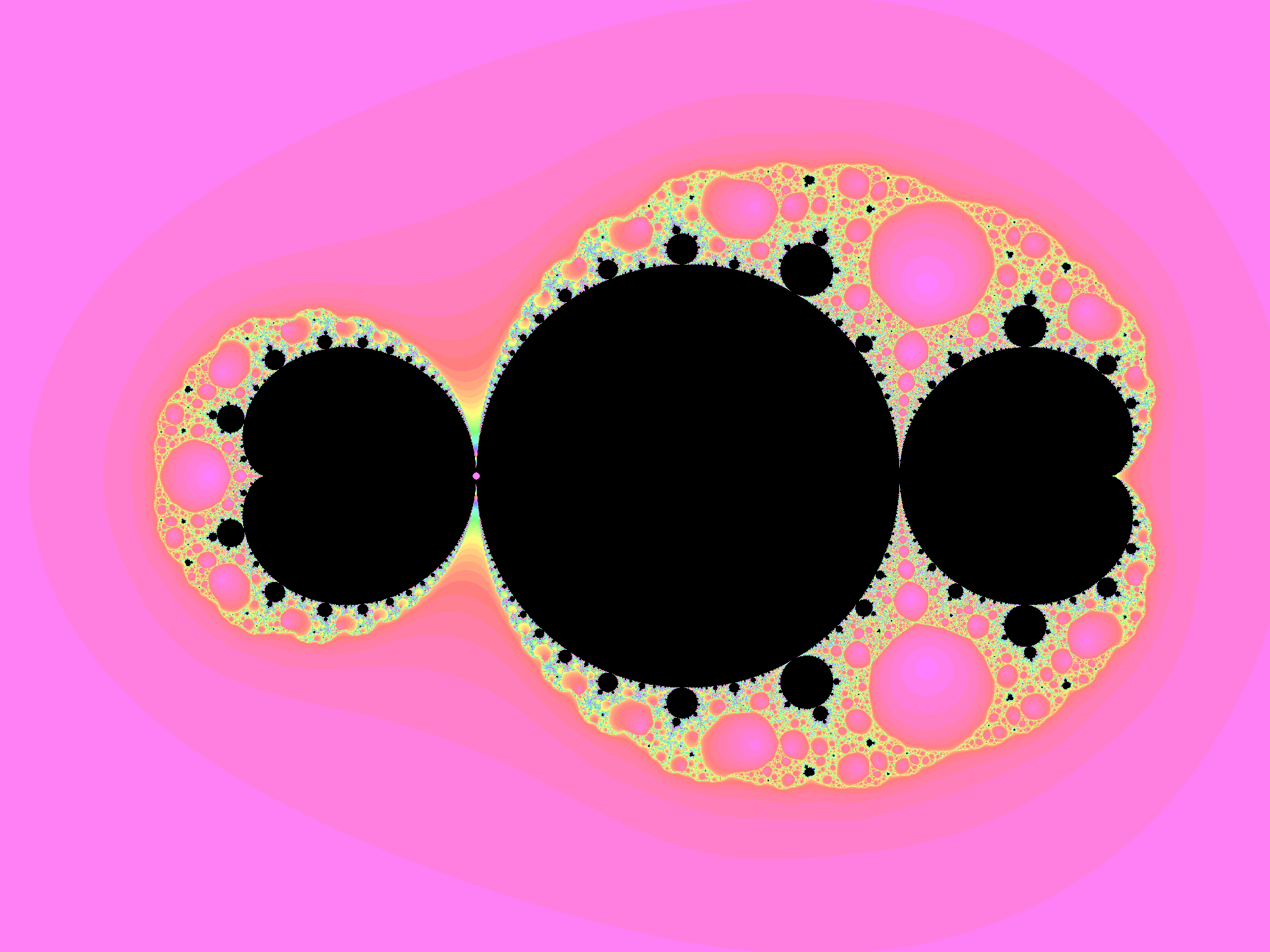}
    \includegraphics[height=6cm]{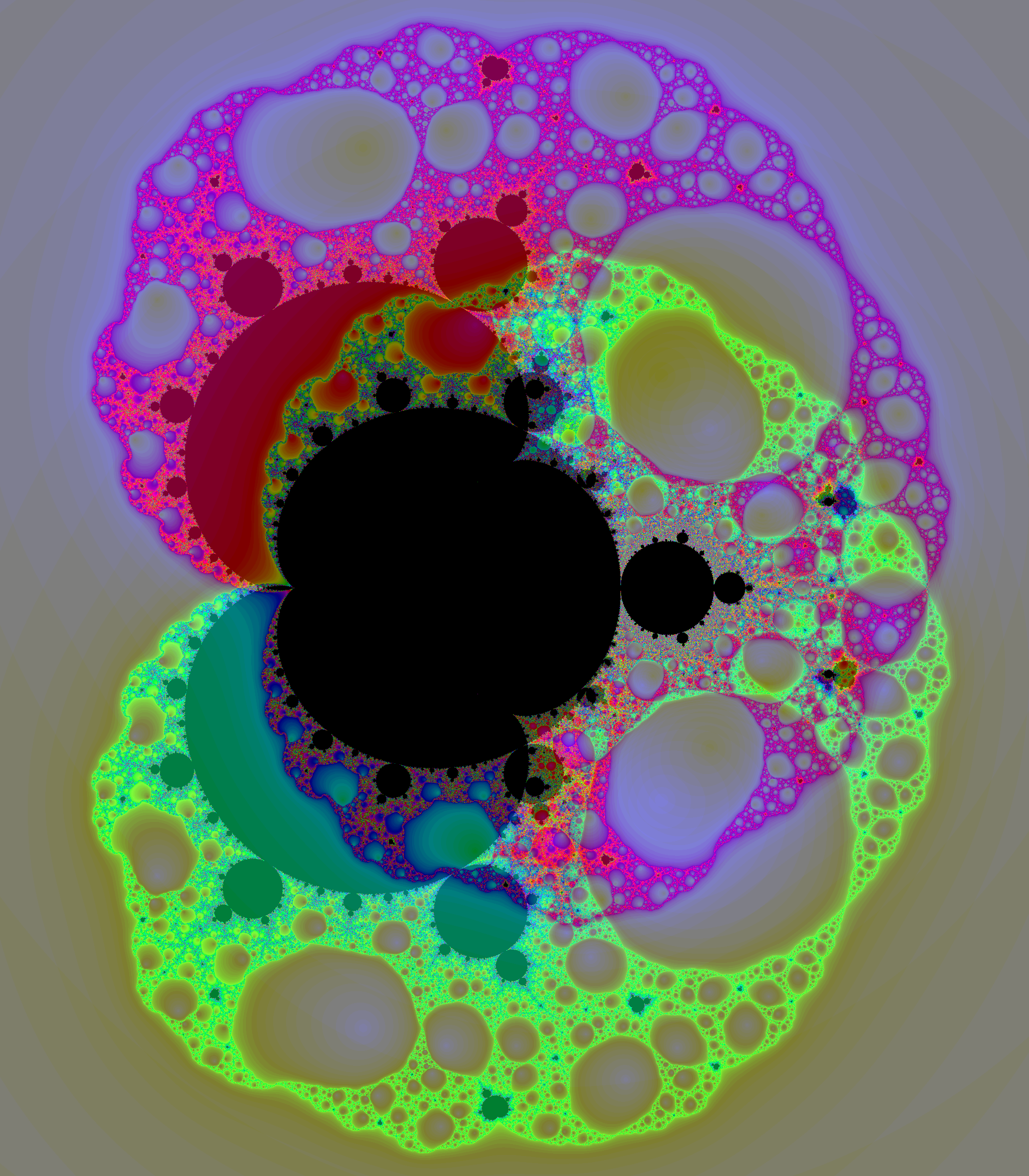}
    \caption{%
        Left: $\MC_{1}(\mathcal{F}_2)$: out of marking any periodic cycle, only marking fixed points yields monodromy around the singularity of $\Per_2(0)$.
        Right: $\MC_{3}(\mathcal{F}_2)$ is disconnected. Each of its sheets is isomorphic to $\Per_2(0)$.
    }\label{fig:per2mc3}
\end{figure}

\section{Combinatorics of Marked Cycle Curves}\label{sec:combo}
In this last section, we will address various combinatorial properties of the marked cycle curves.
In \Cref{sec:num_cells}, we derive formulas for the number of cells in our decompositions of $\MC_n(\Per_m(0))$ for $m=1,2$.
In \Cref{sec:props_cell_decomposition}, we present some properties and conjectures for the relations between cells.

\subsection{Number of cells}\label{sec:num_cells}
We begin by studying the dynamics of two maps:
\begin{itemize}
    \item $ h_1(z) = z^2$, which is the map at the origin in $\Per_1(0)$, and
    \item $h_2(z) = z^{-2}$, which is conjugate to the map at infinity in $\Per_2(0)$.
\end{itemize}

\begin{lemma}\label{lem:cycle_div_counting}
    The number of points on $\mathbb{S}^1$ of period dividing $n$ under $h_1(z) = z^2$ is
    \begin{align*}
        \tilde p_1(n) & = 2^n - 1.
        \intertext{
            The number of points on $\mathbb{S}^1$ of period dividing $n$ under $h_2(z) = z^{-2}$ is
        }
        \tilde p_2(n) & = 2^n - {(-1)}^n.
    \end{align*}
\end{lemma}

\begin{proof}
    A point $z=\exp(2\pi i\theta)$ is periodic of period dividing $n$ under $h_1$
    if and only if $2^n \theta = \theta \pmod 1$, i.e.\ $\theta$ is of the form $\frac{j}{2^n-1}$.

    Similarly, a point $z=\exp(2\pi i\theta)$ is periodic of period dividing $n$ under $h_2$
    if and only if
    ${(-2)}^n \theta = \theta \pmod 1$, i.e.\ $\theta$ is of the form
    \[\frac{j}{\del{-2}^n-1} = \frac{\del{-1}^n j}{2^n-\del{-1}^n}.\]
\end{proof}
\noindent Applying the M\"obius inversion formula, we obtain the following:
\begin{corollary}
    For $m=1,2$, the number of points on $\mathbb{S}^1$ of period $n$ under $h_m(z)$ is
    \begin{align*}
        p_m(n) & = (\mu * \tilde{p}_m)(n) = \sum_{d\mid n} \mu\del{\frac{n}d}\tilde{p}_m(d),
    \end{align*}
    where $\mu$ is the M\"obius function and $*$ denotes Dirichlet convolution.
\end{corollary}

\begin{corollary}
    The number of $n$-cycles under $h_m(z) = z^{\pm2}$ is
    \begin{align*}
        \cyc_m(n) & = \frac1n p_m(n) = \frac1n \sum_{d\mid n} \mu\del{\frac{n}d}\tilde{p}_m(d).
    \end{align*}
\end{corollary}

\begin{remark}\label{rem:cyc1cyc2}
    The values $\cyc_1(n)$ and $\cyc_2(n)$ are equal for $n\ge 3$.
    This is most easily seen since the M\"obius transform is additive, and the
    M\"obius transform of $\iota(n):=\tilde{p}_2(n)-\tilde{p}_1(n) = 1-\del{-1}^n$ is
    \[\del{\mu*\iota}(n) =
        \begin{cases}
            2  & n=1     \\
            -2 & n=2     \\
            0  & n\ge 3.
        \end{cases}
    \]
    Dynamically, the fact that these two quantities are eventually equal should
    come as no surprise.
    Indeed, for any degree $d$ branched self-cover $f$ of the sphere,
    there are only finitely many $n$ for which $f$ has a critical $n$-cycle.
    For all other $n$, the Lefschetz fixed point theorem implies that the
    number of $n$-cycles depends only on $d$.
\end{remark}

These cycle counts are closely related to the number of hyperbolic components in the corresponding parameter spaces.
Indeed, we have the following:
\begin{lemma}
 The number of hyperbolic components of period $n$ in the Mandelbrot set is
    \[
        \Hyp_1(n) = \frac12 p_1(n). 
    \]
    The number of hyperbolic components of period $n$ outside the $\intoo{\frac13,\frac23}$ limb, which is equal to the number of hyperbolic components of period $n$ in $\Per_2(0)$, is
    \[
        \Hyp_2(n) = \frac13 p_2(n).
    \]
\end{lemma}
\begin{proof}
This is straightforward for $n=1$. \\

\noindent    Every hyperbolic component in the Mandelbrot set of period $n\ge2$ is the landing point
    of two external rays of period $n$ under $f_1$.
    It follows that the number of hyperbolic
    components of period $n \geq 2$ is
    \[
        \Hyp_1(n) = \frac12 p_1(n).
    \]
    To count the number of hyperbolic components outside the $\intoo{\frac13,\frac23}$ limb,
    we must determine the number $N$ of multiples of $\frac1{2^n-1}$ (mod $1$) which lie outside
    $\intoo{\frac13,\frac23}$. It is easy to see that
    \begin{align*}
        N & = \begin{cases} \frac23\del{2^n+1} & n \text{ is odd} \\ \frac23\del{2^n-1} & n \text{ is even}\end{cases} \\
          & = \frac23 \tilde{p}_2(n).
    \end{align*}
    It follows that the number of hyperbolic components of period dividing $n$ outside the $\intoo{\frac13,\frac23}$ limb is
    \[
        \widetilde{\Hyp}_2(n) = \frac{N}{2} = \frac13 \tilde{p}_2(n).
    \]
    Applying the M\"obius inversion formula, the result follows.
\end{proof}

\begin{corollary}
Let $\vp$ denote Euler's totient function.
   For $n>2$, the number of \emph{primitive} hyperbolic components of period $n$ in the Mandelbrot set is 
   \[
   \Prim_1(n) = (2\Hyp_1 - \vp * \Hyp_1)(n).
   \]
  The number of primitive hyperbolic components of period $n$ outside the $\intoo{\frac13,\frac23}$ limb is
    \[
        \Prim_2(n) = (2\Hyp_2 - \vp * \Hyp_2)(n).
    \]
    
 \end{corollary}
\begin{proof}
    We first count the number of satellite components of period $n$.
    For every integer $d \in [1,n)$ that divides $n$, we observe that every hyperbolic component of period $d$ has $\vp(n/d)$ satellite tunings of period $n$. So 
    the total number of satellite hyperbolic components of period $n$ is
    \[
        \Sat_1(n) = \sum_{d\mid n, d\ne n}\vp\del{\frac{n}d}\Hyp_1(d) = \vp * \Hyp_1(n) - \Hyp_1(n),
    \]
    and the number outside the $\intoo{\frac13,\frac23}$ limb is, for $n > 2$, 
    \[
        \Sat_2(n) = \sum_{d\mid n, d\ne n}\vp\del{\frac{n}d}\Hyp_2(d) = \vp * \Hyp_2(n) - \Hyp_2(n).
    \]
    Subtracting $\Sat_m(n)$ from $\Hyp_m(n)$ yields the result.
\end{proof}

\begin{lemma}\label{lem:conj-order-divides-cycle-order}
    If an $n$-cycle $z_0\ra z_1\ra\cdots\ra z_n=z_0$ under a map $h$ is invariant under another map $\sigma$, where $\sigma$ commutes with $h$ and satisfies $\sigma^d = \mathrm{id}$, then either $\sigma$ acts as the identity on $\set{z_0,z_1,\cdots,z_{n-1}}$, or $\gcd(d, n) > 1$.
\end{lemma}
\begin{proof}
    Since our cycle is invariant under $\sigma$, we have $\sigma(z_0) = z_k$ for some $k\in\set{0,1,\dots,n-1}$.
    Since $\sigma$ commutes with $h$, we have $\sigma(z_i) = z_{i+k}$ for all $i$, where the sum $i+k$ is computed mod $n$.
    Since $\sigma^d(z_0)=z_0$, it follows that $kd\equiv 0 \mod n$.
    If we assume $d$ is co-prime to $n$, then $k \equiv 0 \mod n$ .
    Since $k<n$, we have $k=0$, so that $\sigma$ fixes all points in the cycle.
\end{proof}
Given a period $n\geq 3$, let $Q_1(n)$ be the number of $n-$cycles of $h_1$ invariant under $\sigma_1(z) = z^{-1}$, and let $Q_2(n)$ be the number of $n-$cycles of $h_2$ invariant under $\sigma_2(z)  = \omega z$, where $\omega$ is a primitive cube root of unity. 
\begin{remark}\label{rem:period-1-and-2}
    The cases $n=1,2$ warrant special consideration.

    Note that $\sigma_1$ has two fixed points, namely $\pm1$.
    Among these, $1$ is fixed by $h_1$, and $-1$ is not periodic under $h_1$.
    In light of \Cref{lem:conj-order-divides-cycle-order},
    all cycles of period $n>1$ invariant under $\sigma_1$ have even period.
    The fixed point $1$ is thus the only $\sigma_1$-invariant $1$-cycle of $h_1$.

    In spite of this, the face of $\MC_1(\mathcal{F}_1)$ corresponding to this fixed point still maps to $\mathcal{F}_1 = \Per_1(0)$ with local degree 2.
    This is because the two fixed points of $h_1$ have the same kneading sequence, so even though they are each invariant under $\sigma_1$, they ultimately lie in the same cycle class.

    From another perspective, this exception arises because for $p>1$, the canonical path $c(t) = t$ is a path in the parameter space of the $f_c$  from $c=0$ to $\infty$ that contains no branch points of the covering  $\MC_p(\mathcal{F}_1) \longrightarrow \mathcal{F}_1$ except $\infty$.    However, when following the lift of this path in $\MC_1(\mathcal{F}_1)$, the fixed points collide at $c=\frac14$ and our argument breaks down.

    For this reason, we make an exception to the definition of $Q_1$ by setting
    \[Q_1(1) = 0.\]
Also note that $h_1$ has exactly one $2$-cycle:   $\omega \leftrightarrow \omega^2$, and this cycle is invariant under $\sigma_1$. Consistent to the definition of $Q_1$, we set \[Q_1(2) = 1.\] 
In the case of $\Per_2(0)$, the map $h_2$ has three fixed points, namely $1$, $\omega$ and $\ol{\omega}$, none of which are invariant under $\sigma_2$.
    In the coordinates $g_a(z) =  \frac{z^2+a}{1-z^2}$, two of these fixed points collide at the puncture $a=-1$, which, like in $\Per_1(0)$, lies on the canonical path from the basilica ``face center" $a=0$ to the ramification point $a=\infty$.
    Unlike in the case of $\Per_1(0)$, though, our analysis is unaffected, since the model map $h_2$ describes $g_a(z)$ near the ramification point $a=\infty$, not near the face center $a=0$.
Thus we set \[Q_2(1) = 0.\]
    Also in $\Per_2(0)$, the critical 2-cycle $0\leftrightarrow\infty$ of $h_2(z)=z^{-2}$ is invariant under $\sigma_2$.
    As $0$ and $\infty$ are the only fixed points of $\sigma_2$, \Cref{lem:conj-order-divides-cycle-order} implies that this 2-cycle is the unique $\sigma_2$-invariant cycle of period not divisible by $3$.
    Since $0\leftrightarrow\infty$ is the unique $2$-cycle of $g_a$ for all $g_a \in\Per_2(0)$, the marked cycle curve $\MC_2(\mathcal{F}_2)$ is canonically identified with $\mathcal{F}_1 = \Per_2(0)$ itself.
    Since the unique face (i.e.\ unbounded component of the complement of the bifurcation locus) maps to itself with degree $1$, it makes sense to define \[Q_2(2) = 1.\]

 \end{remark}

\begin{lemma}\label{lem:self-conj-face-count}
    Let $k$ be a positive integer.    \begin{align*}
Q_1(2k) & = \frac{1}{2k} \sum_{d\mid k,\;2\nmid \frac{k}d} \mu\del{\frac{k}d}2^d.\\
 Q_2(3k) & = \frac{2}{3k} \sum_{d\mid k,\;3\nmid \frac{k}d} \mu\del{\frac{k}d}\tilde p_2(d).
    \end{align*}
\end{lemma}

The following technical lemma will also be useful in proving \Cref{lem:self-conj-face-count}:
\begin{lemma}\label{lem:mobius-recursion}
    Let $\ell$ be prime, let $\gt$ be an arbitrary function on $\N^*$, and suppose that $\tilde{q}$ is a function on $\N^*$ satisfying the recurrence
    \[
        \tilde{q}(n) = \begin{cases}
            \tilde{q}(k) + \gt(k) & \textnormal{if $n=\ell k$ is divisible by $\ell$}, \\
            0                     & \textnormal{otherwise}.
        \end{cases}
    \]
    Then the M\"obius transform of $\tilde{q}$ is given by
    \[
        q(n) = \begin{cases}
            \sum_{d\mid k,\;\ell\nmid d} \mu\del{\frac{k}{d}} \gt(d) & \textnormal{ if $n=\ell k$ is divisible by $\ell$}, \\
            0                                                        & \textnormal{otherwise}.
        \end{cases}
    \]
\end{lemma}
\begin{proof}
    For brevity, let us write
    \[
        \chi(d) = \begin{cases} 0 & \text{if $d\equiv 0 \mod \ell$}, \\
              1 & \text{otherwise}.\end{cases}
    \]
    Note that since $\ell$ is prime, $\mu(\ell d) = -\chi(d)\mu(d)$ for all $d\in\Z$.
    Define also
    \begin{align*}
        \alpha(k) & = \sum_{d\mid k} \chi\del{d}\mu\del{d}\gt(k/d)                   \\
                  & = \sum_{d\mid k,\;\ell\nmid \frac{k}d} \mu\del{\frac{k}d}\gt(d),
        \intertext{and let}
        g(k)      & = \sum_{d\mid k} \mu(d)\gt(k/d)
    \end{align*}
    be the M\"obius transform of $\gt$.

    By linearity of the M\"obius transform, $q(\ell k) = q(k) + g(k)$ for all $k$,
    and $q(n)=0$ for $n$ coprime to $\ell$.

    We will argue by induction on $r=\ord_\ell(k)$ that $q(\ell k) = \alpha(k)$ for all $k$.
    In the base case $r=0$, for $k\not\equiv 0 \pmod{\ell}$, we have
    \begin{align*}
        q(\ell k)
         & = q(k) + \sum_{d\mid k}\mu\del{\frac {k}{d}} \gt(d)                 \\
         & = 0 + \sum_{d\mid k}\chi\del{\frac{k}d}\mu\del{\frac {k}{d}} \gt(d) \\
         & = \alpha(k),
    \end{align*}
    where the second equality follows since all divisors of $k$ are coprime to $\ell$.

    Suppose now that for some $r\ge 0$, $\alpha(\ell^{r} n) = q(\ell^r n)$ for all $n$ coprime to $\ell$.
    For $n\not\equiv 0\pmod{\ell}$, putting $k=\ell^{r}n$, we then have
    \begin{align*}
        \alpha(\ell k) - \alpha(k)
         & = \sum_{d\mid \ell k} \chi(d)\mu(d) \gt(\ell k/d)
        - \sum_{d\mid k} \chi(d)\mu(d) \gt(k/d)              \\
         & = \sum_{d\mid \ell k} \chi(d)\mu(d) \gt(\ell k/d)
        + \sum_{d\mid k} \mu(\ell d) \gt(\ell k/\ell d)      \\
         & = \sum_{d\mid \ell k}\mu(d) \gt(\ell k/d)         \\
         & = g(\ell k).
    \end{align*}
    Thus,
    \[
        \alpha(\ell k) = \alpha(k) + g(\ell k) = q(\ell k) + g(\ell k) = q(\ell^2 k),
    \]
    so by induction on $r$, $q(\ell k)=\alpha(k)$ for all $k$.
\end{proof}

\begin{proof}[Proof of \Cref{lem:self-conj-face-count}, $\Per_1(0)$ case]\mbox{}
    A point $z_0$ is periodic under $h_1(z)=z^2$ with period dividing $n=2k$
    if and only if $z_0 = \exp\del{2\pi i \frac{m}{M}}$ for some $m \in \Z/M\Z$, with  $M := 2^{2k}-1$.
    The cycle containing $z_0$ is invariant under $\sigma_1(z) = z^{-1}$
    if and only if either
    \begin{enumerate}
        \item[(a)] $h_1^{\circ k}(z_0) = z_0^{-1}$, or
        \item[(b)] $h_1^{\circ k}(z_0)=z_0$ and the $k$-cycle $z_0\mapsto\cdots\mapsto z_k=z_0$ is invariant under $\sigma_1$.
    \end{enumerate}
    Condition (a) is equivalent to
    \begin{align*}
        (2^k+1)m \equiv 0 \mod M.
    \end{align*}
    Since $2^k+1$ divides $M=(2^k-1)(2^k+1)$, the above equation has exactly $2^k+1$ solutions in $\Z/M\Z$.
    One of these solutions, however, is the fixed point $m=0$ (corresponding to $z_0=1$), which we discount in light of \Cref{rem:period-1-and-2}.
    Thus, the number of relevant solutions to (a) is $2^k$.

    Adding in the solutions for (b), we find that the number $\tilde{q}_1(n)$ of points of period dividing $n$ whose cycle is invariant under $\sigma_1$ is given by the recurrence
    \[
        \tilde{q}_1(n) = \begin{cases}
            \tilde{q}_1(k) + 2^k & \textrm{if }n=2k \textrm{ is even}, \\
            0                    & \textrm{if $n$ is odd}.
        \end{cases}
    \]
    The number $q_1(n)$ of such points whose period is exactly $n$ is then given by the M\"obius transform of $\tilde{q}_1$.
    By \Cref{lem:mobius-recursion}, this is given by
    \[
        q_1(n) = \begin{cases}
            \sum_{d\mid k,\;2\nmid d} \mu\del{\frac{k}{d}}2^d & \textnormal{if $n=2k$ is even,} \\
            0                                                 & \textnormal{if $n$ is odd}.
        \end{cases}
    \]

    The number $Q_1(n) = \frac
    {1}{n}q_1(n)$, thus yielding the desired result.
\end{proof}
\begin{proof}[Proof of \Cref{lem:self-conj-face-count}, $\Per_2(0)$ case]\mbox{}
    A point $z_0\in \mathbb{S}^1$ is periodic under $h_2(z)=z^{-2}$ of period dividing $n=3k$
    if and only if $z_0 = \exp\del{2\pi i \frac{m}{M}}$,
    where $M := 2^{3k}-\del{-1}^k$ and $m \in \Z/M\Z$.
    The cycle is invariant under $\sigma_2(z) = \omega z$
    if and only if one of the following holds:
    \begin{enumerate}
        \item[(a)] $h_2^{\circ k}(z_0) = \omega z_0$,
        \item[(b)] $h_2^{\circ k}(z_0) = \ol{\omega} z_0$, or
        \item[(c)] $h_2^{\circ k}(z_0) = z_0$ and the $k$-cycle $z_0\mapsto\cdots\mapsto z_k=z_0$ is invariant under $\sigma_2$.
    \end{enumerate}
    Condition (a) is equivalent to
    \begin{align}
        \del{\del{-2}^k - 1} m   & \equiv \frac{M}3 \mod M,
        \intertext{or equivalently}
        \del{2^k - \del{-1}^k} m & \equiv \del{-1}^k \frac M3 \mod M. \label{eq:per2_selfconj_cond}
    \end{align}
    Note that $\pt_2(k) = 2^k-\del{-1}^k$ divides
    \[
        \frac{M}{3} = \frac{2^{2k} + \del{-2}^{k} + 1}{3} \del{2^k-\del{-1}^k},
    \]
    where the first factor above is an integer as seen by reducing the numerator mod 3.
    It follows that equation \eqref{eq:per2_selfconj_cond} has exactly $\tilde{p}_2(k) = 2^k-\del{-1}^k$ solutions in $\Z/M\Z$.

    Analogously, condition (b) produces another $\tilde{p}_2(k)$ solutions.

    Adding in the solutions for (c), we find that the number $\tilde{q}_2(n)$ of points on $\mathbb{S}^1$ of period dividing $n$ whose cycle is invariant under $\sigma_2$ is given by the recurrence
    \[
        \tilde{q}_2(n) = \begin{cases}
            \tilde{q}_2(k) + 2\tilde{p}_2(k) & \textrm{if }n=3k \textrm{ is divisible by 3}, \\
            0                                & \textrm{otherwise}.
        \end{cases}
    \]
    The number $q_2(n)$ of such points whose period is exactly $n$ is then given by the M\"obius transform of $\tilde{q}_2$.
    By \Cref{lem:mobius-recursion}, this is given by
    \[
        q_2(n) = \begin{cases}
            2\sum_{d\mid k,\;2\nmid d} \mu\del{\frac{k}{d}}\tilde{p}_2(d) & \textnormal{if $n=3k$,} \\
            0                                                             & \textnormal{otherwise}.
        \end{cases}
    \]

    The number $Q_2(n) = \frac{1}{n}q_2(n)$, and thus we have the desired result.
    As per \Cref{rem:period-1-and-2}, we also have $Q_2(2)=1$ due to the exceptional 2-cycle $0\leftrightarrow \infty$ off of $\mathbb{S}^1$.
\end{proof}

\begin{theorem} \label{thm:genus}
    For $m=1,2$, the combinatorics of the cell decompositions for $\MC_p(\Per_m(0))$ produced by \Cref{algo-per1-angle} and \Cref{algo-per2} are as follows:
    \begin{itemize}
        \item The number of vertices is $\cyc_m(p)$.
        \item The number of edges is $\Prim_m(p)$.
        \item The number of faces is $\frac1{m+1}\del{\cyc_m(n) + mQ_m(p)}$.
    \end{itemize}
    Thus, the genus $g_m(p)$ of $\MC_n(\Per_m(0))$ is given by
    \begin{align*}
        g_1(p) & = 1 + \frac14 \del{2\Prim_1 - 3\cyc_1 - Q_1}(p),  \\
        g_2(p) & = 1 + \frac16\del{3\Prim_2 - 4\cyc_2 - 2Q_2}(p).
    \end{align*}
\end{theorem}
\begin{proof}
    The edge and vertex counts follow directly from the algorithms.

    To count the faces, recall that faces of $\MC_p(\Per_m(0))$ are in natural bijection with
    orbits under $\tau_{m+1}$ of $p$-cycles of $h_m$,
    where $\tau_{m+1}$ is element-wise incrementation by $1$ mod $(m+1)$  (recall the introduction of cycle duos and trios in \Cref{def:cycle_duo} and \Cref{def:cycle_trio} respectively).
    A given $p$-cycle $\zeta$ of $h_m$ has an orbit of size $m+1$ if it is not fixed by $\tau_{m+1}$,
    and $1$ otherwise. It follows that the total number of faces in $\MC_p(\Per_m(0))$ is
    \[
        \frac{1}{m+1}\del{\cyc_m(p) - Q_m(p)} + Q_m(p) = \frac1{m+1}\del{\cyc_m(p) + mQ_m(p)}
    \]
    as desired.
\end{proof}

\begin{corollary}
    In particular, if $p$ is an odd prime, then
    \begin{align*}
        g_1(p) & = \frac{1}{2p}\del{2^{p-1}(p-3) - p^2 + 2p + 3}, \\ 
        g_2(p) & = \frac{1}{6p}\del{2^{p}(p-4) - 3p^2 + 7p + 8}.  
    \end{align*}
\end{corollary}

\noindent
The following table summarizes the combinatorics of $\MC_p(\Per_1(0))$ and $\MC_p(\Per_2(0))$ for $2\le p\le15$:
\[
    \resizebox{\textwidth}{!}{
        \begin{tabular}{r|r|l|r|l|r|l|r|l|r|l}
            $p$ & $\cyc_1(p)$     & $\cyc_2(p)$
                & $\Prim_1(p)$    & $\Prim_2(p)$
                & $Q_1(p)$        & $Q_2(n)$
                & $\text{F}_1(p)$ & $\text{F}_2(p)$
                & $g_1(p)$        & $g_2(p)$
            \\\hline
            2   & 1               & 1               & 0     & 0     & 1 & 1 & 1    & 1   & 0    & 0    \\\hline
            3   & 2               & 2               & 1     & 0     & 0 & 2 & 1    & 2   & 0    & -1   \\\hline
            4   & 3               & 3               & 3     & 2     & 1 & 0 & 2    & 1   & 0    & 0    \\\hline
            5   & 6               & 6               & 11    & 6     & 0 & 0 & 3    & 2   & 2    & 0    \\\hline
            6   & 9               & 9               & 20    & 14    & 1 & 0 & 5    & 3   & 4    & 2    \\\hline
            7   & 18              & 18              & 57    & 36    & 0 & 0 & 9    & 6   & 16   & 7    \\\hline
            8   & 30              & 30              & 108   & 72    & 2 & 0 & 16   & 10  & 32   & 17   \\\hline
            9   & 56              & 56              & 240   & 158   & 0 & 2 & 28   & 20  & 79   & 42   \\\hline
            10  & 99              & 99              & 472   & 316   & 3 & 0 & 51   & 33  & 162  & 93   \\\hline
            11  & 186             & 186             & 1013  & 672   & 0 & 0 & 93   & 62  & 368  & 213  \\\hline
            12  & 335             & 335             & 1959  & 1306  & 5 & 2 & 170  & 113 & 728  & 430  \\\hline
            13  & 630             & 630             & 4083  & 2718  & 0 & 0 & 315  & 210 & 1570 & 940  \\\hline
            14  & 1161            & 1161            & 8052  & 5370  & 9 & 0 & 585  & 387 & 3154 & 1912 \\\hline
            15  & 2182            & 2182            & 16315 & 10874 & 0 & 4 & 1091 & 730 & 6522 & 3982 \\\hline
        \end{tabular}
    }
\]

\subsection{Properties of the cell decomposition}\label{sec:props_cell_decomposition}

Having constructed the cell decompositions, it becomes a natural question to ask about its geometric and combinatorial properties. 

\subsubsection{Invariance by conjugation}

\begin{proposition}
    The complex conjugation function $(c,\zeta) \mapsto (\overline{c},\overline{\zeta})$ leaves the faces of the cell decomposition for $\MC_p(\mathcal{F}_1)$ invariant.
\end{proposition}
\begin{proof}
    For any face $U$ (a lift of $\hat{\C } \setminus \mathcal{M}$) corresponding to a cycle duo $[\nu]$, there exists a point  $(c,\zeta) \in U$ with $c \in \R_{>0}$ and such that the points in the cycle $\zeta$ are the landing points of the orbit of $\theta \in \Q/\Z$, with $[\theta] = [\nu]$.
    Now, since the itinerary of the cycle $\overline{\zeta}$ in the dynamical plane of $f_{\overline{c}}$ equals
    the itinerary of $\zeta$ in the dynamical plane of  $f_c$, the map $\phi$ defined as in \Cref{lem:faces_are_cyclepairs} satisfies $\phi(c,\zeta) = \phi(\overline{c},\overline{\zeta}) = [\nu]$. Hence 
    $(\overline{c},\overline{\zeta}) \in U$.
\end{proof}

The same reasoning applied to parameters $a \in \mathcal{E} \subset \Per_2(0)$ yields the following result. 

\begin{proposition}
    The complex conjugation function $(a,\zeta) \mapsto (\overline{a},\overline{\zeta})$ leaves the faces of the cell decomposition for $\MC_p(\mathcal{F}_2)$ invariant.
\end{proposition}

We know that the Riemann surface $\MC_p(\mathcal{F}_1)$ is connected for any $p \geq 1$, but it appears that 
even the following stronger statement holds.

\begin{conjecture}[Skewer conjecture]
    The lift of the real axis under $\pi$ has connected closure in $\MC_p(\mathcal{F}_1)$ for all $p\ge3$.
\end{conjecture}
\begin{figure}[H]
    \centering
    \resizebox{\textwidth}{!}{
        \begin{tikzpicture}
            \def\edgelength{1.8cm}

            \def\radfive{1.0047588}
            \def\radseven{1.7623757}
            \def\radeight{1.7623757}
            \def\radtwelve{2.7243972}

            \def\fbaseangle{-180/8}
            \def\anchorx{0}

            \node (face1) at (\anchorx, 0) {$\abr{1}$};
            \node (node-1-0) at ($(face1)+(\fbaseangle:1.9075818)$) {$\del{1}$};
            \node (node-1-1) at ($(node-1-0)+(247.5 + \fbaseangle:1.46)$) {$\del{3}$};
            \node (node-1-2) at ($(node-1-1)+(202.5 + \fbaseangle:1.46)$) {$\del{5}$};
            \node (node-1-3) at ($(node-1-2)+(157.5 + \fbaseangle:1.46)$) {$\del{13}$};
            \node (node-1-4) at ($(node-1-3)+(112.5 + \fbaseangle:1.46)$) {$\del{11}$};
            \node (node-1-5) at ($(node-1-4)+(67.5 + \fbaseangle:1.46)$) {$\del{23}$};
            \node (node-1-6) at ($(node-1-5)+(22.5 + \fbaseangle:1.46)$) {$\del{15}$};
            \node (node-1-7) at ($(node-1-6)+(337.5 + \fbaseangle:1.46)$) {$\del{31}$};
            \draw (node-1-0) -- (node-1-1);
            \draw (node-1-1) -- (node-1-2);
            \draw (node-1-2) -- (node-1-3);
            \draw[double,double distance=2pt] (node-1-3) -- (node-1-4);
            \draw (node-1-4) -- (node-1-5);
            \draw (node-1-5) -- (node-1-6);
            \draw (node-1-6) -- (node-1-7);
            \draw[double,double distance=2pt] (node-1-7) -- (node-1-0);

            \draw[dashed] (node-1-0) -- (face1) -- (node-1-7);

            \def\fbaseangle{9*180/12}
            \def\anchorx{\radeight+\radtwelve}

            \node (face3) at (\anchorx, 0) {$\abr{3}$};
            \node (node-3-0) at ($(face3)+(\fbaseangle:2.8205035)$) {$\del{3}$};
            \node (node-3-1) at ($(node-3-0)+(255 + \fbaseangle:1.46)$) {$\del{1}$};
            \node (node-3-2) at ($(node-3-1)+(225 + \fbaseangle:1.46)$) {$\del{7}$};
            \node (node-3-3) at ($(node-3-2)+(195 + \fbaseangle:1.46)$) {$\del{13}$};
            \node (node-3-4) at ($(node-3-3)+(165 + \fbaseangle:1.46)$) {$\del{5}$};
            \node (node-3-5) at ($(node-3-4)+(135 + \fbaseangle:1.46)$) {$\del{23}$};
            \node (node-3-6) at ($(node-3-5)+(105 + \fbaseangle:1.46)$) {$\del{11}$};
            \node (node-3-7) at ($(node-3-6)+(75 + \fbaseangle:1.46)$) {$\del{7}$};
            \node (node-3-8) at ($(node-3-7)+(45 + \fbaseangle:1.46)$) {$\del{31}$};
            \node (node-3-9) at ($(node-3-8)+(15 + \fbaseangle:1.46)$) {$\del{15}$};
            \draw (node-3-0) -- (node-3-1);
            \draw (node-3-1) -- (node-3-2);
            \draw (node-3-2) -- (node-3-3);
            \draw (node-3-3) -- (node-3-4);
            \draw[double,double distance=2pt] (node-3-4) -- (node-3-5);
            \draw (node-3-5) -- (node-3-6);
            \draw (node-3-6) -- (node-3-7);
            \draw (node-3-7) -- (node-3-8);
            \draw (node-3-8) -- (node-3-9);
            \draw (node-3-9) -- (node-1-0);
            \draw (node-1-7) -- (node-3-0);

            \draw[dashed] (node-3-0) -- (face3) -- (node-3-9);

            \def\fbaseangle{5*180/8}
            \def\anchorx{-2*\radeight}

            \node (face5) at (\anchorx, 0) {$\abr{5}$};
            \node (node-5-0) at ($(face5)+(\fbaseangle:1.9075818)$) {$\del{5}$};
            \node (node-5-1) at ($(node-5-0)+(247.5 + \fbaseangle:1.46)$) {$\del{3}$};
            \node (node-5-4) at ($(node-1-3)+(112.5 + \fbaseangle:1.46)$) {$\del{15}$};
            \node (node-5-5) at ($(node-5-4)+(67.5 + \fbaseangle:1.46)$) {$\del{23}$};
            \node (node-5-6) at ($(node-5-5)+(22.5 + \fbaseangle:1.46)$) {$\del{3}$};
            \node (node-5-7) at ($(node-5-6)+(337.5 + \fbaseangle:1.46)$) {$\del{15}$};
            \draw (node-5-0) -- (node-5-1);
            \draw (node-5-1) -- (node-1-4);
            \draw (node-1-3) -- (node-5-4);
            \draw (node-5-4) -- (node-5-5);
            \draw (node-5-5) -- (node-5-6);
            \draw[double,double distance=2pt] (node-5-6) -- (node-5-7);
            \draw (node-5-7) -- (node-5-0);

            \draw[dashed] (node-5-0) -- (face5) -- (node-5-5);

            \def\fbaseangle{5*180/5}
            \def\anchorx{-3*\radeight-\radfive}

            \node (face7) at (\anchorx, 0) {$\abr{7}$};
            \node (node-7-0) at ($(face7)+(\fbaseangle:1.2419503)$) {$\del{7}$};
            \node (node-7-1) at ($(node-7-0)+(234 + \fbaseangle:1.46)$) {$\del{1}$};
            \node (node-7-4) at ($(node-5-6)+(18 + \fbaseangle:1.46)$) {$\del{31}$};
            \draw (node-7-0) -- (node-7-1);
            \draw (node-7-1) -- (node-5-7);
            \draw (node-5-6) -- (node-7-4);
            \draw (node-7-4) -- (node-7-0);

            \draw[dashed] (node-7-0) -- (face7);

            \def\fbaseangle{-360/7}
            \def\anchorx{1.5158606+\radeight+2*\radtwelve}

            \node (face11) at (\anchorx, 0) {$\abr{11}$};
            \node (node-11-0) at ($(face11)+(\fbaseangle:1.6824783)$) {$\del{11}$};
            \node (node-11-1) at ($(node-11-0)+(244.2857 + \fbaseangle:1.46)$) {$\del{3}$};
            \node (node-11-4) at ($(node-3-4)+(90 + \fbaseangle:1.46)$) {$\del{15}$};
            \node (node-11-5) at ($(node-11-4)+(38.57141 + \fbaseangle:1.46)$) {$\del{13}$};
            \node (node-11-6) at ($(node-11-5)+(347.14285 + \fbaseangle:1.46)$) {$\del{7}$};
            \draw (node-11-0) -- (node-11-1);
            \draw (node-11-1) -- (node-3-5);
            \draw (node-3-4) -- (node-11-4);
            \draw (node-11-4) -- (node-11-5);
            \draw (node-11-5) -- (node-11-6);
            \draw (node-11-6) -- (node-11-0);

            \draw[dashed] (node-11-0) -- (face11) -- (node-11-5);

            \draw[dotted, line width=0.8pt] (face7) -- (face5) -- (face1) -- (face3) -- (face11) -- (node-11-6);
        \end{tikzpicture}
    }
    \caption{Cell structure for the marked cycle curve $\MC_6(\mathcal{F}_1)$, which has genus $g=4$. Recall that eg  $\langle 7\rangle$ denotes the binary rotation class of 7, $000111$.
        Note that in this case, the faces $\abr{7}$ and $\abr{11}$ both have an odd number of edges. The face $\abr{7}$ is reflexive, while the face $\abr{11}$ is not. 
    }
    \label{fig: MC_6_Per_1}
\end{figure}
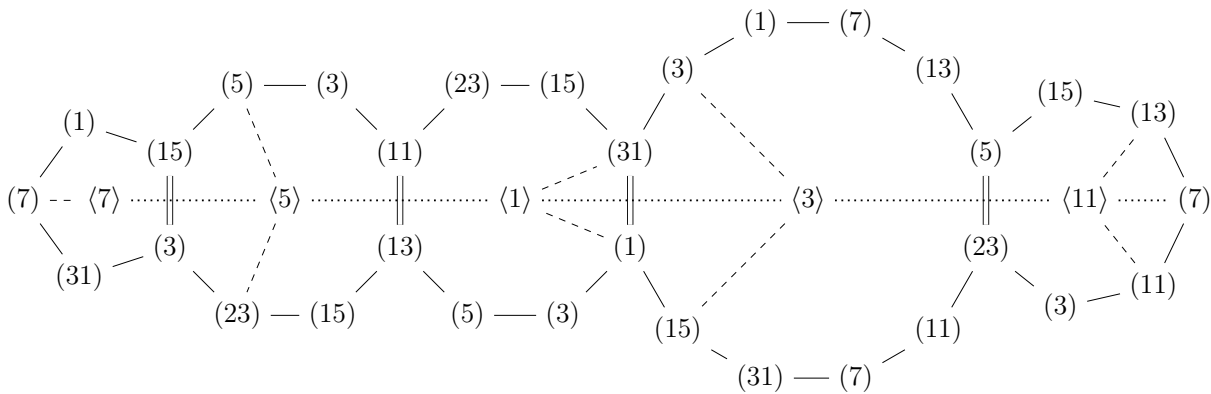

Consider for instance \Cref{fig: MC_6_Per_1}, where the closure of the lift of the real axis in  $\MC_p(\mathcal{F}_1)$ is indicated by the dotted and dashed lines. Clearly, this set is connected. 

\subsubsection{Numbers of sides and bigons}

We also discuss various combinatorial properties of the cells in the decomposition. 
For instance, one might ask what is the number of sides of the smallest face, or more in general what is the statistical distribution 
of the number of faces of a given size, where we measure size just by number of edges. 
Here we establish the much more modest fact that the smallest face is not a bigon. 

\begin{definition}[Bigon]
    A \emph{bigon} is a face that is bounded by two distinct edges --- in other words, it is a cell in the given decomposition of $\MC_p(\mathcal{F}_1)$ consisting of two vertices joined by two distinct edges and the face they bound.
\end{definition}

\begin{proposition}[The No Bigon Criterion] \label{P:no-bigons}
    There are no bigons in the cell decomposition of $\MC_p(\mathcal{F}_1)$ given by \Cref{algo-per1-angle}.
\end{proposition}

To prove this property, we resort to the following classical lemma of Douady \cite{Douady_exploringthe}.
\begin{lemma}[Douady's Lemma]
    Given an angle $\theta \in \mathbb{Q}/\mathbb{Z}$, there exists a unique integer $n\ge 0$ such that the parameter ray at angle $2^{n}\theta$ to the Mandelbrot set lands at a point in $\R$.
    This angle has the minimal distance from the angle $\frac{1}{2}$ out of all angles in the orbit of $\theta$ under angle-doubling.
\end{lemma}

We first show we can strengthen Douady's Lemma:

\begin{lemma}[Strong Douady's Lemma]
    Given parameter rays at angles $\theta,\theta' \in \mathbb{Q}/\mathbb{Z}$ that land at the root of a primitive hyperbolic component, there exists an angle in the orbit of one of $\theta,\theta'$ that lands at a primitive real hyperbolic component.
\end{lemma}
\begin{proof}
    If $\theta$ and $\theta'$ land on the real line we are done.
    So we assume to the contrary.

    By Douady's Lemma, there exist angles $\phi = 2^n\theta$ and $\phi' = 2^m\theta'$ that land at points in $\mathbb{R}$.
    Suppose $\phi$ lands on a satellite  component.
    Then there exists a minimal integer $k>0$ such that \[2^k\phi = 1-\phi. \]
    Since $1-\theta$ and $1-\phi$ are in the same cyclic orbit under angle-doubling, there exists a minimal integer $\ell>0$ such that
    \begin{align*}
        2^\ell \theta & = 1-\theta.
    \end{align*}
    This also implies that the period of $\theta$ under angle-doubling divides $2\ell$,  is of the form $2m$ for some integer $m>0$, and 
    \begin{align*}
    2^{m}\theta & = 1-\theta
    \end{align*}
    If we are working in the setting of odd period, this forms a contradiction, and we are done.
    Now assume otherwise.
    If $\phi'$ also lands on a satellite real component, then $1-\theta'$ is in the same cycle as $\theta'$.
    But since $\theta'$ and $\theta$ have the same period $2m$, it must be true that
    \begin{align*}
        2^m \theta' = 1-\theta'.
    \end{align*}
    From this, we get
    \begin{align*}
        |\theta- \theta'| & = |2^m \theta - 2^m \theta'| .
    \end{align*}
    However, $\theta,\theta'$ are the characteristic angles in an orbit portrait, and the following holds:
    \begin{align*}
        |\theta - \theta'| & < |2^i\theta - 2^i\theta'| & \forall i \text{ such that }\{2^i\theta, 2^i\theta'\} \ne \{\theta, \theta'\}.
    \end{align*}
    Thus we have a contradiction unless $2^m \theta = \theta'$ and $2^m \theta' = \theta$.
    Since $\theta$ and $\theta'$ land on a primitive component,  this is not possible.
\end{proof}

Note that we have also proven that $\theta$ has an angle landing on a primitive real component in its orbit under angle-doubling if and only if $1-\theta$ is not in this orbit.

\begin{proof}[Proof of \Cref{P:no-bigons}]
    Suppose there exists a bigon with vertices $[A]$ and $[B]$, and edges $e$ and $f$.
    Since conjugation preserves faces, exactly one of the following is true.
    \begin{enumerate}
        \item $\Bar{A} = B$ and $\Bar{e} = f$, or
        \item $\Bar{A} = A$, $\Bar{B} = B$,  $\Bar{e} = e$ and $\Bar{f} = f$.
    \end{enumerate}
    In case (1), the cycles $A$ and $B$ both contain angles $\theta, \phi$ respectively that land on real primitive components, and by a uniqueness argument, we have
    \begin{align*}
        1-\theta & = \phi
    \end{align*}
    (otherwise, $\theta$ and $1-\phi$ are distinct angles in the cycle $A$ that land on the real line, which contradicts Douady's lemma).

    But $e$ and $f$ correspond to the components on which $\theta$ and $\phi$ land --- and since $1-\theta = \phi$, we have $e=f$, which is a contradiction.
    \vspace{0.5em}

    In case (2), we note that $e$, $f$ are primitive real components.
    But $\Bar{A} = A$ implies that there exists no angle in $A$ that lands on a primitive real component, a contradiction.
\end{proof}

Motivated by numerical experiments, we put forward the following conjectures regarding the sizes of the smallest and largest faces in our cell decomposition of $\MC_p(\mathcal{F}_1)$.

\begin{conjecture}
    For $p\ge 5$, the largest face of $\MC_p(\mathcal{F}_1)$ has $\Phi(p) + 2$ edges, where
    \[
        \Phi(p) = \sum_{0\le k<p} \vp(k),
    \]
    and where $\vp$ denotes Euler's totient function. This face is unique unless $p=7$.
\end{conjecture}

\begin{conjecture}
    For $p\ge 9$ odd (and for $p=5$), the largest face of $\MC_p(\mathcal{F}_2)$ has edge count $\Phi(p)$.
    For $p\ge 6$ even, the largest face of $\MC_p(\mathcal{F}_2)$ has edge count
    \[
        \Phi(p) + 1 - \frac12 \vp\del{\frac{p}2}.
    \]
    In all cases except for $p= 8$ and $p=10$, there is a unique largest face of $\MC_p(\mathcal{F}_2)$ up to complex conjugation.
\end{conjecture}

\begin{conjecture}
    For $m=1,2$, the size $\kappa_m(p)$ of the smallest face in $\MC_p(\Per_m(0))$ tends to infinity as $p\ra\infty$. Moreover, if $\kappa^+_m(p)\ge \kappa_m(p)$ denotes the size of the smallest \emph{irreflexive} face, then the limit
    \[
        \lim_{p\ra\infty} \frac{\kappa_1^+(p)}{\kappa_2^+(p)}
    \]
    exists and is equal to 1.
\end{conjecture}

\begin{figure}[H]
    \centering
    \includegraphics[width=0.8\textwidth]{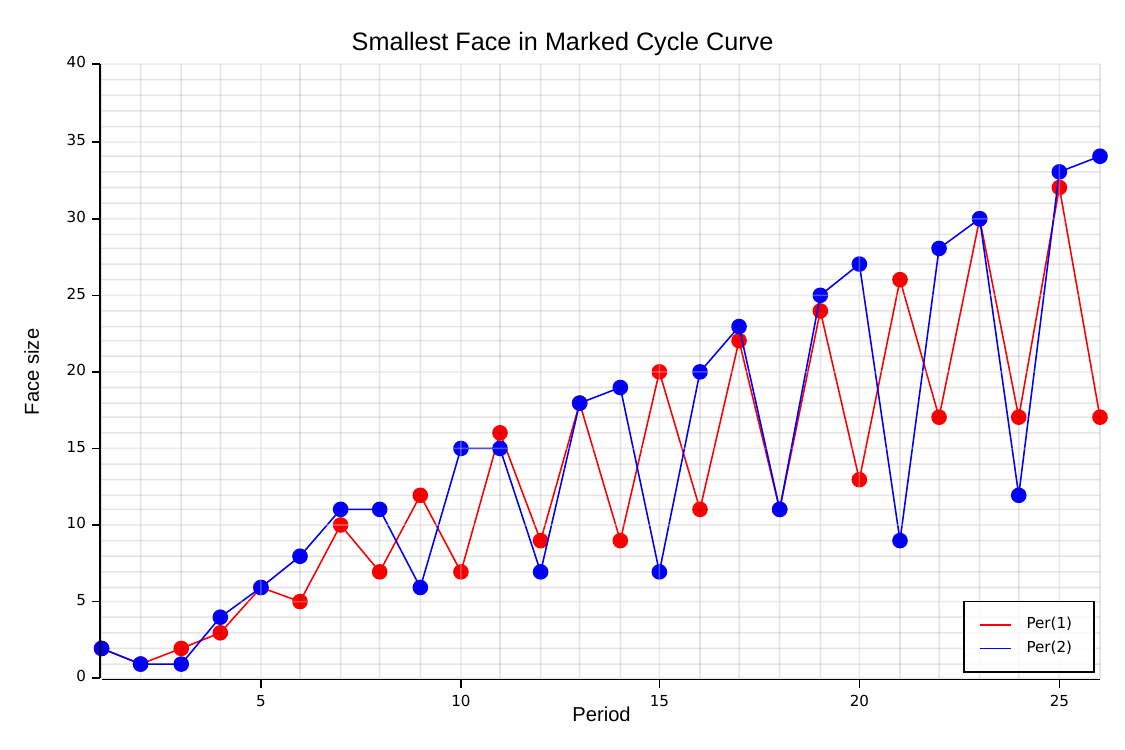}
    \caption{Number $\kappa_m(p)$ of (not necessarily distinct) edges bounding the smallest face in $\MC_p(\Per_m(0))$, $m=1,2$.}
\end{figure}

\begin{figure}[H]
    \centering
    \includegraphics[width=0.8\textwidth]{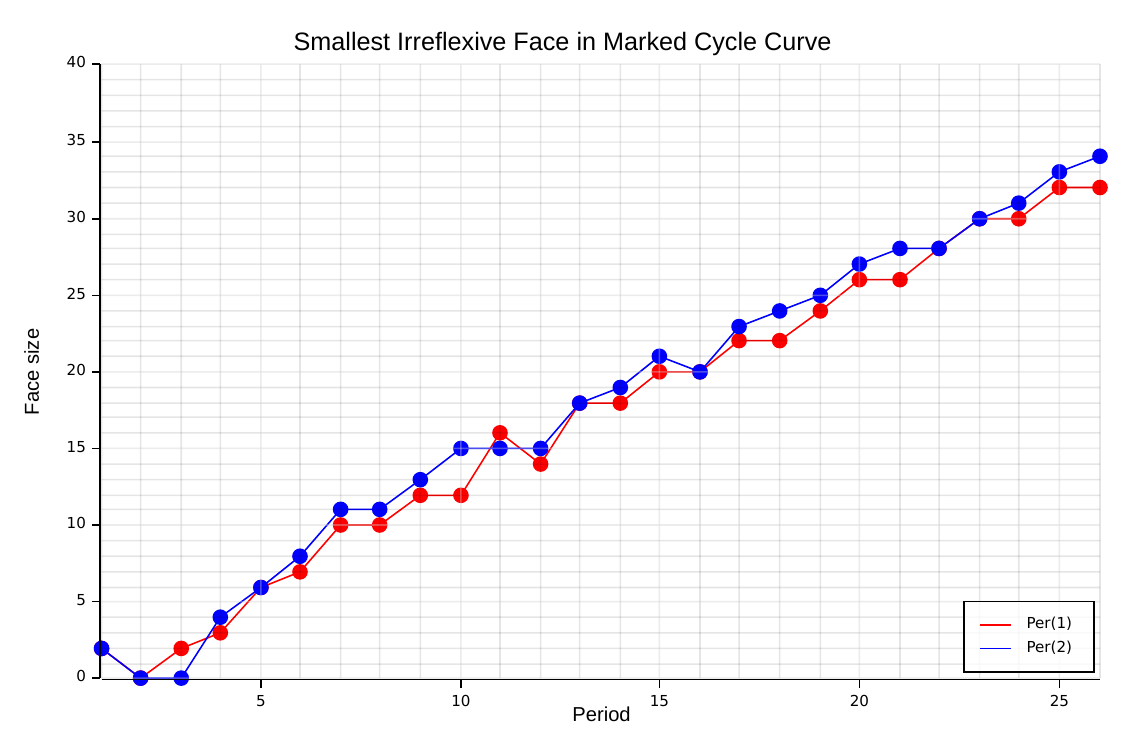}
    \caption{Number $\kappa_m^+(p)$ of (not necesssarily distinct) edges bounding the smallest irreflexive face in $\MC_p(\Per_m(0))$, $m=1,2$.}
\end{figure}

\bibliography{bibtemplate}
\bibliographystyle{amsalpha}
\end{document}